\documentclass[10pt,reqno]{amsart}

\usepackage{enumerate}
\usepackage{amsmath, amssymb, amsthm}
\usepackage{mathrsfs}
\usepackage{esint}
\usepackage{xcolor}
\usepackage{mathtools}
\usepackage{hyperref}
\usepackage{bm}

\newtheorem{thm}{Theorem}[section]
\newtheorem{corollary}[thm]{Corollary}
\newtheorem{lem}[thm]{Lemma}

\theoremstyle{definition}
\newtheorem{defn}[thm]{Definition}
\newtheorem{example}[thm]{Example}
\newtheorem{assumption}[thm]{Assumption}
\theoremstyle{remark}
\newtheorem{rem}[thm]{Remark}

\newcommand\bB{\mathbb{B}}
\newcommand\bC{\mathbb{C}}

\newcommand\bH{\mathbb{H}}

\newcommand\bN{\mathbb{N}}
\newcommand\bM{\mathbb{M}}

\newcommand\bR{\mathbb{R}}
\newcommand\bS{\mathbb{S}}

\newcommand\bZ{\mathbb{Z}}

\newcommand\cB{\mathcal{B}}

\newcommand\cF{\mathcal{F}}

\newcommand\cK{\mathcal{K}}

\newcommand\cP{\mathcal{P}}

\newcommand\cS{\mathcal{S}}
\newcommand\cM{\mathcal{M}}
\newcommand\cT{\mathcal{T}}
\newcommand\cO{\mathcal{O}}

\newcommand{\p}{\partial}
\DeclareMathOperator*{\esssup}{ess\,sup}

\newcommand{\aint}{-\hspace{-0.40cm}\int}

\newcommand{\mysection}[1]{\section{#1}
\setcounter{equation}{0}}

\begin{document}

\title[A weighted $L_p$-regularity theory for parabolic PDEs with P-D operators]
{A weighted $L_p$-regularity theory for parabolic partial differential equations with time-measurable pseudo-differential operators}

\author[J.-H. Choi]{Jae-Hwan Choi}
\address[J.-H. Choi]{Department of Mathematical Sciences, KAIST, 291 Daehak-ro, Yuseong-gu, Daejeon, 34141, Republic of Korea}
\email{jaehwanchoi@kaist.ac.kr}

\author[I. Kim]{Ildoo Kim}
\address[I. Kim]{Department of Mathematics, Korea University, 145 Anam-ro, Seongbuk-gu, Seoul, 02841, Republic of Korea}
\email{waldoo@korea.ac.kr}

\thanks{The authors were supported by the National Research Foundation of Korea(NRF) grant funded by the Korea government(MSIT) (No.2020R1A2C1A01003959)}



\maketitle
\begin{abstract}
We obtain the existence, uniqueness, and regularity estimates of the following Cauchy problem 
\begin{equation}
\begin{cases}
								\label{ab eqn}
\p_tu(t,x)=\psi(t,-i\nabla)u(t,x)+f(t,x),\quad &(t,x)\in(0,T)\times\mathbb{R}^d,\\
u(0,x)=0,\quad & x\in\mathbb{R}^d,
\end{cases}
\end{equation}
in (Muckenhoupt) weighted $L_p$-spaces with time-measurable pseudo-differential operators
\begin{align}
								\label{ab op}
\psi(t,-i\nabla)u(t,x):=\cF^{-1}\left[\psi(t,\cdot)\cF[u](t,\cdot)\right](x).
\end{align}
More precisely, we find sufficient conditions of the symbol $\psi(t,\xi)$ (especially depending on the smoothness of the symbol with respect to $\xi$) to guarantee that equation \eqref{ab eqn} is well-posed in (Muckenhoupt) weighted $L_p$-spaces. 
Here the symbol $\psi(t,\xi)$ is merely measurable with respect to $t$, and the sufficient smoothness of $\psi(t,\xi)$ with respect to $\xi$ is characterized by a property of each weight. In particular, we prove the existence of a positive constant $N$ such that for any solution $u$ to equation \eqref{ab eqn},
\begin{align}
							\label{ab est}
\int_0^T \int_{\bR^d} |(-\Delta)^{\gamma/2} u(t,x) |^p (t^2 + |x|^2)^{\alpha/2} \mathrm{d}x\mathrm{d}t
\leq N\int_0^T \int_{\bR^d} |f(t,x)|^p (t^2 + |x|^2)^{\alpha/2} \mathrm{d}x\mathrm{d}t
\end{align}
and
\begin{align}
							\label{ab est 2}
\int_0^T \left(\int_{\bR^d} |(-\Delta)^{\gamma/2} u(t,x) |^p |x|^{\alpha_2} \mathrm{d}x \right)^{q/p} t^{\alpha_1}\mathrm{d}t
\leq N\int_0^T \left(\int_{\bR^d} |f(t,x) |^p |x|^{\alpha_2} \mathrm{d}x \right)^{q/p} t^{\alpha_1}\mathrm{d}t,
\end{align}
where $p,q\in(1,\infty)$, $-d-1<\alpha < (d+1)(p-1)$, $-1 < \alpha_1 < q-1$, $-d <\alpha_2< d(p-1)$, and $\gamma$ is the order of the operator $\psi(t,-i\nabla)$.
\end{abstract}

\mysection{Introduction}
Pseudo-differential operators have been interesting operators in mathematics for a long time. 
They include not only classical differential operators with a natural number order but also fractional differential operators. 
Moreover, many interesting non-local operators are contained in the class of pseudo-differential operators and it naturally connects them to generators of Markov processes (cf. \cite{niel2001, niel2002, niel2005, taka1984}). 

$L_p$-theories also have a long history in mathematics. There are tremendous results showing $L_p$-boundedness of many interesting operators in Fourier analysis (cf. \cite{grafakos2014classical, stein2016harmonic}).
They have played important roles in theories of partial differential equations (PDEs) to show the well-posedness of solutions  to PDEs in the Sobolev space.
Especially, there commonly appear singularities to obtain optimal regularity estimates of solutions to PDEs and they are overcome by collaboration of many deep theories in analysis and PDEs. Classically, $L_p$-theories in PDEs mostly targeted second-order equations (cf. \cite{krylov2008lectures, lady1988}).
However, these days beyond second-order operators, there occur numerous interesting kinds of research handling high-order operators, non-local operators, and pseudo-differential operators (cf. \cite{chen2018lp, choi2020maximal, dong2012lp, gallarati2017maximal, kang2021lp, kim2015parabolic, kim2016Lp, kim2016lplq, kim2018lp, kim2021lq, mikulevivcius1992, mikulevivcius2014, mikulevivcius2017p, mikulevivcius2019cauchy, zhang2013maximal, zhang2013lp}). 
Even for stochastic PDEs, there have been improvements in $L_p$-theories for various operators (cf. \cite{chang2012stochastic, han2021regularity,kim2013parabolic,kim2016lp, kim2012lp, kim2021sobolev}). 
Particularly, we need to mention that there is progress for these theories in weighted $L_p$-spaces recently (cf. \cite{dong2021nonlocal, dong2021sobolev,gallarati2017maximal, han2020weighted,emiel21,neerven2012stochastic, neerven2012maximal,neerven2015maximal, park2021maximal,portal2019stochastic}).

As mentioned above, there have been tons of researches studying the properties of pseudo-differential operators and the well-posedness of PDEs with them. 
However, these theories are built up in   an elliptic setting mostly.  In other words, there are not many results handling PDEs with pseudo-differential operators in a parabolic setting even though generalizations from elliptic theories to parabolic theories  are considered difficult and important  in theories of mathematics (especially in  PDEs). 
In particular, if an additional (time) variable does not satisfy any regularity condition, then this parabolic generalization becomes complicated and non-trivial even though there are many well-constructed theories in an elliptic setting. 
Our plan of this paper is to construct weighted $L_p$-theories  to equation \eqref{ab eqn} with time-measurable differential operators defined in \eqref{ab op}.
To the best of our knowledge, there are not many results handling general time-measurable differential operators in $L_p$-spaces even without weights. 
For previous results with these operators, we refer readers to \cite{kim2015parabolic,kim2016lplq, kim2018lp}  in $L_p$, $L_q(L_p)$, and $L_p(C^\alpha)$ spaces without weights.

Next, we explain a connection between our results and classical important theories in Fourier analysis.
Recall classical Mikhlin's multiplier theorem (cf. \cite[Section 6.2]{grafakos2014classical}). 
Let $\psi(\xi)$ be a complex-valued function defined on $\bR^d$ and assume that there exists a positive constant $M$ such that
\begin{align}
					\label{2021-11-17 01}
|D^{\alpha}_{\xi}\psi(\xi)|\leq M|\xi|^{-|\alpha|},\quad \forall \xi \in \bR^d 
~\text{for any multi-index $\alpha$ with $|\alpha| \leq \lfloor d/2 \rfloor$}+1,
\end{align}
 where $\lfloor d/2 \rfloor$  denotes the integer which is less than or equal to $d/2$.
Then the operator
$$
\cT_{\psi} f (x) := \cF^{-1}\left[\psi\cF[f]\right](x)
$$
becomes $L_p$-bounded, \textit{i.e.}, for any $p \in (1,\infty)$, there exists a positive constant $N(d,p,M)$ such that 
\begin{align*}
\int_{\bR^d} |\cT_{\psi} f (x)|^p \mathrm{d}x 
\leq N \int_{\bR^d} |f(x)|^p \mathrm{d}x \qquad \forall f \in L_p(\bR^d).
\end{align*}
However, if a weight $w$ is given, then the smoothness on $\psi$ in \eqref{2021-11-17 01} are not sufficient to guarantee 
\begin{align}
								\label{2021 11 17 02}
\int_{\bR^d} |\cT_{\psi} f (x)|^p w(x)\mathrm{d}x 
\leq N \int_{\bR^d} |f(x)|^p w(x) \mathrm{d}x \qquad \forall f \in L_p(\bR^d).
\end{align}
In other words, we expect an extra condition on a symbol $\psi$ if a weight $w$ is additionally given in $L_p$-estimates.
Particularly, it is well-known that  \eqref{2021 11 17 02} holds if $\psi$ satisfies stronger smooth conditions than \eqref{2021-11-17 01} depending on the weight $w$. For instance, let $w$ be a weight in Muckenhoupt's class $A_{pl/d}(\bR^d)$ (see Definition \ref{def weight}) with $s \in (1,2]$, $d/s <l \leq d$, and 
$d/l < p$. Then, if
\begin{align*}
|D^{\alpha}_{\xi}\psi(\xi)|\leq N|\xi|^{\gamma-|\alpha|},\quad \forall \xi \in \bR^d
~\text{for any multi-index $\alpha$ with $|\alpha| \leq l$},
\end{align*}
we have \eqref{2021 11 17 02} (cf. \cite[Theorem 1]{kurtz1979results}).
Based on these theories in Fourier analysis, we  find appropriate conditions on $\psi(t,\xi)$ to enable equation \eqref{ab eqn} to be well-posed in weighted $L_p$-spaces. Formally, the solution $u$ to \eqref{ab eqn} is given by
\begin{align*}
u(t,x) 
&=  \int_0^t \int_s^t \mathrm{e}^{r \psi(r,-i\nabla)}(x-\cdot )f(s,\cdot) \mathrm{d}r \mathrm{d}s := \int_0^t  \int_{\bR^d}  p(t,s,x-y) f(s,y)\mathrm{d}y\mathrm{d}s,
\end{align*}
where
\begin{align*}
p(t,s,x):=1_{0 < s< t} \cdot \frac{1}{(2\pi)^{d/2}}\int_{\bR^d} \exp\left(\int_{s}^t\psi(r,\xi)\mathrm{d}r\right)\mathrm{e}^{ix\cdot\xi}\mathrm{d}\xi.
\end{align*}
Moreover, if $\psi(r,\xi) \simeq |\xi|^\gamma$, then 
\begin{align*}
f \mapsto (-\Delta)^{\gamma/2}u
\end{align*}
becomes a (parabolic) singular integral operator (see Section \ref{section para sing}), where
\begin{align}
								\label{2021 11 18 01}
(-\Delta)^{\gamma/2}u(t,x)=
\int_0^t  \int_{\bR^d} \left( (-\Delta)^{\gamma/2} p \right) (t,s,x-y) f(s,y)\mathrm{d}y\mathrm{d}s
\end{align}
and
\begin{align*}
\left( (-\Delta)^{\gamma/2}p \right) (t,s,x):=1_{0 < s< t} \cdot \frac{1}{(2\pi)^{d/2}}\int_{\bR^d} |\xi|^\gamma \exp\left(\int_{s}^t\psi(r,\xi)\mathrm{d}r\right)\mathrm{e}^{ix\cdot\xi}\mathrm{d}\xi.
\end{align*}
Similarly to  elliptic cases, if 
\begin{align}
					\label{2021 11 17 11}
|D^{\alpha}_{\xi}\psi(t,\xi)|\leq N|\xi|^{\gamma-|\alpha|},\quad \forall (t,\xi) \in \bR \times (\bR^d\setminus\{0\})
~\text{for any multi-index $\alpha$ with $|\alpha| \leq \lfloor d/2 \rfloor$+1},
\end{align}
then $L_p$ and $L_q(L_p)$ norms of $(-\Delta)^{\gamma/2}u(t,x)$ can be controlled by $f$'s (see \cite{kim2015parabolic,kim2016lplq}).
However, if a weight $w$ is given, the upper bound condition of $\psi(t,\xi)$ in \eqref{2021 11 17 11} should be enhanced. 
That is, it needs to  hold for any multi-index with $|\alpha|$ which is greater than $\lfloor d/2 \rfloor$ depending on the weight $w$. 
We characterize this as requiring smoothness based on a constant related to each weight. We call this constant {\bf a regularity constant} and denote it by  $R_{p,d}^w$ (see \eqref{2021-01-19-01}). In particular, one can easily check that
$$
\left\lfloor \frac{d}{2} \right\rfloor  \leq  \left\lfloor\frac{d}{R_{p,d}^w}\right\rfloor  \leq d  + 1.
$$
In our main theorems (Theorem \ref{21.01.08.17.12} and Theorem  \ref{21.01.08.17.12-2}), we characterize numbers of smoothness of symbols $\psi(t,\xi)$  satisfying \eqref{2021 11 17 11} depending on  $A_p$-weights $w$, $w_1$, and $w_2$ with a help of regularity constants for us to find a positive constant $N$ which is independent of $f$ such that
\begin{align*}
\int_0^T \int_{\bR^d} |(-\Delta)^{\gamma/2} u(t,x) |^p w(t,x) \mathrm{d}x\mathrm{d}t
\leq N\int_0^T \int_{\bR^d} |f(t,x)|^p w(t,x) \mathrm{d}x\mathrm{d}t
\end{align*}
and
\begin{align*}
\int_0^T \left(\int_{\bR^d} |(-\Delta)^{\gamma/2} u(t,x) |^p w_2(x) \mathrm{d}x \right)^{q/p}w_1(t)\mathrm{d}t
\leq N\int_0^T \left(\int_{\bR^d} |f(t,x) |^p w_2(x) \mathrm{d}x \right)^{q/p} w_1(t) \mathrm{d}t.
\end{align*}
Estimates \eqref{ab est} and \eqref{ab est 2} are given as particular cases of our theorems. 
We need to mention that if a weight is given by a constant or depends only on time, then our regularity constants are not optimal (Remark \ref{opti rem})
and this initiates for us to prepare the next papers handling only time-dependent weights.

Due to \eqref{2021 11 18 01}, the most important part of our theory is to show $L_p$-boundedness of parabolic singular integral operators.
For future applications, we consider general parabolic singular integral operators in the form of
\begin{align*}
\cT_{\varepsilon} f(t,x) = \int_{-\infty}^t \int_{\bR^d} K_{\varepsilon}(t,s,x-y)f(s,y)\mathrm{d}y\mathrm{d}s,
\end{align*}
where $K_{\varepsilon}$ is a kernel satisfying 
\begin{align}
							\label{2021 11 18 10}
       \sup_{x \in \bR^d}|x|^n|\p_t^mD^{\alpha}_xK_{\varepsilon}(t,s,x)| 
       \leq N_1 |t-s|^{-(m+\varepsilon)-\frac{(d+|\alpha|-n)}{\gamma}}  
\end{align}
and
\begin{align}
							\label{2021 11 18 11}
        \left(\int_{\bR^d}|x|^{2n}|\p_t^mD^{\alpha}_xK_{\varepsilon}(t,s,x)|^2\mathrm{d}x\right)^{1/2}\leq N_2|t-s|^{-(m+\varepsilon)-\frac{(d+|\alpha|-n)}{\gamma}+\frac{d}{2\gamma}},
\end{align}
where $\alpha$  is a multi-index.
For weights $w \in A_p(\bR^{d+1})$,  $w_2 \in A_p(\bR^{d})$, and  $w_1 \in A_p(\bR)$, we show (Theorem \ref{20.12.17.11.23})
\begin{align}
						\label{2021 11 18 30}
\int_0^T \int_{\bR^d} |\cT_{\varepsilon} f(t,x) |^p w(t,x) \mathrm{d}x\mathrm{d}t
\leq NT^{1-\varepsilon}\int_0^T \int_{\bR^d} |f(t,x)|^p w(t,x) \mathrm{d}x\mathrm{d}t
\end{align}
and
\begin{align}
						\label{2021 11 18 31}
\int_0^T \left(\int_{\bR^d} |\cT_{\varepsilon} f(t,x) |^p w_2(x) \mathrm{d}x \right)^{q/p}w_1(t)\mathrm{d}t
\leq NT^{1-\varepsilon}\int_0^T \left(\int_{\bR^d} |f(t,x) |^p w_2(x) \mathrm{d}x \right)^{q/p} w_1(t) \mathrm{d}t
\end{align}
if $K_{\varepsilon}$ satisfies \eqref{2021 11 18 10} and \eqref{2021 11 18 11} for any multi-index $\alpha$ which is less than or equal to a positive integer related to regularity constants 
$R_{p,d+1}^w$, $R_{p,d}^{w_1}$, and $R_{q,1}^{w_2}$. 
Our main tool is Hardy-Littlewood's maximal function and Fefferman-Stein's sharp function.
Based on $L_2$-boundedness of $\cT_{\varepsilon}$ and kernel estimates given in \eqref{2021 11 18 10} and \eqref{2021 11 18 11}, we prove 
\begin{align}
						\label{2021 11 18 20}
(\cT_{\varepsilon} f)^{\sharp}(t,x)\leq NT^{1-\varepsilon}(\bM|f|^{p_0}(t,x))^{1/p_0},
\end{align}
for  $p_0 \in (1,2]$ depending on a given weight, where $(\cT_{\varepsilon} f)^{\sharp}$ denotes the sharp function of $\cT_{\varepsilon} f$ and $\bM|f|$ is the standard maximal function of $f$ (see Section \ref{sharp maximal sec} for the explicit definitions). 
Due to the equivalence of the sharp and maximal operators in weighted $L_p$-spaces, \eqref{2021 11 18 20} leads us to obtain \eqref{2021 11 18 30} and \eqref{2021 11 18 31}.

Our main results are given in Section \ref{21.11.30.14.19} and the existence and uniqueness of solutions to \eqref{ab eqn} for a smooth data $f$ is shown in Section \ref{21.11.30.14.20}.
In Section \ref{21.11.30.14.21}, we obtain an $L_2$-estimate and auxiliary estimates of a fundamental solution to \eqref{ab eqn}.
The Boundedness of general parabolic singular integral operators is handled in Section \ref{section para sing} and \ref{21.11.30.14.22}.
Many interesting properties of weighted $L_p$-spaces are self-contained in Appendix \ref{append} since most well-known properties of Sobolev spaces become unclear due to the effects of unbounded weights. 

\vspace{2mm}
We finish this section with the notations used in the article.

\begin{itemize}
\item 
Let $\bN$, $\bZ$, $\bR$, $\bC$ denote the natural number system, the integer number system, the real number system, and the complex number system, respectively. For $d\in\bN$, $\bR^d$ denotes the $d$-dimensional Euclidean space. 
\item 
For $i=1,...,d$, a multi-index $\alpha=(\alpha_{1},...,\alpha_{d})$ with
$\alpha_{i}\in\{0,1,2,...\}$, and function $g$, we set
$$
\frac{\partial g}{\partial x^{i}}=D_{x^i}g,\quad
D^{\alpha}g=D_{x^1}^{\alpha_{1}}\cdot...\cdot D^{\alpha_{d}}_{x^d}g,\quad |\alpha|:=\sum_{i=1}^d\alpha_i.
$$
For $\alpha_i =0$, we define $D^{\alpha_i}_{x^i} f = f$. We denote the gradient of a function $g:\bR^d\to\bR$ by 
$$
\nabla g = (D_{x^1}g, D_{x^2}g, \cdots, D_{x^d}g).
$$
If $g=g(t,x):\bR\times\bR^d\to\bR^d$, then we denote
$$
 D_x^{\alpha}g=D_{x^1}^{\alpha_{1}}\cdot...\cdot D^{\alpha_{d}}_{x^d}g,
$$
where $\alpha=(\alpha_1,\cdots,\alpha_d)$.
\item 
Let $C_c^{\infty}(\bR^d)$ denote the space of infinitely differentiable functions with compact support, $\cS(\bR^d)$ be the Schwartz space on $\bR^d$ and $\cS'(\bR^d)$ be the space of tempered distributions on $\bR^d$. The convergence $f_n \to f$ in $S(\bR^d)$ as $n \to \infty$ implies
\begin{align*}
\sup_{x \in \bR^d} |(x^1)^{\alpha_1}\cdots(x^d)^{\alpha_d} (D^\beta (f_n-f))(x)| \to 0 \quad \text{as}~ n \to \infty\quad \forall \alpha_1,\cdots\alpha_d,\beta.
\end{align*}

\item 
Let $F$ be a normed space and $(X,\mathcal{M},\mu)$ be a measure space.
\begin{itemize}
    \item $\mathcal{M}^{\mu}$ denotes the completion of $\cM$ with respect to the measure $\mu$.
    \item For $p\in[1,\infty)$, the space of all $\mathcal{M}^{\mu}$-measurable functions $f : X \to F$ with the norm 
\[
\left\Vert f\right\Vert _{L_{p}(X,\cM,\mu;F)}:=\left(\int_{X}\left\Vert f(x)\right\Vert _{F}^{p}\mu(\mathrm{d}x)\right)^{1/p}<\infty
\]
is denoted by $L_{p}(X,\cM,\mu;F)$. We also denote by $L_{\infty}(X,\cM,\mu;F)$ the space of all $\mathcal{M}^{\mu}$-measurable functions $f : X \to F$ with the norm
$$
\|f\|_{L_{\infty}(X,\cM,\mu;F)}:=\inf\left\{r\geq0 : \mu(\{x\in X:\|f(x)\|_F\geq r\})=0\right\}<\infty.
$$
We usually omit the given measure and $\sigma$-algebra if there is no confusion (\textit{e.g.} Lebesgue (or Borel) measure and $\sigma$-algebra).
\item  
We denote by $C(X;F)$ the space of all $F$-valued continuous functions $f : X \to F$ with the norm 
$$
|f|_{C(X;F)}:=\sup_{x\in X}|f(x)|_F<\infty.
$$
\end{itemize}
We omit $F$ if $F=\bR$ or $F=\bC$.

\item
For $r>0$,
$$
B_r(x):=\{y\in\bR^d:|x-y|< r\},\quad
\overline{B_r(x)}:=\{y\in\bR^d:|x-y|\leq r\}
$$

\item 
For $\cO\subseteq \bR^d$, the set of all Borel sets contained in $\cO$ is denoted by $\cB(\cO)$. We denote by $|\cO|$ the Lebesgue measure of a measurable set $\cO \subset \bR^d$. For locally integrable function $f$ on $\bR^d$ and bounded measurable set $A\subseteq \bR^d$ satisfying $|A|>0$,
$$
\aint_{A}f(x)\mathrm{d}x:=\frac{1}{|A|}\int_{A}f(x)\mathrm{d}x
$$

\item 
For integrable function $f$ on $\bR^d$, we denote the $d$-dimensional Fourier transform of $f$ by 
\[
\cF[f](\xi) := \frac{1}{(2\pi)^{d/2}}\int_{\bR^{d}} \mathrm{e}^{- i\xi \cdot x} f(x) \mathrm{d}x
\]
and the $d$-dimensional inverse Fourier transform of $f$ by 
\[
\cF^{-1}[f](x) := \frac{1}{(2\pi)^{d/2}}\int_{\bR^{d}} \mathrm{e}^{  ix \cdot \xi} f(\xi) \mathrm{d}\xi.
\]
We also use the same notations $\cF$ and $\cF^{-1}$ as Fourier transform operator on tempered distributions or $L_2(\bR^d)$ functions.
\item 
We write $N=N(a,b,\cdots)$, if the constant $N$ depends only on $a,b,\cdots$. 
\item 
For $a,b\in \bR$,
$$
a \wedge b := \min\{a,b\},\quad a \vee b := \max\{a,b\},\quad \lfloor a \rfloor:=\max\{n\in\bZ: n\leq a\}.
$$
For $z\in\bC$, $\Re[z]$ denotes the real part of $z$, $\Im[z]$ is the imaginary part of $z$ and $\bar{z}$ is the complex conjugate of $z$.
\end{itemize}

\mysection{Main results}
\label{21.11.30.14.19}
Throughout the paper, we fix $T \in (0,\infty)$ and $d \in \bN$. 
For a suitable complex-valued locally integrable function $\psi(t,\xi)$ on $[0,T] \times \bR^d$ and almost every $t \in [0,T]$,
we  can consider a pseudo-differential operator $\psi(t,i \nabla)$ given by 
$$
\psi(t,-i\nabla)u(x):=\cF^{-1}\left[\psi(t,\cdot)\cF[u]\right](x) \qquad u \in C_c^\infty(\bR^d).
$$
More generally, for a complex-valued nice function $u(t,x)$ on $[0,T] \times \bR^d$
$$
\psi(t,-i\nabla)u(t,x):=\cF^{-1}\left[\psi(t,\cdot)\cF[u](t,\cdot)\right](x).
$$
This function $\psi(t,\xi)$ is usually called {\bf the symbol} of the pseudo-differential operator. 
The operator $\psi(t,-i\nabla)$ is called {\bf a time-measurable pseudo-differential operator} if there is no regularity condition on the symbol $\psi$ with respect to $t$. 
By considering natural constant extensions at $t=0$ and $t=T$, we may assume the symbol $\psi(t,\xi)$ is defined on $\bR \times \bR^d$. 
In this paper, we study the following Cauchy problem with a time-measurable pseudo-differential operator
\begin{equation}
\label{21.01.07.15.22}
\begin{cases}
\p_tu(t,x)=\psi(t,-i\nabla)u(t,x)+f(t,x),\quad &(t,x)\in(0,T)\times\mathbb{R}^d,\\
u(0,x)=0,\quad & x\in\mathbb{R}^d
\end{cases}
\end{equation}
and find appropriate conditions on the symbol for us to establish a well-posedness theory to \eqref{21.01.07.15.22} in weighted $L_p$-spaces. 
First, we introduce related functions spaces handling solutions, and inhomogeneous data.

\begin{defn}[Muckenhoupt weight]
						\label{def weight}
For $p\in(1,\infty)$, let $A_p(\bR^d)$ be the class of all nonnegative and locally integrable functions $w$ satisfying
$$
[w]_{A_p(\bR^d)}:=\sup_{x_0\in\bR^d,r>0}\left(\aint_{B_r(x_0)}w(x)\mathrm{d}x\right)\left(\aint_{B_r(x_0)}w(x)^{-1/(p-1)}\mathrm{d}x\right)^{p-1}<\infty.
$$
\end{defn}
We need to relate each $A_p$-weight to a nonnegative integer to mention the appropriate smoothness of symbols guaranteeing our $L_p$-theory. 
\begin{rem}
				\label{21.02.23.13.03}
The class $A_p(\bR^d)$ is increasing as $p$ increases, and it holds that
$$
A_p(\bR^d)=\bigcup_{q\in(1,p)}A_q(\bR^d).
$$
More precisely, by \cite[Corollary 7.2.6]{grafakos2014classical}, for any $w\in A_p(\bR^d)$, there exist $c_{w}=c_w(d,p,[w]_{A_p(\bR^d)})>0$ and $C_{w}=C_{w}(d,p,[w]_{A_p(\bR^d)})>0$ such that 
$$
1<q:=\frac{p+c_w}{1+c_w}<p,\quad \text{and}\quad [w]_{A_q(\bR^d)}\leq C_w[w]_{A_p(\bR^d)}.
$$
Note that the constants $C_w$ and $c_w^{-1}$ increase according to $[w]_{A_p(\bR^d)}$. 
\end{rem}

\begin{defn}[Regularity constant of a weight in $A_p$]
							\label{regular exponent}
Let  $ w \in A_p(\bR^d)$ and define
\begin{align}
							\label{2021-01-19-01}
R_{p,d}^{w} := \sup \{ p_0 \in (1,2] :  w \in A_{p/p_0}(\bR^d) \}.
\end{align}
We say that $R_{p,d}^{w}$ is \textbf{the regularity constant of the weight $w\in A_p(\bR^d)$} since this constant plays an important role to characterize the differentiability of a symbol to make $L_p$-theories possible.  
\end{defn}
\begin{rem}
In this remark, we assume $w\in A_p(\mathbb{R}^d)$, and we denote 
$$
p_w:=\frac{p(1+c_w)}{p+c_w}\wedge 2\in(1,2].
$$
As pointed out in Remark \ref{21.02.23.13.03}, we have $w\in A_{p/p_w}(\mathbb{R}^d)$. 
Therefore, the set $I_{w,p}:=\{r \in[p/2,\infty):w\in A_r(\mathbb{R}^d)\}$ is either $[p/2,\infty)$ or $(q,\infty)$ with $p/2\leq q< p$. This implies that the constant $R_{p,d}^w$ is well-defined (finite).

We do not know if $w \in A_{p/R_{p,d}^{w}}(\bR^d)$ for each $w \in A_p(\bR^d)$ in general since there is no guarantee that $p/R_{p,d}^{w} \in I_{w,p}$.
However, there exists a constant $p_0 \in (1,2]$ such that $w \in A_{p/p_0}(\bR^d)$, $p_0 \leq R_{p,d}^w \wedge p$, and 
$$
\left\lfloor\frac{d}{p_0}\right\rfloor = \left\lfloor\frac{d}{R_{p,d}^w}\right\rfloor.
$$
To prove this claim, note that $\left\lfloor d/p\right\rfloor$ is left-continuous and piecewise-constant with respect to $p$.
Thus, there exists a $p_0 \in \left(1,R_{p,d}^{w} \right)$ such that $\left\lfloor d/p_0\right\rfloor = \left\lfloor d/R_{p,d}^w\right\rfloor$ and $w \in A_{p/p_0}(\bR^d)$.
It remains to show that $p_0 \leq R_{p,d}^{w} \wedge p$.

If $p \geq 2$, then it is obvious that $p_0< 2$, so we only need to consider the case $p \in (1,2)$.
However, it should be noted that $A_p(\bR^d)$ is only defined for $p>1$, and $A_1(\bR^d)$-class is not introduced in this paper (see Definition \eqref{def weight}).
Therefore, for any $p_0 \in (1,2)$ such that $w \in A_{p/p_0}(\bR^d)$, it is clear that $p_0$ is less than $p$.
\end{rem}

\begin{example}
The easiest examples of a function in $A_p(\bR^d)$ are polynomials, $w_{\alpha}(x):=|x|^{\alpha}$. It is well-known that $w_{\alpha}\in A_p(\bR^d)$ if and only if $-d < \alpha < d(p-1)$ (cf. \cite[Example 7.1.7]{grafakos2014classical}). This implies that for $p_0>1$,
$$
p_0 < \frac{p}{\frac{\alpha}{d} + 1} = \frac{pd}{\alpha+d}\quad \Longleftrightarrow\quad w_{\alpha}\in A_{p/p_0}(\bR^d).
$$
Therefore, for $w_{\alpha}(x):=|x|^{\alpha}$ with $-d < \alpha < d(p-1)$,  one can easily check that the exact value of $R_{p,d}^{w_{\alpha}}$ is given by $\frac{pd}{\alpha+d}$.
\end{example}

Next, we introduce two types of weighted $L_p$-spaces.
\begin{defn}
Let $p,q \in (1,\infty)$. 
\begin{enumerate}[(i)]
\item
For  $w_1\in A_q(\bR)$ and $w_2\in A_{p}(\bR^d)$ , we write $f\in L_q(\bR,w_1;L_p(\bR^d,w_2))$ if
\begin{equation*}
\|f\|_{L_q(\bR,w_1;L_p(\bR^d,w_2))}:=\left(\int_{-\infty}^{\infty}\left(\int_{\bR^d}|f(t,x)|^pw_2(x)\mathrm{d}x\right)^{q/p}w_1(t)\mathrm{d}t\right)^{1/q}<\infty.
\end{equation*}
\item
For   $w\in A_{p}(\bR^{d+1})$ , we write $f\in L_p(\bR^{d+1},w)$ if
\begin{equation*}
    \|f\|_{L_p(\bR^{d+1},w)}:=\left( \int_{-\infty}^\infty\int_{\bR^{d}}|f(t,x)|^pw(t,x)\mathrm{d}x \mathrm{d}t\right)^{1/p}<\infty.
\end{equation*}
\end{enumerate}
Note that even if $p=q \in (1,\infty)$, there is no inclusion between $L_q(\bR,w_1;L_p(\bR^d,w_2))$ and $L_p(\bR^{d+1},w)$.
\end{defn}
Next, we introduce Sobolev spaces (Bessel potentials) which fit our weighted spaces. 
Recall that $\cS(\bR^d)$ and  $\cS'(\bR^d)$ denote the Schwartz space   and the space of tempered distributions on $\bR^d$, respectively.
\begin{defn} 
						\label{defn solu space}
Let $p\in(1,\infty)$.
\begin{enumerate}[(i)]
    \item For $\nu\in\bR$ and $w\in A_p(\bR^d)$, $H_p^{\nu}(\bR^d,w)$ denotes the set of all  $f \in \cS'(\bR^d)$  satisfying
$$
\|f\|_{H_p^{\nu}(\bR^d,w)}:=\|(1-\Delta)^{\nu/2}f\|_{L_p(\bR^d,w)}<\infty,
$$
where
$$
(1-\Delta)^{\nu/2} f(x) := \cF^{-1} \left[(1+|\cdot|^2)^{\nu/2}\cF [f] \right](x).
$$
If $w\equiv1$, then we omit $w$.
\item For $\nu\in[0,\infty)$ and $w\in A_p(\bR^{d+1})$, $\bH_p^{\nu}((0,T)\times\bR^d,w)$ denotes the set of all $\cS'(\bR^d)$-valued measurable functions $f$ on $(0,T)$ satisfying
$$
\|f\|_{\bH_p^{\nu}((0,T)\times\bR^d,w)}:=\|f\|_{L_p((0,T)\times\bR^d,w)}+\|(-\Delta)^{\nu}f\|_{L_p((0,T)\times\bR^d,w)}<\infty,
$$
where
$$
(-\Delta)^{\nu/2} f(t,x) := \cF^{-1} \left[|\cdot|^{\nu}\cF [f(t,\cdot)] \right](x).
$$
If $\nu<0$, then $\bH_p^{\nu}((0,T)\times\bR^d,w)$ denotes the set of all $\cS'(\bR^d)$-valued measurable functions $f$ on $(0,T)$ satisfying
$$
\|f\|_{\bH_p^{\nu}((0,T)\times\bR^d,w)}:=\|(1-\Delta)^{\nu/2}f\|_{L_p((0,T)\times\bR^d,w)}<\infty.
$$
\end{enumerate}
\end{defn}
\begin{rem}
						\label{rem 2021-12-29}
Generally, there is no guarantee that two norms  
$$
\|f\|_{L_p((0,T)\times\bR^d,w)}+\|(-\Delta)^{\nu}f\|_{L_p((0,T)\times\bR^d,w)}\quad \text{and}\quad\|(1-\Delta)^{\nu/2}f\|_{L_p((0,T)\times\bR^d,w)}
$$
are equivalent due to the effect of weights.
However, if our weight $w$ additionally satisfies 
\begin{equation}
\label{21.12.05.14.03}
\esssup_{t\in[0,T]}[w(t,\cdot)]_{A_p(\bR^d)}=: N_0<\infty,    
\end{equation}
then two norms are equivalent, which will be shown in Appendix \ref{append}.
In particular, we show that the polynomial $w(t,x):=(t^2+|x|^2)^{\alpha/2}$ with $-d<\alpha<d(p-1)$ satisfies \eqref{21.12.05.14.03} (see Lemma \ref{lem 2021-12-29}).

\end{rem}

\begin{defn}
\label{21.01.08.11.02}
Let $p\in(1,\infty)$. 
 \begin{enumerate}[(i)]
    \item  $C_p^{\infty}([0,T]\times\bR^d)$ denotes the set of all $\cB([0,T]\times\bR^d)$-measurable functions $f$ on $[0,T]\times\bR^d$ such that for any multi-index $\alpha$ with respect to the space variable, their derivatives
\begin{equation*}
D^{\alpha}_{x}f\in L_{\infty}([0,T];L_2(\bR^d)\cap L_p(\bR^d)).
\end{equation*}
\item $C_p^{1,\infty}([0,T]\times\bR^d)$ denotes the set of all $f\in C_p^{\infty}([0,T]\times\bR^d)$ such that for any multi-index $\alpha$, their derivatives
\begin{equation*}
D^{\alpha}_{x}f\in C([0,T]\times \bR^d)\quad \text{and}\quad \p_tf\in C_p^{\infty}([0,T]\times\bR^d),
\end{equation*}
where $\p_tf(t)$ denotes the right derivative and the left derivative at $t=0$ and $t=T$, respectively.
\end{enumerate}
\end{defn}
\begin{rem}
\label{21.01.25.16.01}
($i$) In our previous work~\cite{choi2020maximal}, we introduced similar smooth function spaces.
The operators discussed in~\cite{choi2020maximal} have integral representations and are $L_1$-bounded.
However, in the present article, we address general pseudo-differential operators that do not possess $L_1$-integrability and lack integral representations.
Therefore, it becomes necessary to consider a larger class of functions than those presented in~\cite{choi2020maximal}, in order to approximate solutions.

($ii$) Let $f \in C_p^{\infty}([0,T]\times\bR^d)$. Note that for any ($d$-dimensional) multi-index $\alpha$ and  almost all $t\in[0,T]$,
$$
D^{\alpha}_xf(t,\cdot)\in C(\bR^d)
$$
due to the Sobolev Embedding theorem (\textit{e.g.} \cite[Theorem 13.1.8]{krylov2008lectures}).
\end{rem}

\begin{defn}[Solution]
\label{21.03.01.15.51}
Let $\nu\in\bR$, $p,q\in(1,\infty)$, $w\in A_p(\bR^{d+1})$, $w_1\in A_q(\bR)$, and $w_2\in A_p(\bR^d)$.  For a given 
$$f\in \bH_p^{\nu}((0,T)\times\bR^d,w)\quad(resp.\,\, f\in L_q((0,T),w_1;H_p^{\nu}(\bR^d,w_2))\,),$$
we say that $u\in\bH_p^{\nu+\gamma}((0,T)\times\bR^d,w)$ (resp. $u\in L_q((0,T),w_1;H_p^{\nu+\gamma}(\bR^d,w_2))$\,) is a solution to the Cauchy problem
\begin{equation*}
\begin{cases}
\p_tu(t,x)=\psi(t,-i\nabla)u(t,x)+f(t,x),\quad &(t,x)\in(0,T)\times\bR^d,\\
u(0,x)=0,\quad & x\in\bR^d,
\end{cases}
\end{equation*}
if there exists a sequence of functions $u_n\in C_p^{1,\infty}([0,T]\times\bR^d)$ such that $u_n(0,\cdot)=0$,
\begin{equation*}
\begin{gathered}
\p_tu_n-\psi(t,-i\nabla)u_n\to f\quad\text{in}\quad \bH_p^{\nu}((0,T)\times\bR^d,w),\quad (resp.\,\,L_q((0,T),w_1;H_p^{\nu}(\bR^d,w_2))\,),
\end{gathered}
\end{equation*}
and
$$
u_n\to u \quad\text{in}\quad \bH_p^{\nu+\gamma}((0,T)\times\bR^d,w),\quad (resp.\,\,L_q((0,T),w_1;H_p^{\nu+\gamma}(\bR^d,w_2))\,)
$$
as $n\to\infty$.
\end{defn}

\begin{defn}[Ellipticity condition of a symbol]
\label{20.11.18.13.07}
We say that a symbol $\psi$ satisfies the ellipticity condition (with ($\gamma$, $\kappa$)) if there exists a $\gamma\in(0,\infty)$ and $\kappa\in(0,1]$ such that
$$
\Re [-\psi(t,\xi)]\geq \kappa |\xi|^{\gamma},\quad \forall (t,\xi)\in \bR\times\bR^d,
$$
where $\Re[z]$ denotes the real part of the complex number $z$.
\end{defn}

\begin{defn}[regular upper bounds of a symbol]
					\label{regular upper}
Let $n \in \bN$. We say that a symbol $\psi$ has {\bf a $n$-times regular upper bound} (with ($\gamma$, $M$)) if there exist positive constants $\gamma$ and $M$ 
such that 
$$
|D^{\alpha}_{\xi}\psi(t,\xi)|\leq M|\xi|^{\gamma-|\alpha|},\quad \forall (t,\xi)\in\bR\times(\bR^{d}\setminus\{0\}),
$$
for any ($d$-dimensional) multi-index $\alpha$ with $|\alpha| \leq n$. 
\end{defn}

Recall the regularity constant for a weigh  $w \in A_p(\bR^d)$, \textit{i.e.},
$$
R_{p,d}^{w} := \sup \{ p_0 \in (1,2] :  w \in A_{p/p_0}(\bR^d) \}.
$$
Here are the main results of this article and we prove them in Section \ref{pf main theorems}.
\begin{thm}
\label{21.01.08.17.12}
Let $p\in(1,\infty)$ and $w\in A_p(\bR^{d+1})$.
Suppose that $\psi(t,\xi)$ is a symbol satisfying the ellipticity condition with $(\gamma,\kappa)$ (Definition \ref{20.11.18.13.07}) and having a $( \lfloor d/R_{p,d+1}^{w}\rfloor+2)$-times regular upper bound with $(\gamma,M)$ (Definition \ref{regular upper}).
Then for any $f\in L_p((0,T)\times\bR^d,w)$, the Cauchy problem \eqref{21.01.07.15.22}  has a unique solution $u\in\bH_p^{\gamma}((0,T)\times\bR^d,w)$ with estimates
    \begin{equation}
							\label{a priori 1}
        \begin{gathered}
        \|\partial_tu\|_{L_p((0,T)\times\bR^d,w)}+\|(-\Delta)^{\gamma/2}u\|_{L_p((0,T)\times\bR^d,w)}\leq N\|f\|_{L_p((0,T)\times\bR^d,w)},\\
        \|u\|_{\bH_p^{\gamma}((0,T)\times\bR^d,w)}\leq N(1+T)\|f\|_{L_p((0,T)\times\bR^d,w)},
        \end{gathered}
    \end{equation}
where $N=N(d,p,\gamma,\kappa,M,[w]_{A_p(\bR^{d+1})})$.
Additionally assume that $w(t,\cdot)\in A_p(\bR^d)$ for almost all $t\in[0,T]$ and
\begin{equation*}
   \esssup_{t\in[0,T]}[w(t,\cdot)]_{A_p(\bR^d)}=:N_0<\infty. 
\end{equation*}
Then, for all $\nu\in\bR$ and $f\in \bH_p^{\nu}((0,T)\times\bR^d,w)$, the Cauchy problem \eqref{21.01.07.15.22}  has a unique solution $u\in\bH_p^{\nu+\gamma}((0,T)\times\bR^d,w)$ with estimates
    \begin{equation}
								\label{a priori 2}
        \begin{gathered}
        \|\partial_tu\|_{\bH_p^{\nu}((0,T)\times\bR^d,w)}+\|(-\Delta)^{\gamma/2}u\|_{\bH_p^{\nu}((0,T)\times\bR^d,w)}\leq N\|f\|_{\bH_p^{\nu}((0,T)\times\bR^d,w)},\\
        \|u\|_{\bH_p^{\nu+\gamma}((0,T)\times\bR^d,w)}\leq N(1+T)\|f\|_{\bH_p^{\nu}((0,T)\times\bR^d,w)},
        \end{gathered}
    \end{equation}
where $N=N(d,p,\gamma,\kappa,M,[w]_{A_p(\bR^{d+1})},\nu,N_0)$.
\end{thm}

\begin{thm}
\label{21.01.08.17.12-2}
Let $p,q\in(1,\infty)$, $w_1\in A_q(\bR)$, and $w_2\in A_p(\bR^d)$. 
Suppose that $\psi(t,\xi)$ is a symbol satisfying the ellipticity condition with $(\gamma,\kappa)$ (Definition \ref{20.11.18.13.07}) and having a $\left( \lfloor d/R_{q,1}^{w_1}\rfloor\vee\lfloor d/R_{p,d}^{w_2}\rfloor+2\right)$-times regular upper bound with $(\gamma,M)$ (Definition \ref{regular upper}).
Then, for any  $f\in L_q((0,T),w_1;H_p^{\nu}(\bR^d,w_2))$, the Cauchy problem \eqref{21.01.07.15.22} has a unique solution $u\in L_q((0,T),w_1;H_p^{\nu+\gamma}(\bR^d,w_2))$ with estimates
    \begin{equation*}
        \begin{gathered}
        \|\partial_tu\|_{L_q((0,T),w_1;H_p^{\nu}(\bR^d,w_2))}+\|(-\Delta)^{\gamma/2}u\|_{L_q((0,T),w_1;H_p^{\nu}(\bR^d,w_2))}\leq N\|f\|_{L_q((0,T),w_1;H_p^{\nu}(\bR^d,w_2))},\\
        \|u\|_{L_q((0,T),w_1;H_p^{\nu+\gamma}(\bR^d,w_2))}\leq N(1+T)\|f\|_{L_q((0,T),w_1;H_p^{\nu}(\bR^d,w_2))},
        \end{gathered}
    \end{equation*}
where $N=N(d,p,\nu,\gamma,\kappa,M,[w_1]_{A_q(\bR)},[w_2]_{A_p(\bR^{d})})$.
\end{thm}

\begin{rem}
					\label{opti rem}
If our weight is a constant function (\textit{i.e.}, without weight), the differentiability $(  \lfloor d/R_{p,d+1}^{w}\rfloor+2)$ for the symbol 
$\psi(t,\xi)$ in Theorem  \ref{21.01.08.17.12} is not optimal.
Indeed, if $w$ is a constant, then 
$$
\left\lfloor \frac{d}{R_{p,d+1}^{w}}\right\rfloor+2=\left\lfloor \frac{d}{2} \right\rfloor +2.
$$
However, if there is no weight term in the estimate, $\left\lfloor \frac{d}{2} \right\rfloor + 1$ is enough differentiability for the symbol $\psi(t,\xi)$ to obtain \eqref{a priori 1} as shown in \cite{kim2015parabolic} (or see Theorem \ref{20.12.21.13.23}). 
 Moreover, if one considers a weight depending only on the time variable $t$,  then the differentiability $ \lfloor d/R_{w,p}\rfloor+2$ can be relaxed to 
$\lfloor d/R_{w,p}\rfloor+1$. We will show this result in our next paper based on $A_2$-theory for operator-valued Calder\'on-Zygmund theorem 
(cf. see \cite[Chapter 3.4]{emiel21}).
For a general $A_p$-weight $w$ depending on both the time variable $t$ and space variable $x$, we do not know yet whether the regular constant 
$ \lfloor d/R_{p,d+1}^{w}\rfloor+2$ for the symbol is optimal to guarantee \eqref{a priori 1}, which seems to be an interesting open problem. 
\end{rem}

For applications, we believe that the most important example of weights is a polynomial type. 
We give explicit statements for the reader's convenience when weights are given by a polynomial type. 
\begin{thm}
								\label{poly thm}
Let $p\in(1,\infty)$  and  $\alpha \in (-d-1, (d+1)(p-1))$.
Suppose that $\psi(t,\xi)$ is a symbol satisfying the ellipticity condition with $(\gamma,\kappa)$ and having a $\left( \left\lfloor \frac{d(\alpha+d+1)}{p(d+1)}   \right\rfloor+2\right)$-times regular upper bound with $(\gamma,M)$ (Definition \ref{regular upper}).

Then, for any $f\in L_p((0,T)\times\bR^d,(t^2 + |x|^2)^{\alpha/2})$, the Cauchy problem \eqref{21.01.07.15.22}  has a unique solution $u\in\bH_p^{\gamma}((0,T)\times\bR^d,(t^2 + |x|^2)^{\alpha/2})$ with estimates
\begin{align*}
\int_0^T \int_{\bR^d} |(-\Delta)^{\gamma/2} u(t,x) |^p (t^2 + |x|^2)^{\alpha/2} \mathrm{d}x\mathrm{d}t
\leq N\int_0^T \int_{\bR^d} |f(t,x)|^p (t^2 + |x|^2)^{\alpha/2} \mathrm{d}x\mathrm{d}t
\end{align*}
and
\begin{align*}
\int_0^T \int_{\bR^d} | u(t,x) |^p (t^2 + |x|^2)^{\alpha/2} \mathrm{d}x\mathrm{d}t
\leq N T\int_0^T \int_{\bR^d} |f(t,x)|^p (t^2 + |x|^2)^{\alpha/2} \mathrm{d}x\mathrm{d}t,
\end{align*}
where $N=N(d,p,\gamma,\kappa,M,\alpha)$.
\end{thm}

\begin{thm}
								\label{poly thm 2}
Let $p,q\in(1,\infty)$,  $\alpha_1 \in (-1, q-1)$, and $\alpha_2 \in (-d , d(p-1))$. 
Suppose that $\psi(t,\xi)$ is a symbol satisfying the ellipticity condition with $(\gamma,\kappa)$ and having a $\left( \left\lfloor\frac{d(\alpha_1+1)}{q}\right\rfloor\vee\left\lfloor\frac{\alpha_2+d}{p}\right\rfloor+2\right)$-times regular upper bound with $(\gamma,M)$ (Definition \ref{regular upper}).
Then, for any  $f\in L_q((0,T),t^{\alpha_1};L_p(\bR^d,|x|^{\alpha_2}))$, the Cauchy problem \eqref{21.01.07.15.22} has a unique solution $u\in L_q((0,T),t^{\alpha_1};H_p^{\gamma}(\bR^d,|x|^{\alpha_2}))$ with estimates

\begin{align*}
\int_0^T \left(\int_{\bR^d} |(-\Delta)^{\gamma/2} u(t,x) |^p |x|^{\alpha_2} \mathrm{d}x \right)^{q/p} t^{\alpha_1}\mathrm{d}t
\leq N\int_0^T \left(\int_{\bR^d} |f(t,x) |^p |x|^{\alpha_2} \mathrm{d}x \right)^{q/p} t^{\alpha_1}\mathrm{d}t
\end{align*}
and
\begin{align*}
\int_0^T \left(\int_{\bR^d} | u(t,x) |^p |x|^{\alpha_2} \mathrm{d}x \right)^{q/p} t^{\alpha_1}\mathrm{d}t
\leq N T\int_0^T \left(\int_{\bR^d} |f(t,x) |^p |x|^{\alpha_2} \mathrm{d}x \right)^{q/p} t^{\alpha_1}\mathrm{d}t
\end{align*}
where $N=N(d,p,\gamma,\kappa,M, \alpha_1, \alpha_2)$.
\end{thm}

\mysection{The Cauchy problems with smooth functions}
\label{21.11.30.14.20}

Recall the target equation
\begin{equation}
				\label{21.01.12.12.19}
\begin{cases}
\p_tu(t,x)=\psi(t,-i\nabla)u(t,x)+f(t,x),\quad &(t,x)\in(0,T)\times\mathbb{R}^d,\\
u(0,x)=0,\quad & x\in\mathbb{R}^d,
\end{cases}
\end{equation}
where 
$$
\psi(t,-i\nabla)u(t,x):=\cF^{-1}\left[\psi(t,\cdot)\cF[u](t,\cdot)\right](x).
$$
Then the fundamental solution to \eqref{21.01.12.12.19} is given by
\begin{align*}
p(t,s,x):=1_{0 < s< t} \cdot \frac{1}{(2\pi)^{d/2}}\int_{\bR^d} \exp\left(\int_{s}^t\psi(r,\xi)\mathrm{d}r\right)\mathrm{e}^{ix\cdot\xi}\mathrm{d}\xi,
\end{align*}
that is, the integral operator 
\begin{align}
						\label{k defn}
\cK f (t,x) := \int_{-\infty}^{t}\int_{\bR^d}p(t,s,x-y)1_{0<s<T}f(s,y)\mathrm{d}y\mathrm{d}s
\end{align}
becomes a solution to \eqref{21.01.12.12.19} if $f$ is a sufficiently smooth function. 
In this section, we prove that the function $p(t,s,x)$ becomes the fundamental solution to \eqref{21.01.12.12.19} indeed.
We prove some preliminaries first.

\begin{lem}
\label{21.01.25.16.54}
Let $p\in(1,\infty)$.
\begin{enumerate}[(i)]
    \item For any $f\in C_p^{\infty}([0,T]\times\bR^d)$,
    $$
    \int_0^tf(s,x)\mathrm{d}s\in C_p^{1,\infty}([0,T]\times\bR^d).
    $$
    \item For all $\nu\in\bR$, $(1-\Delta)^{\nu/2}$ is a bijective linear operator on $C_p^{\infty}([0,T]\times\bR^d)$.
    \item Suppose that $\psi(t,\xi)$ has a $\left(\lfloor \frac{d}{2} \rfloor+1\right) $-times regular upper bound with $(\gamma,M)$.
    Then $\psi(t,-i\nabla)$ is a bounded linear operator on $C_p^{\infty}([0,T]\times\bR^d)$.
\end{enumerate}

\end{lem}
\begin{proof} ($i$) 
Put 
$$
g(t,x):=\int_0^tf(s,x)\mathrm{d}s.
$$
By the dominated convergence theorem and Minkowski's inequality,
\begin{equation}
\label{21.01.25.16.15}
\|D^{\alpha}_xg(t_2,\cdot)-D^{\alpha}_xg(t_1,\cdot)\|_{L_2(\bR^d)\cap L_p(\bR^d)}\leq |t_2-t_1|\|D^{\alpha}_xf\|_{L_{\infty}([0,T];L_2(\bR^d)\cap L_p(\bR^d))}
\end{equation}
for all positive numbers $t_1$ and $t_2$. 
This implies $g\in C_p^{\infty}([0,T]\times\bR^d)$ and by Remark \ref{21.01.25.16.01}($ii$) and \eqref{21.01.25.16.15}, $g$ is continuous on $[0,T]\times\bR^d$. The fundamental theorem of calculus yields for almost all $t\in[0,T]$,
$$
\p_tg(t,x)=f(t,x),\quad \forall x\in \bR^d.
$$
Therefore, $g\in C_p^{1,\infty}([0,T]\times\bR^d)$.

($ii$) Let $f\in C_p^{\infty}([0,T]\times\bR^d)$ and $\nu \in \bR$. 
It is well-known that $(1-\Delta)^{\nu/2}$ is a bijective mapping on the Schwartz class (cf. \cite[Chapter 13]{krylov2008lectures}).
Thus it suffices to show
$$
(1-\Delta)^{\nu/2}f\in C_p^{\infty}([0,T]\times\bR^d)
$$
since $\nu$ is an arbitrary real number, 
Let $\alpha$ be a ($d$-dimensional) multi-index $\alpha$ (with respect to space variable).
By a well-known property of the operator $(1-\Delta)^{\nu/2}$ in $L_p$-spaces, 
\begin{align*}
\|D^{\alpha}_x(1-\Delta)^{\nu/2}f\|_{L_{\infty}([0,T];L_2(\bR^d)\cap L_p(\bR^d))}&=\|(1-\Delta)^{\nu/2}D^{\alpha}_xf\|_{L_{\infty}([0,T];L_2(\bR^d)\cap L_p(\bR^d))}\\
&\leq N\sum_{|\beta|\leq|\alpha|+\lfloor|\nu|\rfloor+1}\|D^{\beta}f\|_{L_{\infty}([0,T];L_2(\bR^d)\cap L_p(\bR^d))},   
\end{align*}
where $N$ is independent of $f$. This certainly implies $(1-\Delta)^{\nu/2}f\in C_p^{\infty}([0,T]\times\bR^d)$.

($iii$) Let $m(t,\xi):=|\xi|^{-\gamma}\psi(t,\xi)$. 
Since $\psi$ has a $\left(\lfloor \frac{d}{2} \rfloor+1\right) $-times regular upper bound with $\gamma$ and $M$, we have
$$
|D^{\alpha}_{\xi}m(t,\xi)|\leq N(\alpha,\gamma,M)|\xi|^{-|\alpha|},\quad \forall |\alpha|\leq \left\lfloor\frac{d}{2}\right\rfloor+1, \quad \forall (t,\xi) \in [0,T] \times (\bR^d\setminus\{0\})
$$
by a simple multiplicative law of derivatives. 
Thus, using Mikhlin's multiplier theorem (\textit{e.g.} \cite[Theorem 6.2.7]{grafakos2014classical}), for all $f\in C_p^{\infty}([0,T]\times\bR^d)$,  ($d$-dimensional) multi-index $\alpha$, and   $s\in[0,T]$, we have
\begin{align*}
    \|D^{\alpha}\psi(t,-i\nabla)f(s,\cdot)\|_{L_2(\bR^d)\cap L_p(\bR^d)}\leq N\|(-\Delta)^{\gamma/2}D^{\alpha}f(s,\cdot)\|_{L_2(\bR^d)\cap L_p(\bR^d)},
\end{align*}
where $N=N(d,p,\gamma,M)$. 
Moreover, since $p\in(1,\infty)$, by a well-known embedding theorem which can be easily obtained from Mikhlin's multiplier theorem, we have
\begin{align*}
\|(-\Delta)^{\gamma/2}D^{\alpha}f(s,\cdot)\|_{L_2(\bR^d)\cap L_p(\bR^d)}&\leq N\sum_{|\beta|\leq |\alpha|+\lfloor\gamma\rfloor+1}\|D^{\beta}f(s,\cdot)\|_{L_2(\bR^d)\cap L_p(\bR^d)}\\
&\leq N\sum_{|\beta|\leq|\alpha|+ \lfloor\gamma\rfloor+1}\|D^{\beta}f\|_{L_{\infty}([0,T];L_2(\bR^d)\cap L_p(\bR^d))},    
\end{align*}
where $N$ is independent of $f$. The lemma is proved.
\end{proof}

\begin{thm}
\label{21.01.08.11.03}
Let $p\in(1,\infty)$ and $f\in C_p^{\infty}([0,T]\times\bR^d)$.
Assume that $\psi(t,\xi)$ satisfies the ellipticity condition with $(\gamma,\kappa)$ and has a $\left(\lfloor \frac{d}{2} \rfloor  +1 \right) $-times regular upper bound with $(\gamma,M)$.
Then, there exists a unique solution $u\in C_p^{1,\infty}([0,T]\times\bR^d)$  to the Cauchy problem
\begin{equation}
					\label{21.01.08.11.06}
\begin{cases}
\p_tu(t,x)=\psi(t,-i\nabla)u(t,x)+f(t,x),\quad &(t,x)\in(0,T)\times\mathbb{R}^d,\\
u(0,x)=0,\quad & x\in\mathbb{R}^d,
\end{cases}
\end{equation}
Moreover, the solution $u$ has the following representation :
$$
u(t,x)=\cK f(t,x) ,\quad  \forall (t,x)\in [0,T]\times\bR^d,
$$
where $\cK f$ is defined in \eqref{k defn}.
\end{thm}
\begin{proof}
\textbf{(Uniqueness)} Suppose that $u\in C_p^{1,\infty}([0,T]\times\bR^d)$ is a solution of the Cauchy problem
\begin{equation}
\label{21.01.25.16.51}
\begin{cases}
\p_tu(t,x)=\psi(t,-i\nabla)u(t,x),\quad &(t,x)\in(0,T)\times\bR^d,\\
u(0,x)=0,\quad & x\in\bR^d.
\end{cases}
\end{equation}
Taking the Fourier transform to \eqref{21.01.25.16.51} and using the fundamental theorem of calculus, we have
$$
|\cF[u(t,\cdot)](\xi)|\leq M|\xi|^{\gamma}\int_0^t|\cF[u(s,\cdot)](\xi)|\mathrm{d}s.
$$
By Gronwall's inequality, we conclude that $u=0$. From the linearity of equation \eqref{21.01.08.11.06}, we obtain the uniqueness. 
\vspace{3mm}

\textbf{(Existence)} Define
\begin{equation*}
u(t,x):=\cK f(t,x).
\end{equation*}
First, we claim that $u\in C_p^{\infty}([0,T]\times\bR^d)$.
In  virtue of the dominated convergence theorem,  for a ($d$-dimensional) multi-index $\alpha$ and $(t,x)\in[0,T]\times\bR^d$,
\begin{equation*}
    D^{\alpha}_xu(t,x)=D^{\alpha}_x\cK f(t,x)=\cK (D^{\alpha}_xf)(t,x).
\end{equation*}
Note that by \cite[Corollary 3.6]{kim2018lp},
\begin{equation}
\label{21.01.09.19.45}
\sup_{s<t}\int_{\bR^d}|p(t,s,x)|\mathrm{d}x\leq N,    
\end{equation}
where $N=N(d,\gamma,\kappa,M)$. Using Minkowski's inequality and \eqref{21.01.09.19.45}, for all positive numbers $t_1$ and $t_2$, we have
\begin{align*}
\|\cK(D^{\alpha}_xf)(t_2,\cdot)-\cK(D^{\alpha}_xf)(t_1,\cdot)\|_{L_p(\bR^d)}&\leq \left|\int_{t_1}^{t_2}\|p(t,s,\cdot)\|_{L_1(\bR^d)}\|D^{\alpha}f(s,\cdot)\|_{L_p(\bR^d)}\mathrm{d}s\right|\\
&\leq N|t_2-t_1|\|D^{\alpha}f\|_{L_{\infty}([0,T];L_p(\bR^d))},
\end{align*}
where $N=N(d,\gamma,\kappa,M)$. Thus, since  $f\in C_p^{\infty}([0,T]\times\bR^d)$, we conclude that $u\in C_p^{\infty}([0,T]\times\bR^d)$. Next, define
$$
v(t,x):=\int_{0}^t \left(\psi(s,-i\nabla)u(s,x)+f(s,x) \right) \mathrm{d}s.
$$
By Lemma \ref{21.01.25.16.54}($i$), $v\in C_p^{1,\infty}([0,T]\times\bR^d)$.
We claim that 
$$
u(t,x)=v(t,x),\quad\forall(t,x)\in[0,T]\times\bR^d,
$$
which obviously completes the proof. For $(t,x)\in[0,T]\times\bR^d$, applying Fubini's theorem, we have
\begin{equation*}
    \begin{aligned}
    v(t,x)-\int_0^tf(s,x)\mathrm{d}s&=\int_{0}^t\cF^{-1}\left[\psi(s,\cdot)\int_0^s\cF[f(l,\cdot)]\mathrm{e}^{\int_{l}^s\psi(r,\cdot)\mathrm{d}r}\mathrm{d}l\right](x)\mathrm{d}s\\
    &=\int_0^t\cF^{-1}\left[\cF[f(l,\cdot)](\mathrm{e}^{\int_{l}^t\psi(r,\cdot)\mathrm{d}r}-1)\right](x)\mathrm{d}l=u(t,x)-\int_0^tf(s,x)\mathrm{d}s.
    \end{aligned}
\end{equation*}
The theorem is proved.
\end{proof}

\mysection{An $L_2$-boundedness and estimates of the fundamental solution}
\label{21.11.30.14.21}

In this section, we obtain $L_2$-estimates of the Cauchy problem and prove some properties of the fundamental solution.
Recall 
\begin{align*}
p(t,s,x):=1_{0 < s< t} \cdot \frac{1}{(2\pi)^{d/2}}\int_{\bR^d} \exp\left(\int_{s}^t\psi(r,\xi)\mathrm{d}r\right)\mathrm{e}^{ix\cdot\xi}\mathrm{d}\xi.
\end{align*}
Let $f\in C_p^{\infty}([0,T]\times\bR^d)$. Then by Theorem \ref{21.01.08.11.03},
\begin{align*}
u(t,x)=\cK f (t,x) := \int_{-\infty}^{t}\int_{\bR^d}p(t,s,x-y)1_{0<s<T}f(s,y)\mathrm{d}y\mathrm{d}s
\end{align*}
becomes the solution to \eqref{21.01.08.11.06}. Moreover, one can easily check that for any $\varepsilon\in[0,1]$,
\begin{equation}
    \label{eqn 2020012102}
    \begin{aligned}
(-\Delta)^{\frac{\varepsilon \gamma}{2}}u(t,x)
&=(-\Delta)^{\frac{\varepsilon \gamma}{2}} \cK f (t,x)  \\
&=\int_{-\infty}^{t} \left[\int_{\bR^d}P_\varepsilon (t,s,x-y)1_{0<s<T}f(s,y)\mathrm{d}y \right] \mathrm{d}s \\
&:=\int_{-\infty}^{t} \left[\int_{\bR^d}(-\Delta)^{\frac{\varepsilon\gamma}{2}} p(t,s,x-y)1_{0<s<T}f(s,y)\mathrm{d}y \right] \mathrm{d}s.
\end{aligned}
\end{equation}

Thus to show the boundedness of $(-\Delta)^{\frac{\varepsilon \gamma}{2}}u$ in an $L_p$-space, it is required to estimate the kernel $P_\varepsilon$.

\begin{lem}
\label{20.11.19.13.24}
Let $n\in \bN$ and assume that $\psi(t,\xi)$ satisfies the ellipticity condition with $(\gamma,\kappa)$ and has a $n$-times regular upper bound with $(\gamma,M)$.
Then there exists a positive constant $N=N(n,M)$ such that for all $t>s$ and $\xi \in \bR^d$,
\begin{equation}
\label{21.01.06.14.58}
\left|D^n_{\xi}\left(\exp\left(\int_{s}^t\psi(r,\xi)\mathrm{d}r\right)\right)\right|\leq N|\xi|^{-n}\exp(-\kappa(t-s)|\xi|^{\gamma})\sum_{k=1}^n|t-s|^{k}|\xi|^{k\gamma},
\end{equation}
where $D^n_{\xi}$ denotes a ($d$-dimensional) partial derivative with respect to the variable $\xi$ whose order is $n$. 
\end{lem}
\begin{proof}
We use mathematical induction. 
The case $n=1$ is clear since $|D_\xi \psi(t,\xi)|$ has a upper bound $N|\xi|^{\gamma-1}$ and $|\psi(t,\xi)|$ has both lower and upper bounds which are given by $|\xi|^\gamma$ multiplied by positive constants. 
Next, suppose that \eqref{21.01.06.14.58} holds for all $1,2,\cdots,k$. Then, by Leibniz's product rule and the hypothesis of the induction,
\begin{align*}
    &\left|D^{k+1}_{\xi}\left(\exp\left(\int_{s}^t\psi(r,\xi)\mathrm{d}r\right)\right)\right|=\left|D^{k}_{\xi}\left(\exp\left(\int_{s}^t\psi(r,\xi)\mathrm{d}r\right)\int_{s}^{t}D_{\xi}^1\psi(r,\xi)\mathrm{d}r\right)\right|\\
    &\leq N\sum_{j=0}^k\left|D^{j}_{\xi}\left(\exp\left(\int_{s}^t\psi(r,\xi)\mathrm{d}r\right)\right)\int_{s}^{t}D_{\xi}^{k-j+1}\psi(r,\xi)\mathrm{d}r\right|\\
    &\leq N\sum_{j=0}^k\left(|\xi|^{-j}\exp(-\kappa(t-s)|\xi|^{\gamma})\sum_{l=1}^j|t-s|^{l}|\xi|^{l\gamma}|t-s||\xi|^{\gamma-k+j-1}\right)\\
    &= N|\xi|^{-k-1}\exp(-\kappa(t-s)|\xi|^{\gamma})\sum_{j=0}^k\sum_{l=1}^j|t-s|^{l+1}|\xi|^{(l+1)\gamma}\\
    &\leq N|\xi|^{-k-1}\exp(-\kappa(t-s)|\xi|^{\gamma})\sum_{l=1}^{k+1}|t-s|^{l}|\xi|^{l\gamma},
\end{align*}
where $N=N(k,M)$. The lemma is proved.
\end{proof}
\begin{thm}
\label{21.01.13.14.11}
Let $n\in \bN$, $\alpha$ be a ($d$-dimensional) multi-index, and $m\in\{0,1\}$.
Assume that $\psi(t,\xi)$ satisfies the ellipticity condition with $(\gamma,\kappa)$ and  has a $n$-times regular upper bound with $(\gamma,M)$.
Then there exists a  positive constant $N=N(n,\kappa,\gamma,M,\varepsilon,m,|\alpha|)$ such that for all $s<t$,
\begin{equation*}
    \begin{gathered}
    \sup_{x\in\bR^d}|x|^{n} \left(|\p_t^mD^{\alpha}_xP_{\varepsilon}(t,s,x)| \right) \leq N|t-s|^{-(m+\varepsilon)-\frac{(d+|\alpha|-n)}{\gamma}},\\
    \left(\int_{\bR^d}|x|^{2n}|\p_t^mD^{\alpha}_xP_{\varepsilon}(t,s,x)|^2\mathrm{d}x\right)^{1/2}\leq N|t-s|^{-(m+\varepsilon)-\frac{(d+|\alpha|-n)}{\gamma}+\frac{d}{2\gamma}},
    \end{gathered}
\end{equation*}
\end{thm}
\begin{proof}
It is well-known that
$$
|x|\simeq |x^1|+\cdots+|x^d|~ \quad \forall x \in \bR^d.
$$
In particular, we have,
\begin{equation}
\label{20.11.19.16.20}
 |x|^n\leq N(n)(|x^1|^n+\cdots+|x^d|^n),\quad n\in\bN.   
\end{equation}
By \eqref{20.11.19.16.20} and elementary properties of the Fourier transform,
\begin{align*}
    |x|^{n}|\p_t^mD_x^{\alpha}P_{\varepsilon}(t,s,x)|&\leq N\sum_{i=1}^d|x^i|^n|\p_t^mD^{\alpha}_xP_{\varepsilon}(t,s,x)|\\
    &\leq N\sum_{|\beta|=n}\int_{\bR^d}\left|D_{\xi}^{\beta}\left(\xi^{\alpha}|\xi|^{\varepsilon\gamma}\p_t^m\exp\left(\int_{s}^t\psi(r,\xi)\mathrm{d}r\right)\right)\right|\mathrm{d}\xi.
\end{align*}
By Leibniz's product rule,
\begin{align*}
&D_{\xi}^{\beta}\left(\xi^{\alpha}|\xi|^{\varepsilon\gamma}\p_t^m\exp\left(\int_{s}^t\psi(r,\xi)\mathrm{d}r\right)\right)\\
 &=D_{\xi}^{\beta}\left(\psi(t,\xi)^m\xi^{\alpha}|\xi|^{\varepsilon\gamma}\exp\left(\int_{s}^t\psi(r,\xi)\mathrm{d}r\right)\right)\\
 &=\sum_{\beta=\beta_0+\beta_1}c_{\beta_0,\beta_1}D^{\beta_0}_{\xi}(\psi(t,\xi)^m\xi^{\alpha}|\xi|^{\varepsilon\gamma})D^{\beta_1}_{\xi}\left(\exp\left(\int_{s}^t\psi(r,\xi)\mathrm{d}r\right)\right).
\end{align*}
By Definition \ref{20.11.18.13.07} and Lemma \ref{20.11.19.13.24},
\begin{equation}
\label{20.11.19.16.39}
\begin{aligned}
&\left|D_{\xi}^{\beta}\left(\psi(t,\xi)^m\xi^{\alpha}|\xi|^{\varepsilon\gamma}\exp\left(\int_{s}^t\psi(r,\xi)\mathrm{d}r\right)\right)\right|\\
&\leq N|\xi|^{(m+\varepsilon)\gamma+|\alpha|-n}\exp(-\kappa(t-s)|\xi|^{\gamma})\sum_{k=1}^{n}(t-s)^k|\xi|^{k\gamma}.   
\end{aligned}
\end{equation}
Therefore,
$$
\sum_{|\beta|=n}\int_{\bR^d}\left|D_{\xi}^{\beta}\left(\xi^{\alpha}|\xi|^{\varepsilon\gamma}\p_t^m\exp\left(\int_{s}^t\psi(r,\xi)\mathrm{d}r\right)\right)\right|\mathrm{d}\xi\leq N(t-s)^{-(m+\varepsilon)-\frac{(d+|\alpha|-n)}{\gamma}},
$$
where $N=N(n,\kappa,\gamma,\varepsilon,m,|\alpha|)$. Similarly, by \eqref{20.11.19.16.20}, \eqref{20.11.19.16.39}, and Plancherel's theorem,
\begin{align*}
\int_{\bR^d}|x|^{2n}|\p_t^mD_x^{\alpha} P_{\varepsilon}(t,s,x)|^2\mathrm{d}x&\leq N\int_{\bR^d}|x^i|^{2n}|\psi(t,-i\nabla)^mD_x^{\alpha} P_{\varepsilon}(t,s,x)|^2\mathrm{d}x\\
&\leq N\sum_{|\beta|=n}\int_{\bR^d}\left|D^{\beta}_{\xi}\left(\psi(t,\xi)^m\xi^{\alpha}|\xi|^{\varepsilon\gamma}\exp\left(\int_{s}^t\psi(r,\xi)\mathrm{d}r\right)\right)\right|^2\mathrm{d}\xi\\
&\leq N|t-s|^{-2(m+\varepsilon)-\frac{(d+2|\alpha|-2n)}{\gamma}},
\end{align*}
where $N=N(n,\kappa,\gamma,\varepsilon,m,|\alpha|)$.
The theorem is proved.
\end{proof}

Recalling \eqref{eqn 2020012102} and handling estimates of various regularities of the solution simultaneously, we introduce the following family of  iterated integral operators 
\begin{equation*}
  \cK_{\varepsilon,T} f(t,x):=\int_{-\infty}^{t} \left[ \int_{\bR^d}\left(1_{\varepsilon\in[0,1)}1_{|t-s|<T}+1_{\varepsilon=1}\right)P_{\varepsilon}(t,s,x-y)f(s,y)\mathrm{d}y \right]\mathrm{d}s,
\end{equation*}
where $f \in C_c^\infty(\bR^{d+1})$.  It is remarkable that $P_\varepsilon(t,\cdot,\cdot)$ is integrable on $(0,t) \times \bR^d$ for all $t>0$ and $\varepsilon \in [0,1)$. Thus the integral operator $  \cK_{\varepsilon,T} f$ is well-defined in a point-wise sense. 
However, $P_1(t,\cdot,\cdot)$ is not integrable on $(0,t) \times \bR^d$ for any $t>0$.
In this sense, we say that the operator $\cK_{1,T}f$ is (parabolic) singular. However, one can understand the function $\cK_{1,T}f(t,x)$ in a point-wise sense if one regards the integral as an iterated integral. Indeed,
\begin{align*}
  \cK_{1,T} f(t,x)
  &:=\int_{-\infty}^{t} \left[ \int_{\bR^d} P_1(t,s,x-y)f(s,y)\mathrm{d}y \right]\mathrm{d}s \\
  &=\int_{-\infty}^{t} \left[ \int_{\bR^d}p(t,s,x-y)\left((-\Delta)_y^{\frac{\gamma}{2}}f\right)(s,y)\mathrm{d}y \right]\mathrm{d}s \\
  &=\int_{-\infty}^{t} \int_{\bR^d}p(t,s,x-y)\left((-\Delta)_y^{\frac{\gamma}{2}}f\right)(s,y)\mathrm{d}y \mathrm{d}s.
\end{align*}
Moreover, since the mapping $f \mapsto \cK_{\varepsilon,T} f$ is a pseudo-differential operator as well, the $L_2$-boundedness can be easily obtained based on Plancherel's theorem without any regularity condition on the symbol.
\begin{thm}
\label{21.03.01.17.32}
Assume that the symbol $\psi$ satisfies the ellipticity condition with $(\gamma,\kappa)$. 
Then there exists a constant $N=N(\kappa)$ such that
$$
\|\cK_{\varepsilon,T} f\|_{L_2(\bR^{d+1})}\leq NT^{1-\varepsilon}\|f\|_{L_2(\bR^{d+1})}
\quad \forall f\in C_c^{\infty}(\bR^{d+1}).
$$
\end{thm}
\begin{proof}
 By Plancherel's theorem, Fubini's theorem, and Minkowski's inequality,
\begin{align*}
    &\int_{-\infty}^{\infty}\int_{\bR^d}|\cK_{\varepsilon,T}f(t,x)|^2\mathrm{d}x\mathrm{d}t\\
    &= \int_{-\infty}^{\infty}\int_{\bR^d}\left|\int_0^t|\xi|^{\varepsilon\gamma}\exp\left(\int_s^t\psi(r,\xi)\mathrm{d}r\right)\cF[f(s,\cdot)](\xi)h(\varepsilon,t-s,T)\mathrm{d}s\right|^2\mathrm{d}\xi \mathrm{d}t\\
    &\leq \int_{-\infty}^{\infty}\int_{\bR^d}\left(\int_0^t|\xi|^{\varepsilon\gamma}\exp\left(-\kappa(t-s)\right|\xi|^{\gamma})|\cF[f(s,\cdot)](\xi)|h(\varepsilon,t-s,T)\mathrm{d}s\right)^2\mathrm{d}\xi \mathrm{d}t\\
    &=\int_{-\infty}^{\infty}\int_{\bR^d}\left(\int_0^t|\xi|^{\varepsilon\gamma}\exp\left(-\kappa s\right|\xi|^{\gamma})|\cF[f(t-s.\cdot)](\xi)|h(\varepsilon,s,T)\mathrm{d}s\right)^2\mathrm{d}\xi \mathrm{d}t \\
    &\leq \int_{\bR^d}\left(\int_{0}^{\infty}|\xi|^{\varepsilon\gamma}\exp\left(-2\kappa s\right|\xi|^{\gamma})h(\varepsilon,s,T)\left(\int_s^{\infty}|\cF[f(t-s,\cdot)](\xi)|^2\mathrm{d}t\right)^{1/2}\mathrm{d}s\right)^{2}\mathrm{d}\xi\\
    &\leq \int_{\bR^d}\int_{-\infty}^{\infty}|\cF[f(t,\cdot)](\xi)|^2\mathrm{d}t\left(\int_{0}^{\infty}|\xi|^{\varepsilon\gamma}\exp\left(-2\kappa s\right|\xi|^{\gamma})h(\varepsilon,s,T)\mathrm{d}s\right)^{2}\mathrm{d}\xi,
\end{align*}
where
$$
h(\varepsilon,s,T):=1_{\varepsilon\in[0,1)}1_{0<s<T}+1_{\varepsilon=1}.
$$
By changing variables,
\begin{align*}
    \int_{0}^{\infty}|\xi|^{\varepsilon\gamma}\exp\left(-2\kappa s\right|\xi|^{\gamma})h(\varepsilon,s,T)\mathrm{d}s=N|\xi|^{(\varepsilon-1)\gamma}\int_0^{\infty}\mathrm{e}^{-s}h(\varepsilon,s,T|\xi|^{\gamma})\mathrm{d}s,
\end{align*}
where $N=N(\kappa)$. If $\varepsilon=1$, then
$$
|\xi|^{(\varepsilon-1)\gamma}\int_0^{\infty}\mathrm{e}^{-s}h(\varepsilon,s,T|\xi|^{\gamma})\mathrm{d}s=\int_0^{\infty}\mathrm{e}^{-s}\mathrm{d}s=1.
$$
If $\varepsilon\in[0,1)$, then we split the estimate into two cases, $T|\xi|^{\gamma}\geq1$ and $0<T|\xi|^{\gamma}<1$. 
If $T|\xi|^{\gamma}\geq1$, then
$$
|\xi|^{(\varepsilon-1)\gamma}\int_0^{\infty}\mathrm{e}^{-s}h(\varepsilon,s,T|\xi|^{\gamma})\mathrm{d}s=|\xi|^{(\varepsilon-1)\gamma}\int_0^{T|\xi|^{\gamma}}\mathrm{e}^{-s}\mathrm{d}s\leq |\xi|^{(\varepsilon-1)\gamma}\leq T^{1-\varepsilon}.
$$
If $0<T|\xi|^{\gamma}<1$, then
$$
|\xi|^{(\varepsilon-1)\gamma}\int_0^{\infty}\mathrm{e}^{-s}h(\varepsilon,s,T|\xi|^{\gamma})\mathrm{d}s=|\xi|^{(\varepsilon-1)\gamma}\int_0^{T|\xi|^{\gamma}}\mathrm{e}^{-s}\mathrm{d}s\leq T|\xi|^{\varepsilon\gamma}\leq T^{1-\varepsilon}.
$$
Therefore, by Plancherel's theorem,
$$
\int_{-\infty}^{\infty}\int_{\bR^d}|\cK_{\varepsilon,T}f(t,x)|^2\mathrm{d}x\mathrm{d}t\leq NT^{2(1-\varepsilon)}\int_{-\infty}^{\infty}\int_{\bR^d}|\cF[f(t,\cdot)](\xi)|^2\mathrm{d}\xi \mathrm{d}t=NT^{2(1-\varepsilon)}\|f\|_{L_2(\bR^{d+1})}^2,
$$
where $N=N(\kappa)$. The theorem is proved.
\end{proof}

\mysection{Weighted boundedness for a family of parabolic singular integral operators}
									\label{section para sing}

Let $u$ be a solution to \eqref{21.01.08.11.06} with a smooth $f$.
Due to the solution representation given in Theorem \ref{21.01.08.11.03}, we have
\begin{align*}
u(t,x) = \cK f (t,x) = \int_{-\infty}^{t}\int_{\bR^d}p(t,s,x-y)1_{0<s<T}f(s,y)\mathrm{d}y\mathrm{d}s
\end{align*}
and
\begin{align*}
(-\Delta)^{\frac{\varepsilon \gamma}{2}}u(t,x)
&=\int_{-\infty}^{t} \left[\int_{\bR^d}P_\varepsilon (t,s,x-y)1_{0<s<T}f(s,y)\mathrm{d}y \right] \mathrm{d}s \\
&:=\int_{-\infty}^{t} \left[\int_{\bR^d}(-\Delta)^{\frac{\varepsilon\gamma}{2}} p(t,s,x-y)1_{0<s<T}f(s,y)\mathrm{d}y \right] \mathrm{d}s.
\end{align*}
Therefore, to obtain a priori estimates \eqref{a priori 1} and \eqref{a priori 2}, we need to show the boundedness of the above integral operators in weighted $L_p$-spaces. For future applications and specifying kernel estimates used in the proofs, we consider kernels in a general setting.
More precisely, for $\varepsilon \in [0,1]$, we consider a family of complex-valued measurable functions $K_{\varepsilon}(t,s,x)$ defined on $\bR\times\bR\times\bR^d $. Then we can consider corresponding (iterated) integral operators given by
\begin{align}
								\label{eqn 2021012101}
\int_0^t \left[\int_{\bR^d}\left(1_{\varepsilon\in[0,1)}1_{0<t-s<T}+1_{\varepsilon=1}\right) K_{\varepsilon}(t,s,y) f(s,x-y)\mathrm{d}y\right] \mathrm{d}s \qquad \forall f \in C_c^\infty(\bR^{d+1}).
\end{align}
We say that the integral operator is (parabolic) singular if $\int_0^t \int_{\bR^d} K_{\varepsilon}(t,s,y)\mathrm{d}y \mathrm{d}s=\infty$ for some $t >0$. 
It is remarkable that \eqref{eqn 2021012101} is well-defined even though the kernel $K_{\varepsilon}(t,s,y)$ is singular.
A typical example of these kernels are  classical heat kernels, \textit{i.e.},
$$
K_\varepsilon(t,x,y)=\Delta^\varepsilon \left( (2\pi t)^{d/2} \mathrm{e}^{- |x|^2/(2t)} \right).
$$
In particular, the kernel $K_1(t,x,y)=\Delta \left( (2\pi t)^{d/2} \mathrm{e}^{- |x|^2/(2t)} \right)$ is singular. 

Note that if the kernel $K_\varepsilon(t,s,x)$ is not singular and the given weight is a constant, then the boundedness of the operator \eqref{eqn 2021012101} is easily obtained on the basis of the generalized Minkowski inequality. 
However, if a general weight is given, then the boundedness becomes non-trivial even though the kernel is not singular. 
Therefore, we have to show the boundedness of the operators for all $\varepsilon \in [0,1]$ based on weighted estimates, and the cut-off function $1_{|t-s|<T}$ is necessary to overcome the lack of the integrability of the kernels at large $t$ when they are not singular. 
For a simplification of the notation, we denote
    $$
    h(\varepsilon,s,T):=1_{\varepsilon\in[0,1)}1_{0<s<T}+1_{\varepsilon=1}
    $$
and
    $$
    \cT_{\varepsilon,T} f(t,x):=\int_{-\infty}^{t}\int_{\bR^d}h(\varepsilon,t-s,T)K_{\varepsilon}(t,s,x-y)f(s,y)\mathrm{d}y\mathrm{d}s.
    $$
Since we can obtain a more exact upper bound of the integral operator related to $T$ with the help of the cut-off function $h(\varepsilon,t-s,T)$, we separated it from the kernel $K_{\varepsilon}(t,s,x-y)$.

\begin{defn}[Regularity condition on $K_\varepsilon$]
									\label{regular condition}
For $k \in \bN$, we say that a kernel $K_\varepsilon(t,s,x)$ (or the operator $\cT_{\varepsilon,T}$) satisfies the $k$-times regular condition with $(\gamma,N_1,N_2)$ if there exist positive constants $\gamma$, $N_1$, and $N_2$ such that for all $(n,m,|\alpha|)\in \left\{0,1,2,\cdots, k \right\}\times\{0,1\}\times \{0,1,2\}$ and $s<t$,
    \begin{align*}
       \sup_{x \in \bR^d}|x|^n|\p_t^mD^{\alpha}_xK_{\varepsilon}(t,s,x)| 
       \leq N_1 |t-s|^{-(m+\varepsilon)-\frac{(d+|\alpha|-n)}{\gamma}}  
        \end{align*}
and
\begin{align*}
        \left(\int_{\bR^d}|x|^{2n}|\p_t^mD^{\alpha}_xK_{\varepsilon}(t,s,x)|^2\mathrm{d}x\right)^{1/2}\leq N_2|t-s|^{-(m+\varepsilon)-\frac{(d+|\alpha|-n)}{\gamma}+\frac{d}{2\gamma}}.
\end{align*}        
\end{defn}

Recall that for any Muckenhoupt's weight $w$, the constant $ R_{p,d}^{w}$ denotes the  regularity constant of the weight $w$ (Definition \ref{regular exponent}).

\begin{thm}[A weighted extrapolation theorem]
				\label{20.12.17.11.23}
Assume that there exists a constant $N_3>0$ such that
    $$
    \|\cT_{\varepsilon,T} f\|_{L_2(\bR^{d+1})}\leq N_3T^{1-\varepsilon}\|f\|_{L_2(\bR^{d+1})} \quad \forall f \in C_c^\infty(\bR^{d+1}).
    $$
\begin{enumerate}[(i)]
    \item Let $p \in (1,\infty)$ and $w \in A_p(\bR^{d+1})$.  
    Suppose that $\cT_{\varepsilon,T}$ satisfies the $( \lfloor d/R_{p,d+1}^{w}\rfloor+2)$-times regular condition.
    
    Then, there exists a constant $N=N(d,p, [w]_{A_p(\bR^{d+1})},\gamma,\varepsilon,N_1,N_2,N_3)$ such that
    $$
    \|\cT_{\varepsilon,T} f\|_{L_p(\bR^{d+1},w)}\leq NT^{1-\varepsilon}\| f\|_{L_p(\bR^{d+1},w)} \quad \forall f \in C_c^\infty(\bR^{d+1}).
    $$
    \item Let $p,q \in (1,\infty)$, $w_1\in A_q(\bR)$, and $w_2\in A_p(\bR^{d})$.
    Suppose that $\cT_{\varepsilon,T}$ satisfies the $ \left( \lfloor d/R_{q,1}^{w_1}\rfloor\vee \lfloor d/R_{p,d}^{w_2}\rfloor+2 \right)$-times regular condition.
    
    Then, there exists a constant $N=N(d,p,q,[w_1]_{A_q(\bR)},[w_2]_{A_p(\bR^d)},\gamma,\varepsilon,N_1,N_2,N_3)$ such that
    $$
    \|\cT_{\varepsilon,T} f\|_{L_q(\bR,w_1;L_p(\bR^d,w_2))}\leq NT^{1-\varepsilon}\| f\|_{L_q(\bR,w_1;L_p(\bR^d,w_2))}.
    $$
\end{enumerate}
\end{thm}

\mysection{Proof of Theorem \ref{20.12.17.11.23}}
\label{21.11.30.14.22}

Recall that for each $\varepsilon \in [0,1]$,  $K_{\varepsilon}(t,s,x)$ is a complex-valued function defined on $\bR\times\bR\times\bR^d$.
We fix a complex-valued function $K_{\varepsilon}(t,s,x)$ defined on $\bR\times\bR\times\bR^d$ with some $\varepsilon \in [0,1]$ throughout the section. We start the proof with kernel estimates.  
In this section, $\alpha=(\alpha_1,\ldots, \alpha_d)$ denotes a ($d$-dimensional) multi-index and $|\alpha|=\alpha_1+ \cdots + \alpha_d$. 

\subsection{Estimates of a family of parabolic singular kernels}

\begin{thm}
						\label{20.12.24.12.47}
Let $k$ be an integer such that $k>\lfloor d/2\rfloor$ and suppose that $K_{\varepsilon}$ satisfies the $k$-times regular condition with $(\gamma,N_1,N_2)$ (Definition \ref{regular condition}).
\begin{enumerate}[(i)]
    \item Let $p\in[2,\infty]$, $(n,m,|\alpha|)\in\{0,1,2,\cdots k\}\times\{0,1\}\times \{0,1,2\}$, and $\delta\in[0,1)$ be a constant satisfying $n+\delta\leq k$. Then there exists a positive constant $N=N(p,N_1,N_2)$ such that
\begin{equation}
					\label{21.01.12.17.23}
    \left\||\cdot|^{n+\delta}\p_t^mD_x^{\alpha}K_{\varepsilon}(t,s,\cdot)\right\|_{L_p(\bR^d)}\leq N|t-s|^{-(m+\varepsilon)-\frac{(d+|\alpha|-n-\delta)}{\gamma}+\frac{d}{p\gamma}},\quad(c/\infty:=0).
\end{equation}

\item Let $p\in[1,2)$ and $(m,|\alpha|)\in\times\{0,1\}\times \{0,1,2\}$.
Then there exists a positive constant $N=N(p,N_1,N_2)$ such that 
\begin{equation}
\label{21.01.13.13.21}
    \left\|\p_t^mD_x^{\alpha}K_{\varepsilon}(t,s,\cdot)\right\|_{L_p(\bR^d)}\leq N|t-s|^{-(m+\varepsilon)-\frac{(d+|\alpha|)}{\gamma}+\frac{d}{p\gamma}}.
\end{equation}

\end{enumerate}
\end{thm}

\begin{proof} $(i)$ We divide the proof into two cases.

\vspace{2mm}

\textbf{Case 1.} $\delta=0$.

\vspace{2mm}

If $p=2$ or $p=\infty$, then obviously \eqref{21.01.12.17.23} holds  since $K_{\varepsilon}$ satisfies the $k$-times regular condition with $(\gamma,N_1,N_2)$. If $p\in(2,\infty)$, then
\begin{align*}
  &\int_{\bR^d}|x|^{pn}|\p_t^mD_x^{\alpha}K_{\varepsilon}(t,s,x)|^{p}\mathrm{d}x\\
  &=\int_{\bR^d}|x|^{2n}|\p_t^mD_x^{\alpha}K_{\varepsilon}(t,s,x)|^{2}|x|^{(p-2)n}|\p_t^mD_x^{\alpha}K_{\varepsilon}(t,s,x)|^{p-2}\mathrm{d}x\\
  &\leq N(p,N_2)|t-s|^{(p-2)\left(-(m+\varepsilon)-\frac{(d+|\alpha|-n)}{\gamma}\right)}\int_{\bR^d}|x|^{2n}|\p_t^mD_x^{\alpha}K_{\varepsilon}(t,s,x)|^{2}\mathrm{d}x\\
  &\leq N(p,N_1,N_2)|t-s|^{(p-2)\left(-(m+\varepsilon)-\frac{(d+|\alpha|-n)}{\gamma}\right)}|t-s|^{-2(m+\varepsilon)-\frac{(d+2|\alpha|-2n)}{\gamma}}\\
  &=N(p,N_1,N_2)|t-s|^{-p(m+\varepsilon)-\frac{p(d+|\alpha|-n)}{\gamma}+\frac{d}{\gamma}}.
\end{align*}

\vspace{2mm}

\textbf{Case 2.} $\delta\in(0,1)$.

\vspace{2mm}

We use the result from Case 1 repeatedly, \textit{i.e.}, we apply \eqref{21.01.12.17.23} with $\delta=0$ several times for the proof of this case. 
Observe that for any $q\in[1,\infty)$,
\begin{equation}
\label{21.01.12.17.26}
    \begin{aligned}
    &|x|^{q(n+\delta)}|\p_t^mD^{\alpha}_xK_{\varepsilon}(t,s,x)|^q\\
    &=\left(|x|^{qn}|\p_t^mD^{\alpha}_xK_{\varepsilon}(t,s,x)|^q\right)^{1-\delta}\times\left(|x|^{q(n+1)}|\p_t^mD^{\alpha}_xK_{\varepsilon}(t,s,x)|^q\right)^{\delta}.
    \end{aligned}
\end{equation}
Using \eqref{21.01.12.17.26} with $q=1$, \eqref{21.01.12.17.23} holds for $p=\infty$. 
Thus it only remains to consider the case $p \in [2,\infty)$. 
Taking integrals and applying H\"older's inequality to \eqref{21.01.12.17.26} with $q=p$, we have
\begin{align*}
    &\left(\int_{\bR^d}|x|^{p(n+\delta)}|\p_t^mD^{\alpha}_xK_{\varepsilon}(t,s,x)|^p\mathrm{d}x\right)^{1/p}\\
    &\leq \left(\int_{\bR^d}|x|^{p(n+1)}|\p_t^mD^{\alpha}_xK_{\varepsilon}(t,s,x)|^p\mathrm{d}x\right)^{\delta/p}\times\left(\int_{\bR^d}|x|^{pn}|\p_t^mD^{\alpha}_xK_{\varepsilon}(t,s,x)|^p\mathrm{d}x\right)^{(1-\delta)/p}\\
    &\leq N(p,N_1,N_2)|t-s|^{-(m+\varepsilon)-\frac{(d+|\alpha|-n-\delta)}{\gamma}+\frac{d}{p\gamma}}.
\end{align*}
$(ii)$ The case $p=2$ is obvious due to the assumption. Moreover, we claim that it is sufficient to show that \eqref{21.01.13.13.21} holds for $p=1$. 
Indeed, for $p \in (1,2)$ there exists a $\lambda \in (0,1)$ such that $p=\lambda+2(1-\lambda)$.
Thus  taking integrals and applying H\"older's inequality to
\begin{equation*}
    \begin{aligned}
    |\p_t^mD^{\alpha}_xK_{\varepsilon}(t,s,x)|^p=|\p_t^mD^{\alpha}_xK_{\varepsilon}(t,s,x)|^{\lambda}\times|\p_t^mD^{\alpha}_xK_{\varepsilon}(t,s,x)|^{2(1-\lambda)},
    \end{aligned}
\end{equation*}
we obtain \eqref{21.01.13.13.21}.
Thus, we focus on showing \eqref{21.01.13.13.21} with $p=1$. 
Observe that 
$$
\int_{\bR^d}|\p_t^mD^{\alpha}_xK_{\varepsilon}(t,s,x)|\mathrm{d}x=\int_{|x|\leq |t-s|^{1/\gamma}}|\p_t^mD^{\alpha}_xK_{\varepsilon}(t,s,x)|\mathrm{d}x + \int_{|x|> |t-s|^{1/\gamma}}|\p_t^mD^{\alpha}_xK_{\varepsilon}(t,s,x)|\mathrm{d}x.
$$
By $(i)$ with $p=\infty$,
$$
\int_{|x|\leq |t-s|^{1/\gamma}}|\p_t^mD^{\alpha}_xK_{\varepsilon}(t,s,x)|\mathrm{d}x\leq N(d,N_1)|t-s|^{-(m+\varepsilon)-\frac{(d+|\alpha|)}{\gamma}+\frac{d}{\gamma}}.
$$
Put $d_2:=\lfloor d/2\rfloor+1$ and note that $d_2\leq k$.
Then by H\"older's inequality and $(i)$,
\begin{align*}
   &\int_{|x|> |t-s|^{1/\gamma}}|\p_t^mD^{\alpha}_xK_{\varepsilon}(t,s,x)|\mathrm{d}x\\
   &\leq \left(\int_{|x|> |t-s|^{1/\gamma}}|x|^{-2d_2}\mathrm{d}x\right)^{1/2} \left(\int_{|x|> |t-s|^{1/\gamma}}|x|^{2d_2}|\p_t^mD^{\alpha}_xK_{\varepsilon}(t,s,x)|^2\mathrm{d}x\right)^{1/2}\\
   &\leq N(d,N_2)|t-s|^{-\frac{d_2}{\gamma}+\frac{d}{2\gamma}}\times|t-s|^{-(m+\varepsilon)-\frac{(d+|\alpha|-d_2)}{\gamma}+\frac{d}{2\gamma}}\\
   &=N(d,N_2)|t-s|^{-(m+\varepsilon)-\frac{(d+|\alpha|)}{\gamma}+\frac{d}{\gamma}}.
\end{align*}
The theorem is proved.
\end{proof}

\subsection{Auxiliary estimates for integral operators}
									\label{sharp maximal sec}

In this subsection, we introduce an equivalence of weighted $L_p$-norms of sharp and maximal functions with filtration of partitions (cf. \cite{dong18Apweights}) and obtain auxiliary estimates in an appropriate partition that fits our integral operators $\cT_{\varepsilon,T}$.
\begin{defn}
Let $\cB_0=\cB_0(\bR^{d+1})$ be a collection of all Borel sets $A\subseteq\bR^{d+1}$ such that $|A|<\infty$. We say that a collection $\cP\subseteq\cB_0$ is a partition of $\bR^{d+1}$ if and only if elements of $\cP$ are countable, pairwise disjoint, and
$$ \bigcup_{A\in\cP}A=\bR^{d+1}.$$
\end{defn}
\begin{defn}
\label{filtration}
Let $\{\cP_n\}_{n\in\bZ}$ be a sequence of partitions of $\bR^{d+1}_+$. We say that $\{\cP_n\}_{n\in\bZ}$ is a filtration of partitions on $\bR^{d+1}$ if the followings are satisfied ;
\begin{enumerate}[(i)]
\item The partitions are finer as $n\to\infty$. That is,
$$\inf _{A\in\cP_n}|A|\to\infty\quad\text{as}\quad n\to-\infty$$
and for any $f\in L_{1}(\bR^{d+1})$
$$ \lim_{n\to\infty}f_{|n}(t,x):=\lim_{n\to\infty}\frac{1}{|A_n(t,x)|}\int_{A_n(t,x)}f(s,y)\mathrm{d}y\mathrm{d}s=f(t,x)\quad(a.e.),$$
where $A_n(t,x)$ is the elements of $\cP_n$ containing $(t,x)$.
\item For each $n\in\bZ$ and $A\in\cP_n$, there is a (unique) $A'\in\cP_{n-1}$ such that $A\subseteq A'$ and
$$ |A'|\leq N_0|A|,$$
where $N_0$ is a constant independent of $n$, $A$ and $A'$.
\end{enumerate}
\end{defn}
An example of filtration of partitions that will be used in the proof of Theorem \ref{20.12.24.12.47} is given in the following lemma. 
Recall that a nonnegative function $d$ defined on $S \times S$ with a set $S$ is called a quasi-metric on $S$ if 
$$
d(s_1,s_2) =0 \qquad \text{iff}~ s_1=s_2,
$$
$$
d(s_1,s_2) =d(s_2,s_1) \qquad \forall s_1,s_2 \in S,
$$
and there exists a positive constant $N_d$ such that
$$
d(s_1, s_2) \leq N_d \left( d(s_1,s_3) + d(s_3,s_2)\right) \quad \forall s_1,s_2,s_3 \in S.
$$
\begin{lem}
					\label{20.12.22.19.10}
Let $\gamma\in(0,\infty)$ and $n\in\bZ$.
Define
\begin{align*}
    D_n(i_1,\cdots,i_d):=[i_12^{-n},(i_1+1)2^{-n})\times\cdots\times[i_d2^{-n},(i_d+1)2^{-n}),
\end{align*}
\begin{align*}
    \cP_n^{\gamma}:=\{[i_02^{-n\gamma},(i_0+1)2^{-n\gamma})\times D_n(i_1,\cdots,i_d):i_0,\cdots,i_d\in\bZ\},
\end{align*}
and
\begin{align*}
    d_{\gamma}((t,x),(s,y)):=|t-s|^{\frac{1}{\gamma}}+|x-y|,\quad B_b^{\gamma}(t,x):=\{(s,y):d_{\gamma}((t,x),(s,y))<b\}.
\end{align*}
Then
\begin{enumerate}[(i)]
    \item $\{\cP_n^{\gamma}\}_{n\in\bZ}$ is a filtration of partitions.
    \item $d_{\gamma}$ is a quasi-metric on $\bR^{d+1}$.
    \item $B_{b}^{\gamma}$ satisfies the doubling property. More precisely,
    $$
    |B_{2b}^{\gamma}(t,x)|=2^{\gamma+d}|B_{b}^{\gamma}(t,x)|.
    $$
    \item for each $A\in\cP_n^{\gamma}$, there exist constants $\delta=\delta(\gamma)\in(1/2,1)$ and $N=N(d)$ such that $A\subseteq B_{N\delta^n}^{\gamma}(t,x)$ for all $(t,x)\in A$.
    
\end{enumerate}
\end{lem}
\begin{proof}
($i$) It can be easily checked that $\{\cP_n^{\gamma}\}_{n\in\bZ}$ satisfies Definition \ref{filtration}.

($ii$) For $\gamma>0$, there exists a constant $N=N(\gamma)\geq 1$ such that
$$
|t-r|^{\frac{1}{\gamma}}\leq N(|t-s|^{\frac{1}{\gamma}}+|s-r|^{\frac{1}{\gamma}}),\quad \forall t,s,r\in\bR,
$$
and this certainly implies $d_{\gamma}$ is a quasi-metric on $\bR^{d+1}$.

($iii$) Using the change of variable formula,
\begin{align*}
 |B_{2b}^{\gamma}(t,x)|&=\int_{\bR^{d+1}}1_{B_{2b}^{\gamma}(t,x)}(s,y)\mathrm{d}y\mathrm{d}s=\int_{\bR^{d+1}}1_{B_{b}^{\gamma}(t/2^{\gamma},x/2)}(s/2^{\gamma},y/2)\mathrm{d}y\mathrm{d}s\\
 &=2^{\gamma+d}\int_{\bR^{d+1}}1_{B_{b}^{\gamma}(t/2^{\gamma},x/2)}(s,y)\mathrm{d}y\mathrm{d}s=2^{\gamma+d}|B_{b}^{\gamma}(t/2^{\gamma},x/2)| =2^{\gamma+d}|B_{b}^{\gamma}(t,x)|.
\end{align*}

($iv$) Let $\delta:=(2^{-\gamma}\vee2^{-1})\in(1/2,1)$. Note that for $A\in\cP_n^{\gamma}$,
$$
diam(A)=\sqrt{4^{-n\gamma}+d4^{-n}}< N(d)\delta^{n},
$$
and this certainly implies the result. The lemma is proved.
\end{proof}

For $(t,x)\in\bR\times\bR^d$ and $b,\gamma>0$, denote
$$
Q_b^{\gamma}(t,x):=(t-b^{\gamma},t+b^{\gamma})\times \{ y \in \bR^d : |x-y| < b\}.
$$
\begin{thm}
\label{20.12.21.16.37}
Let $p,q\in(1,\infty)$, $\gamma\in(0,\infty)$, $w_1\in A_q(\bR)$, $w_2\in A_p(\bR^d)$, and $w\in A_p(\bR^{d+1})$. Suppose that
$$
[w_1]_{A_q(\bR)}\leq K_0,\quad [w_2]_{A_p(\bR^d)}\leq K_1,\quad [w]_{A_p(\bR^{d+1})}\leq K_2.
$$
Then
\begin{enumerate}[(i)]
    \item for any $f\in L_q(\bR,w_1;L_p(\bR^d,w_2))$ (resp. $f\in L_p(\bR^{d+1},w)$),
    \begin{equation*}
        \|f\|_{L_q(\bR,w_1;L_p(\bR^d,w_2))}\leq N \|f^{\sharp}\|_{L_q(\bR,w_1;L_p(\bR^d,w_2))}\quad (resp.\, \|f\|_{L_p(\bR^{d+1},w)}\leq N \|f^{\sharp}\|_{L_p(\bR^{d+1},w)}),
    \end{equation*}
    where $N=N(d,\gamma,p,q,K_0,K_1)$ (resp. $N=N(d,\gamma,p,K_2$))  and 
    $$
    f^{\sharp}(t,x):=\sup_{(t,x)\in Q_b^{\gamma}}\frac{1}{|Q_b^{\gamma}|^2}\int_{Q_b^{\gamma}}\int_{Q_b^{\gamma}}|f(s_1,y_1)-f(s_0,y_0)|\mathrm{d}y_1\mathrm{d}s_1\mathrm{d}y_0\mathrm{d}s_0.
    $$
    \item for any $f\in L_q(\bR,w_1;L_p(\bR^d,w_2))$ (resp. $f\in L_p(\bR^{d+1},w)$),
    \begin{equation*}
        \|\bM f\|_{L_q(w',\bR;L_p(w,\bR^d))}\leq  N\|f\|_{L_q(w',\bR;L_p(w,\bR^d))}\quad (resp.\, \|\bM f\|_{L_p(w'',\bR^{d+1})}\leq N \|f\|_{L_p(w'',\bR^{d+1})}),
    \end{equation*}
    where $N=N(d,\gamma,p,q,K_0,K_1)$ (resp. $N=N(d,\gamma,p,K_2$)) and 
    \begin{equation*}
        \begin{gathered}
        \bM f(t,x):=\sup_{(t,x)\in Q^{\gamma}_b}\frac{1}{|Q^{\gamma}_b|}\int_{Q^{\gamma}_b}|f(s,y)|\mathrm{d}y\mathrm{d}s.
        \end{gathered}
    \end{equation*}
\end{enumerate}
\end{thm}
\begin{proof}
By Lemma \ref{20.12.22.19.10}, the filtration $\{\cP_n^{\gamma}\}_{n\in\bZ}$ satisfies \cite[Theorem 2.1]{dong18Apweights}. Therefore, by \cite[Section 2]{dong18Apweights}, we obtain the results if we prove that there exists a constant $N$ such that
$$
f^{\#}_{dy}\leq Nf^{\sharp},\quad \bM f\leq N\cM f,
$$
where
$$
f^{\#}_{dy}(t,x):=\sup_{n\in\bZ}\aint_{A_n(t,x)}|f(s,y)-f_{|n}(t,x)|\mathrm{d}y\mathrm{d}s,\quad \cM f(t,x):=\sup_{(t,x)\in B_b^{\gamma}}\aint_{B_b^{\gamma}}|f(s,y)|\mathrm{d}y\mathrm{d}s.
$$

Since $(t,x)\in A_n(t,x)$, by Lemma \ref{20.12.22.19.10} ($iv$) and the definition of $Q_b^{\gamma}$, 
$$
A_n(t,x)\subseteq B_{N\delta^n}^{\gamma}(t,x)\subseteq Q_{N\delta^n}^{\gamma}(t,x).
$$
One can observe that
$$
|A_n(t,x)|=2^{-n(\gamma+d)},\quad |Q_{N\delta^n}(t,x)|=N_1(d,\gamma)\delta^{n(\gamma+d)}.
$$
Since $\delta=(2^{-\gamma}\vee 2^{-1})$, we have $1\leq 2\delta \leq 1+2^{1-\gamma}$. Thus, $|Q_{N\delta^n}(t,x)|\leq N_2(d,\gamma)|A_n(t,x)|$,
and this implies
\begin{align*}
f^{\#}_{dy}(t,x)\leq N(d,\gamma)f^{\sharp}(t,x).
\end{align*}

Note that $Q_{b/2}^{\gamma}\subseteq B_b^{\gamma}\subseteq Q_b^{\gamma}$, and this implies
$$
\bM f(t,x)\leq N(d,\gamma)\cM f(t,x).
$$
The theorem is proved.
\end{proof}

\begin{lem}
\label{20.11.16.14.53}
Let $f\in C_c(\bR^d)$ and $g$ be a complex-valued continuously differentiable function on $\bR^d$ such that 
$$\lim_{|z|\to\infty}|g(z)|=0.
$$
Suppose that $x,y\in\bR^d$, $|x-y|\leq R_1$ and $f(y-z)=0$ for $|z|\leq R_2$. Then for any $p_0\in(1,\infty)$ and its H\"older conjugate $p_0'$, \textit{i.e.}, $1/p_0+1/p_0'=1$,
$$
|f\ast g(y)|\leq N\int_{R_2}^{\infty}\left(r^{d}\int_{\bS^{d-1}}|\nabla g(r \omega)\cdot \omega|^{p_0'}\sigma(\mathrm{d}\omega)\right)^{1/p_0'}\left(\int_{B_{R_1+r}(x)}|f(z)|^{p_0}\mathrm{d}z\right)^{1/p_0}\mathrm{d}r,
$$
where $N=N(d,p_0)$ and $\bS^{d-1}$ denotes the $d-1$-dimensional unit sphere in $\bR^d$. 
\end{lem}

\begin{proof}
By polar coordinates and Fubini's theorem,
\begin{align*}
    f\ast g(y)&=\int_{|z|>R_2}f(y-z)g(z)\mathrm{d}z=\int_{R_2}^{\infty}\int_{\bS^{d-1}}f(y-rw)g(rw)r^{d-1}\sigma(dw)\mathrm{d}r\\
    &=\int_{\bS^{d-1}}\int_{R_2}^{\infty}g(rw)\left(\frac{d}{\mathrm{d}r}\int_{R_2}^{r}f(y-\rho w)\rho^{d-1}\mathrm{d}\rho\right)\mathrm{d}r\sigma(dw)
\end{align*}
Due to the integration by parts,
\begin{equation*}
    \int_{R_2}^{\infty}g(rw)\left(\frac{d}{\mathrm{d}r}\int_{R_2}^{r}f(y-\rho w)\rho^{d-1}\mathrm{d}\rho\right)\mathrm{d}r=-\int_{R_2}^{\infty} \nabla g(rw)\cdot w \left(\int_{R_2}^rf(y-\rho w)\rho^{d-1}\mathrm{d}\rho\right) \mathrm{d}r.
\end{equation*}
By Fubini's theorem and H\"older's inequality,
\begin{align*}
    |f\ast g(y)|&\leq \int_{R_2}^{\infty}\int_{\bS^{d-1}}\int_{R_2}^r| \nabla g(rw)\cdot w||f(y-\rho w)|\rho^{d-1}\mathrm{d}\rho \sigma(dw)\mathrm{d}r\leq \int_{R_2}^{\infty}I_1(r) I_2(y,r)\mathrm{d}r
\end{align*}
where
\begin{equation*}
    \begin{gathered}
    I_1(r):=\left(\int_{R_2}^r G_{p_0}(r)\rho^{d-1}\mathrm{d}\rho \right)^{1/p_0'},\\ I_2(y,r):=\left(\int_{\bS^{d-1}}\int_{R_2}^r|f(y-\rho w)|^{p_0}\rho^{d-1}\mathrm{d}\rho \sigma(dw)\right)^{1/p_0}.
    \end{gathered}
\end{equation*}
Obviously,
$$
I_1(r)\leq N(d,p_0)\left(r^{d}\int_{\bS^{d-1}}|\nabla g(r \omega)\cdot \omega|^{p_0'}\sigma(\mathrm{d}\omega)\right)^{1/p_0'}
$$
and
\begin{align*}
I_2(y,r)^{p_0}&=\int_{|z|\leq r}|f(y-z)|^{p_0}\mathrm{d}z= \int_{|y-z|\leq r}|f(z)|^{p_0}\mathrm{d}z\leq \int_{|x-z|\leq R_1+r}|f(z)|^{p_0}\mathrm{d}z=\int_{B_{R_1+r}(x)}|f(z)|^{p_0}\mathrm{d}z.
\end{align*}
The lemma is proved.
\end{proof}

\begin{lem}
\label{20.12.17.20.21}
Let $p_0\in(1,2]$ and $f\in C_c(\bR^d)$.
Suppose that $K_{\varepsilon}$ satisfies the $\left(\left\lfloor\frac{d}{p_0}\right\rfloor+2\right)$-times regular condition with $(\gamma,N_1,N_2)$ (Definition \ref{regular condition}).
Then for any  $(m,|\alpha|)\in\{0,1\}\times\{0,1,2\}$, there exists a positive constant $N=N(d,p_0,\varepsilon,N_1,N_2)$ such that
\begin{enumerate}[(i)]
    \item for all $a_0,a_1\in(0,\infty)$,
\begin{equation*}
    \begin{aligned}
    &\int_{a_1}^{\infty}\left(\int_{B_{a_0+\lambda}(x)}|f(z)|^{p_0}\mathrm{d}z\right)^{1/p_0}\left(\lambda^d \int_{\bS^{d-1}}|\p_t^mD_x^{\alpha} K_{\varepsilon}(t,s,\lambda \omega)|^{p_0'}\sigma(\mathrm{d}\omega) \right)^{1/p_0'}\mathrm{d}\lambda\\
    &\leq N\left(\int_{a_1}^{\infty}\lambda^{-p_0d(p_0)-1}\left(\int_{B_{a_0+\lambda}(x)}|f(z)|^{p_0}\mathrm{d}z\right)\mathrm{d}\lambda \right)^{1/p_0}\times|t-s|^{-(m+\varepsilon)-\frac{(d+|\alpha|-d(p_0)-1)}{\gamma}+\frac{d}{p_0'\gamma}},
    \end{aligned}
\end{equation*}

\item for any $a_0\in(0,\infty)$,
\begin{equation}
\label{20.12.27.16.23}
    \begin{aligned}
    &\left(\int_{0}^{\infty}\left(\int_{B_{a_0+\lambda}(x)}|f(z)|^{p_0}\mathrm{d}z\right)^{1/p_0} \left(\lambda^d \int_{\bS^{d-1}}|\p_t^mD_x^{\alpha} K_{\varepsilon}(t,s,\lambda \omega)|^{p_0'}\sigma(\mathrm{d}\omega) \right)^{1/p_0'}\mathrm{d}\lambda \right)^{p_0}\\
    &\leq N\Bigg(\int_{0}^{\infty}|t-s|^{-\frac{d}{\gamma}}\left(|t-s|^{-\frac{1}{\gamma}}1_{\lambda\leq |t-s|^{\frac{1}{\gamma}}}+|t-s|^{\frac{p_0d(p_0)}{\gamma}}\lambda^{-p_0d(p_0)-1}1_{|t-s|^{\frac{1}{\gamma}}\leq \lambda}\right)\\
    &\qquad \qquad \times\left(\int_{B_{a_0+\lambda}(x)}|f(z)|^{p_0}\mathrm{d}z\right)\mathrm{d}\lambda\Bigg) \times |t-s|^{-p_0(m+\varepsilon)-p_0\frac{(|\alpha|-1)}{\gamma}},
    \end{aligned}
\end{equation}

\item for all $|t-s|\geq a^{\gamma}>0$ and $c_0>0$,
\begin{equation}
\label{20.12.27.17.17}
\begin{aligned}
    \int_{0}^{\infty}(c_0a+\lambda)^d |t-s|^{-\frac{d}{\gamma}}\left(|t-s|^{-\frac{1}{\gamma}}1_{\lambda\leq |t-s|^{\frac{1}{\gamma}}}+|t-s|^{\frac{p_0d(p_0)}{\gamma}}\lambda^{-p_0d(p_0)-1}1_{|t-s|^{\frac{1}{\gamma}}\leq \lambda}\right)\mathrm{d}\lambda \leq N,
\end{aligned}    
\end{equation}
\end{enumerate}
where $N=N(d,p_0,c_0)$ and
$$
d(p_0) :=\left\lfloor\frac{d}{p_0}\right\rfloor+1
$$
and $p_0'$ is the H\"older conjugate of $p_0$, \textit{i.e.}, $1/p_0+1/p_0'=1$.

\end{lem}
\begin{proof} 
For simpler notations, we denote
$$
\mu:=d(p_0)+\frac{1}{p_0}=\left\lfloor\frac{d}{p_0}\right\rfloor+1+\frac{1}{p_0}>\frac{d+1}{p_0},
$$
\begin{equation*}
    H(r,\lambda):=r^{-\frac{d}{\gamma}}\left(r^{-\frac{1}{\gamma}}1_{\lambda\leq r^{\frac{1}{\gamma}}}+r^{\frac{p_0d(p_0)}{\gamma}}\lambda^{-p_0d(p_0)-1}1_{r^{\frac{1}{\gamma}}\leq \lambda}\right),
\end{equation*}
and
$$
K_{\varepsilon,m,\alpha}^{p_0'}(t,s,\lambda):=\int_{\bS^{d-1}}|\p_t^mD_x^{\alpha} K_{\varepsilon}(t,s,\lambda w)|^{p_0'}\sigma(\mathrm{d}\omega).
$$
($i$) By H\"older's inequality and Lemma \ref{20.12.24.12.47},
\begin{align*}
&\int_{a_1}^{\infty}\left(\int_{B_{a_0+\lambda}(x)}|f(z)|^{p_0}\mathrm{d}z\right)^{1/p_0}(\lambda^dK_{\varepsilon,m,\alpha}^{p_0'}(t,s,\lambda))^{1/p_0'}\mathrm{d}\lambda\\
&\leq\left(\int_{a_1}^{\infty}\lambda^{-p_0\mu}\int_{B_{a_0+\lambda}(x)}|f(z)|^{p_0}\mathrm{d}z\mathrm{d}\lambda \right)^{1/p_0}\left(\int_{a_1}^{\infty}\lambda^{d+p_0'\mu}K_{\varepsilon,m,\alpha}^{p_0'}(t,s,\lambda)\mathrm{d}\lambda \right)^{1/p_0'}\\
&\leq\left(\int_{a_1}^{\infty}\lambda^{-p_0\mu}\int_{B_{a_0+\lambda}(x)}|f(z)|^{p_0}\mathrm{d}z\mathrm{d}\lambda \right)^{1/p_0}\left(\int_{\bR^d}|z|^{1+p_0'\mu}|\p_t^mD_x^{\alpha}K_{\varepsilon}(t,s,z)|^{p_0'} \mathrm{d}z\right)^{1/p_0'}\\
&\leq N\left(\int_{a_1}^{\infty}\lambda^{-p_0d(p_0)-1}\int_{B_{a_0+\lambda}(x)}|f(z)|^{p_0}\mathrm{d}z\mathrm{d}\lambda \right)^{1/p_0}|t-s|^{-(m+\varepsilon)-\frac{(d+|\alpha|-d(p_0)-1)}{\gamma}+\frac{d}{p_0'\gamma}},
\end{align*}
where $N=N(d,p_0,N_1,N_2)$.

($ii$) Decompose the integrals in \eqref{20.12.27.16.23} into two parts,
\begin{align*}
  &\int_{0}^{\infty}\left(\int_{B_{c_0a+\lambda}(x)}|f(z)|^{p_0}\mathrm{d}z\right)^{1/p_0} (\lambda^dK_{\varepsilon,m,\alpha}^{p_0'}(t,s,\lambda))^{1/p_0'}\mathrm{d}\lambda\\
  &=\int_{0}^{|t-s|^{1/\gamma}}\cdots \mathrm{d}\lambda+\int_{|t-s|^{1/\gamma}}^{\infty}\cdots \mathrm{d}\lambda=:I_{1}(t,s)+I_{2}(t,s).
\end{align*}
By H\"older's inequality and Lemma \ref{20.12.24.12.47},
\begin{align*}
&I_1(t,s)\leq\left(\int_{0}^{|t-s|^{1/\gamma}}\int_{B_{a_0+\lambda}(x)}|f(z)|^{p_0}\mathrm{d}z\mathrm{d}\lambda \right)^{1/p_0}\left(\int_{0}^{|t-s|^{1/\gamma}}\lambda^{d}K_{\varepsilon,m,\alpha}^{p_0'}(t,s,\lambda)\mathrm{d}\lambda \right)^{1/p_0'}\\
&\leq\left(\int_{0}^{|t-s|^{1/\gamma}}\int_{B_{a_0+\lambda}(x)}|f(z)|^{p_0}\mathrm{d}z\mathrm{d}\lambda \right)^{1/p_0}\left(\int_{\bR^d}|z||\p_t^mD_x^{\alpha}K_{\varepsilon}(t,s,z)|^{p_0'} \mathrm{d}z\right)^{1/p_0'}\\
&\leq N\left(\int_{0}^{|t-s|^{1/\gamma}}\int_{B_{a_0+\lambda}(x)}|f(z)|^{p_0}\mathrm{d}z\mathrm{d}\lambda \right)^{1/p_0}|t-s|^{-(m+\varepsilon)-\frac{(d+|\alpha|-1/p_0')}{\gamma}+\frac{d}{p_0'\gamma}}\\
&=N\left(\int_{0}^{|t-s|^{1/\gamma}}|t-s|^{-\frac{(d+1)}{\gamma}}\int_{B_{a_0+\lambda}(x)}|f(z)|^{p_0}\mathrm{d}z\mathrm{d}\lambda \right)^{1/p_0}|t-s|^{-(m+\varepsilon)-\frac{(|\alpha|-1)}{\gamma}}
\end{align*}
where $N=N(d,p_0,N_1,N_2)$. For $I_2$, by $(i)$ with $a_1=|t-s|^{1/\gamma}$,
\begin{align*}
I_2(t,s)\leq N\left(\int_{|t-s|^{1/\gamma}}^{\infty}|t-s|^{\frac{p_0d(p_0)-d}{\gamma}}\lambda^{-p_0d(p_0)-1}\int_{B_{a_0+\lambda}(x)}|f(z)|^{p_0}\mathrm{d}z\mathrm{d}\lambda \right)^{1/p_0}|t-s|^{-(m+\varepsilon)-\frac{(|\alpha|-1)}{\gamma}},    
\end{align*}
where $N=N(d,p_0,c_0,N_1,N_2)$.

($iii$) Decompose the integral in \eqref{20.12.27.17.17} into two parts,
\begin{align*}
  &\int_{0}^{\infty}(c_0a+\lambda)^d H(|t-s|,\lambda)\mathrm{d}\lambda=\int_{0}^{|t-s|^{1/\gamma}}\cdots \mathrm{d}\lambda+\int_{|t-s|^{1/\gamma}}^{\infty}\cdots \mathrm{d}\lambda=:I_{3}(t,s)+I_{4}(t,s).
\end{align*}
By the definition of $H$,
$$
I_3(t,s)=\int_0^{|t-s|^{1/\gamma}}(c_0a+\lambda)^d|t-s|^{-\frac{(d+1)}{\gamma}}\mathrm{d}\lambda\leq N(d,c_0)\int_0^{|t-s|^{1/\gamma}}|t-s|^{-\frac{1}{\gamma}}\mathrm{d}\lambda=N(d,c_0),
$$
and
\begin{align*}
    I_4(t,s)&=|t-s|^{-\frac{(p_0d(p_0)-d)}{\gamma}}\int_{|t-s|^{1/\gamma}}^{\infty}(c_0a+\lambda)^\mathrm{d}\lambda^{-p_0d(p_0)-1}\mathrm{d}\lambda\\
    &\leq N(d,c_0)|t-s|^{-\frac{(p_0d(p_0)-d)}{\gamma}}\int_{|t-s|^{1/\gamma}}^{\infty}\lambda^{d-p_0d(p_0)-1}\mathrm{d}\lambda=N(d,p_0,c_0).
\end{align*}
The lemma is proved.
\end{proof}

\begin{lem}
\label{20.12.28.16.02}
Let  $p_0\in(1,2]$, $t_0\in\bR$, $b>0$, $s\in (t_0-b^{\gamma},t_0+b^{\gamma})$, and $f \in C_c^\infty(\bR^{d+1})$.
Then for any $(t,x)\in Q_{b}^{\gamma}(t_0,0)$,
\begin{align*}
    &\int_{-\infty}^{t_0-2b^{\gamma}}\left(\int_{0}^{\infty}H(|s-r|,\lambda)\int_{B_{2b+\lambda}(x)}|f(r,z)|^{p_0}\mathrm{d}z\mathrm{d}\lambda \right)^{1/p_0}h_0(m,|\alpha|,\varepsilon,s-r,T) \mathrm{d}r\\
    & \leq N(T^{1-\varepsilon}h_1(0,1,\varepsilon)+T^{1-\varepsilon}b^{-1}h_1(0,2,\varepsilon)+b^{-\gamma}h_1(1,1,1))(\bM |f|^{p_0}(t,x))^{1/p_0},
\end{align*}
where $N=N(d,p_0,\varepsilon,N_1,N_2)$ and
\begin{equation*}
    \begin{gathered}
    h_0(m,|\alpha|,\varepsilon,s-r,T):=|s-r|^{-(m+\varepsilon)-\frac{(|\alpha|-1)}{\gamma}}h_1(m,|\alpha|,\varepsilon)(1_{\varepsilon\in[0,1)}1_{|s-r|\leq T}+1_{\varepsilon=1}),\\
    h_1(m,|\alpha|,\varepsilon):=1_{(m,|\alpha|)=(0,1)}1_{\varepsilon\in[0,1)}+1_{(m,|\alpha|)=(0,2)}+1_{(m,|\alpha|)=(1,1)}1_{\varepsilon=1},\\
    H(r,\lambda):=r^{-\frac{d}{\gamma}}\left(r^{-\frac{1}{\gamma}}1_{\lambda\leq r^{\frac{1}{\gamma}}}+r^{\frac{p_0d(p_0)}{\gamma}}\lambda^{-p_0d(p_0)-1}1_{r^{\frac{1}{\gamma}}\leq \lambda}\right).
    \end{gathered}
\end{equation*}
\end{lem}
\begin{proof}
Let $(t,x)\in Q_{b}^{\gamma}(t_0,0)$. By H\"older's inequality,
\begin{equation}\label{2021040302}
\begin{aligned}
    &\int_{-\infty}^{t_0-2b^{\gamma}}\left(\int_{0}^{\infty}H(|s-r|,\lambda)\int_{B_{2b+\lambda}(x)}|f(r,z)|^{p_0}\mathrm{d}z\mathrm{d}\lambda \right)^{1/p_0}h_0(m,|\alpha|,\varepsilon,s-r,T) \mathrm{d}r\\
    &\leq \left(\int_{-\infty}^{t_0-2b^{\gamma}}\int_{0}^{\infty}H(|s-r|,\lambda)h_0(m,|\alpha|,\varepsilon,s-r,T)\int_{B_{2b+\lambda}(x)}|f(r,z)|^{p_0}\mathrm{d}z\mathrm{d}\lambda \mathrm{d}r\right)^{1/p_0}\\
    &\qquad\times\left(\int_{-\infty}^{t_0-2b^{\gamma}}h(m,|\alpha|,\varepsilon,s-r,T)\mathrm{d}r\right)^{1/p_0'}.
\end{aligned}    
\end{equation}

Due to the integration by parts formula,
\begin{equation}\label{2021040301}
    \begin{aligned}
    &\int_{-\infty}^{t_0-2b^{\gamma}}H(|s-r|,\lambda)h_0(m,|\alpha|,\varepsilon,s-r,T)\int_{B_{2b+\lambda}(x)}|f(r,z)|^{p_0}\mathrm{d}z \mathrm{d}r\\
    &\leq \int_{-\infty}^{t_0-2b^{\gamma}}\frac{\p}{\p r}\left(H(|s-r|,\lambda)h_0(m,|\alpha|,\varepsilon,s-r,T)\right)\left(\int_{r}^{t_0-2b^{\gamma}}\int_{B_{2b+\lambda}(x)}|f(\varphi,z)|^{p_0}\mathrm{d}z\mathrm{d}\varphi\right)\mathrm{d}r\\
    &\leq \bM|f|^{p_0}(t,x)(2b+\lambda)^d\times\left(\int_{-\infty}^{t_0-2b^{\gamma}}\frac{\p}{\p r}\left(H(|s-r|,\lambda)h_0(m,|\alpha|,\varepsilon,s-r,T)\right)|t_0+b^{\gamma}-r|\mathrm{d}r\right).
\end{aligned}
\end{equation}

Note that if $r\in(-\infty,t_0-2b^{\gamma})$ and $s\in(t_0-b^{\gamma},t_0+b^{\gamma})$, then
$$
1<\frac{t_0+b^{\gamma}-r}{s-r}<3.
$$
Also, observe that
$$
\frac{\p}{\p r}\left(H(|s-r|,\lambda)h_0(m,|\alpha|,\varepsilon,s-r,T)\right)=N|s-r|^{-1}H(|s-r|,\lambda)h_0(m,|\alpha|,\varepsilon,s-r,T),
$$
where $N=N(d,p_0,\gamma,\varepsilon)$. Thus from \eqref{2021040301}
\begin{align}
								\notag
   &\int_{-\infty}^{t_0-2b^{\gamma}} \int_0^\infty H(|s-r|,\lambda)h_0(m,|\alpha|,\varepsilon,s-r,T)\int_{B_{2b+\lambda}(x)}|f(r,z)|^{p_0}\mathrm{d}z\mathrm{d}\lambda \mathrm{d}r\\
								\notag
   &\leq N\bM|f|^{p_0}(t,x)\int_{-\infty}^{t_0-2b^{\gamma}}h_0(m,|\alpha|,\varepsilon,s-r,T)\left(\int_{0}^{\infty}(2b+\lambda)^dH(|s-r|,\lambda)\mathrm{d}\lambda\right)\mathrm{d}r\\
								\label{2021040303}
   &\leq N\bM|f|^{p_0}(t,x)\int_{-\infty}^{t_0-2b^{\gamma}}h_0(m,|\alpha|,\varepsilon,s-r,T)\mathrm{d}r,
\end{align}
where $N=N(d,p_0,\gamma,\varepsilon,N_1,N_2)$. Therefore, combining \eqref{2021040302} and \eqref{2021040303}, we have
\begin{align*}
    &\int_{-\infty}^{t_0-2b^{\gamma}}\left(\int_{0}^{\infty}H(|s-r|,\lambda)\int_{B_{2b+\lambda}(x)}|f(r,z)|^{p_0}\mathrm{d}z\mathrm{d}\lambda \right)^{1/p_0}h_0(m,|\alpha|,\varepsilon,s-r,T) \mathrm{d}r\\
    &\leq N(\bM|f|^{p_0}(t,x))^{1/p_0}\int_{-\infty}^{t_0-2b^{\gamma}}h_0(m,|\alpha|,\varepsilon,s-r,T)\mathrm{d}r,
\end{align*}
where $N=N(d,p_0,\gamma,\varepsilon,N_1,N_2)$. 
It only remains to observe 
\begin{align*}
\int_{-\infty}^{t_0-2b^{\gamma}}h_0(m,|\alpha|,\varepsilon,s-r,T)\mathrm{d}r&\leq N(T^{1-\varepsilon}h_1(0,1,\varepsilon)+T^{1-\varepsilon}b^{-1}h_1(0,2,\varepsilon)+b^{-\gamma}h_1(1,1,1))
\end{align*}
where $N=N(d,\gamma,\varepsilon)$. The lemma is proved.
\end{proof}

\subsection{$L_p$ estimates without weights for $p \in (1,2)$}

Recall
    $$
    h(\varepsilon,t-s,T):=1_{\varepsilon\in[0,1)}1_{|t-s|<T}+1_{\varepsilon=1}
    $$
and
    $$
    \cT_{\varepsilon,T} f(t,x):=\int_{-\infty}^{t}\int_{\bR^d}h(\varepsilon,t-s,T)K_{\varepsilon}(t,s,x-y)f(s,y)\mathrm{d}y\mathrm{d}s.
    $$

In this subsection, we prove the $L_p$-boundeness of $\cT_{\varepsilon,T}$ without weights. 
More precisely, we show the following theorem:
\begin{thm}
						\label{20.12.21.13.23}
Suppose that $K_{\varepsilon}$ satisfies the $\left(\left\lfloor\frac{d}{2}\right\rfloor+1\right)$-times regular condition with $(\gamma,N_1,N_2)$ (Definition \ref{regular condition}) and  there exists a constant $N_3>0$ such that
    $$
    \|\cT_{\varepsilon,T} f\|_{L_2(\bR^{d+1})}\leq N_3T^{1-\varepsilon}\|f\|_{L_2(\bR^{d+1})}, \quad \forall f \in C_c^\infty(\bR^{d+1}).
    $$
Then for $p\in(1,2)$, there exists a constant $N=N(d,p,\gamma,\varepsilon,N_1,N_2,N_3)$ such that
$$
\|\cT_{\varepsilon,T}f\|_{L_p(\bR^{d+1})}\leq NT^{1-\varepsilon}\|f\|_{L_p(\bR^{d+1})}, \quad \forall f \in C_c^\infty(\bR^{d+1}).
$$
\end{thm}
The proof of the theorem will be given in the last of this section.
\bigskip

We first show that our kernel satisfies the important condition so-called H\"ormander's condition to guarantee the  $L_p$-boundedness.
Recall 
\begin{align*}
    \cP_n^{\gamma}:=\{[i_02^{-n\gamma},(i_0+1)2^{-n\gamma})\times D_n(i_1,\cdots,i_d):i_0,\cdots,i_d\in\bZ\},
\end{align*}
and $\{\cP_n^{\gamma}\}_{n\in\bZ}$ is a filtration of partitions by Lemma \ref{20.12.22.19.10}.
\begin{lem}[H\"ormander's condition for the kernel $K_\varepsilon$]
		\label{20.12.21.13.02}
Suppose that $K_\varepsilon$ satisfies the $\left(\left\lfloor\frac{d}{2}\right\rfloor+1\right)$-times regular condition $(\gamma,N_1,N_2)$ (Definition \ref{regular condition}).
Then for each $A\in \cup_{n\in\bZ}\cP_n^{\gamma}$, there exist constant $N=N(d,\gamma,\varepsilon,N_1,N_2)$ and closed set $A^*$ such that $\overline{A}\subset A^*$, $|A|\leq N|A^*|$ and
\begin{align}
								\label{2021040310}
\int_{\bR^{d+1}\setminus A^*}|K_{\varepsilon}(t,s_0,x-y_0)h(\varepsilon,t-s_0,T)-K_{\varepsilon}(t,s_1,x-y_1)h(\varepsilon,t-s_1,T)|\mathrm{d}x\mathrm{d}t
\leq NT^{1-\varepsilon}
\end{align}
whenever $(s_0,y_0),(s_1,y_1)\in A$.
\end{lem}
\begin{proof}
Let
\begin{equation*}
    \begin{gathered}
    A=[i_02^{-n\gamma},(i_0+1)2^{-n\gamma})\times[0,2^{-n})^d\in \cP_n^{\gamma}
    \end{gathered}
\end{equation*}
and put
$$
    A^*=[i_02^{-n\gamma}-2^{-n\gamma},i_02^{-n\gamma}+2^{-n\gamma+1}]\times \overline{B_{2^{-(n-1)}\sqrt{d}}(0)}.
$$
We show that \eqref{2021040310} holds with a positive constant $N$. 
Denote
\begin{equation*}
    \begin{gathered}
    I(s_0,y_0,s_1,y_1):=\int_{\bR^{d+1}\setminus A^*}|K_{\varepsilon}(t,s_0,x-y_0)h(\varepsilon,t-s_0,T)-K_{\varepsilon}(t,s_1,x-y_1)h(\varepsilon,t-s_1,T)|\mathrm{d}x\mathrm{d}t,\\
     A_1:=(i_02^{-n\gamma}-2^{-n\gamma},i_02^{-n\gamma}+2^{-n\gamma+1})\times (\bR^{d}\setminus \overline{B_{2^{-(n-1)}\sqrt{d}}(0)}),\\
     A_2:=(-\infty,i_02^{-n\gamma}-2^{-n\gamma})\times \bR^d,\quad A_3:= (i_02^{-n\gamma}+2^{-n\gamma+1},\infty)\times \bR^d.
    \end{gathered}
\end{equation*}
Set
\begin{align*}
I(s_0,y_0,s_1,y_1)&= \left(\int_{A_1}\cdots \mathrm{d}x \mathrm{d}t\right)+\left(\int_{A_2}\cdots \mathrm{d}x\mathrm{d}t\right)+\left(\int_{A_3}\cdots \mathrm{d}x\mathrm{d}t\right)\\
&=:I_1(s_0,y_0,s_1,y_1)+I_2(s_0,y_0,s_1,y_1)+I_3(s_0,y_0,s_1,y_1).
\end{align*}
and with the triangle inequality ($i=1,2,3$),
\begin{align*}
    I_i(s_0,y_0,s_1,y_1)&\leq \int_{A_i}|K_{\varepsilon}(t,s_0,x-y_0)h(\varepsilon,t-s_0,T)-K_{\varepsilon}(t,s_0,x-y_1)h(\varepsilon,t-s_0,T)|\mathrm{d}x\mathrm{d}t\\
    &\quad+\int_{A_i}|K_{\varepsilon}(t,s_0,x-y_1)h(\varepsilon,t-s_0,T)-K_{\varepsilon}(t,s_1,x-y_1)h(\varepsilon,t-s_1,T)|\mathrm{d}x\mathrm{d}t\\
    &=:I_{i,1}(s_0,y_0,y_1)+I_{i,2}(s_0,s_1,y_1).
\end{align*}
Choose a sequence of nonnegative functions  $\eta_m\in C_c^{\infty}(\bR^d)$ so that $\eta_m\uparrow 1$ as $m\to\infty$ and $\eta_m(-x)=\eta_m(x)$ for all $x\in\bR^d$.
First, we prove that $I_1\leq NT^{1-\varepsilon}$.
If $x\in \bR^{d}\setminus \overline{B_{2^{-(n-1)}\sqrt{d}}(0)}$ and $y\in [0,2^{-n})^d$, then $|x-y|>2^{-n}\sqrt{d}$. By H\"older's inequality and Definition \ref{regular condition},
\begin{align*}
    \int_{|x|\geq 2^{-(n-1)}\sqrt{d}}|K_{\varepsilon}(t,s,x-y)|\mathrm{d}x&\leq \int_{|x|\geq 2^{-n}\sqrt{d}}|K_{\varepsilon}(t,s,x)|\mathrm{d}x\\
    &\leq \left(\int_{|x|\geq 2^{-n}\sqrt{d}}|x|^{-2d(2)}\mathrm{d}x\right)^{1/2}\left(\int_{|x|\geq 2^{-n}\sqrt{d}}|x|^{2d(2)}|K_{\varepsilon}(t,s,x)|^2\mathrm{d}x\right)^{1/2}\\
    &\leq N(2^{-n}\sqrt{d})^{-d(2)+\frac{d}{2}}|t-s|^{-\varepsilon-\frac{(d-d(2))}{\gamma}+\frac{d}{2\gamma}},
\end{align*}
where $N=N(d,N_2)$. Thus, since $s\in [i_02^{-n\gamma},(i_0+1)2^{-n\gamma})$,
\begin{align*}
    I_1(s_0,y_0,s_1,y_1)&\leq N2^{nd(2)-\frac{nd}{2}}\int_{i_02^{-n\gamma}-2^{-n\gamma}}^{i_02^{-n\gamma}+2^{-n\gamma+1}}|t-s|^{^{-\varepsilon-\frac{(d-d(2))}{\gamma}+\frac{d}{2\gamma}}}(1_{0<t-s<T}1_{\varepsilon\in[0,1)}+1_{\varepsilon=1})\mathrm{d}t\\
    &\leq NT^{1-\varepsilon}2^{nd(2)-\frac{nd}{2}}\int_{i_02^{-n\gamma}-2^{-n\gamma}}^{i_02^{-n\gamma}+2^{-n\gamma+1}}|t-s|^{^{-1-\frac{(d-d(2))}{\gamma}+\frac{d}{2\gamma}}}\mathrm{d}t\\
    &\leq NT^{1-\varepsilon}2^{nd(2)-\frac{nd}{2}}2^{n(d-d(2))-\frac{nd}{2}}=NT^{1-\varepsilon},
\end{align*}
where $N=N(d,\gamma,N_2)$.
\smallskip

Next we show that $I_{2,1}+I_{3,1}\leq NT^{1-\varepsilon}$.
By similarity, we only estimate  $I_{3,1}$. By the fundamental theorem of calculus, Fubini's theorem and Theorem \ref{20.12.24.12.47},
\begin{equation*}
\begin{aligned}
    \int_{\bR^d}|K_{\varepsilon}(t,s_0,x-y_0)-K_{\varepsilon}(t,s_0,x-y_1)|\mathrm{d}x&\leq |y_1-y_0|\int_{\bR^d}\int_{0}^{1}|\nabla_x K_{\varepsilon}(t,s_0,x-y_{\theta})|\mathrm{d}\theta \mathrm{d}x\\
    & \leq 2^{-n}\sqrt{d}\int_{\bR^d}|\nabla_xK_{\varepsilon}(t,s_0,x)|\mathrm{d}x\leq N2^{-n}|t-s_0|^{-\varepsilon-\frac{1}{\gamma}},
\end{aligned}    
\end{equation*}
where $N=N(d,N_1,N_2)$ and $y_{\theta}:=\theta y_0+(1-\theta)y_1$. Therefore, since $s_0\in [i_02^{-n\gamma},(i_0+1)2^{-n\gamma})$,
\begin{align*}
I_{3,1}(s_0,y_0,y_1)&\leq N2^{-n}\int_{i_02^{-n\gamma}+2^{-n\gamma+1}}^{\infty}|t-s_0|^{-\varepsilon-\frac{1}{\gamma}}(1_{0<t-s_0<T}1_{\varepsilon\in[0,1)}+1_{\varepsilon=1})\mathrm{d}t\\
&\leq N T^{1-\varepsilon}2^{-n}\int_{i_02^{-n\gamma}+2^{-n\gamma+1}}^{\infty}|t-s_0|^{-1-\frac{1}{\gamma}}\mathrm{d}t=NT^{1-\varepsilon},
\end{align*}
where $N=N(d,\gamma,\varepsilon,N_1,N_2)$. 
Finally we claim that $I_{2,2}+I_{3,2}\leq NT^{1-\varepsilon}$.
Due to similarity, we only estimate $I_{3,2}$. For estimating $I_{3,2}$, we consider the two cases: $\varepsilon\in[0,1)$ and $\varepsilon=1$.
\smallskip

\textbf{Case 1.} $\varepsilon\in[0,1)$.

\smallskip

By Fubini's theorem and Theorem \ref{20.12.24.12.47},
\begin{equation*}
\begin{aligned}
I_{3,2}(s_0,s_1,y_1)&\leq\sum_{i=1}^2\int_{i_02^{-n\gamma}+2^{-n\gamma+1}}^{\infty}\left(\int_{\bR^d}|K_{\varepsilon}(t,s_i,x)|\mathrm{d}x\right)h(\varepsilon,t-s_i,T)\mathrm{d}t\\
&\leq N\sum_{i=1}^2\int_{i_02^{-n\gamma}+2^{-n\gamma+1}}^{\infty}|t-s_i|^{-\varepsilon}1_{0<t-s_i<T}\mathrm{d}t\\
&\leq N\sum_{i=1}^2\int_{s_i}^{s_i+T}|t-s_i|^{-\varepsilon}\mathrm{d}t=NT^{1-\varepsilon},
\end{aligned}    
\end{equation*}
where $N=N(d,\varepsilon,N_1,N_2)$.

\smallskip

\textbf{Case 2.} $\varepsilon=1$.

\smallskip

By the fundamental theorem of calculus, Fubini's theorem, and Theorem \ref{20.12.24.12.47},
\begin{equation*}
\begin{aligned}
    \int_{\bR^d}|K_{1}(t,s_0,x-y_1)-K_{1}(t,s_1,x-y_1)|\mathrm{d}x&\leq |s_1-s_0|\int_{\bR^d}\int_{0}^{1}|\p_t K_{1}(t,s_{\theta},x-y_1)|\mathrm{d}\theta \mathrm{d}x\\
    & \leq 2^{-n\gamma}\int_{0}^1\int_{\bR^d}|\p_tK_{1}(t,s_{\theta},x)|\mathrm{d}x\mathrm{d}\theta\\
    &\leq N2^{-n\gamma}\int_0^1|t-s_{\theta}|^{-2}\mathrm{d}\theta
\end{aligned}    
\end{equation*}
where $N=N(d,N_1,N_2)$ and $s_{\theta}:=(1-\theta)s_0+\theta s_1$. One can observe that for $s\in[i_02^{-n\gamma},(i_0+1)2^{-n\gamma})$,
$$
\frac{1}{2}\leq \frac{|t-s|}{|t-s_{\theta}|}\leq 2,
$$
Therefore,
\begin{align*}
    I_{3,2}(s_0,s_1,y_1)\leq N2^{-n\gamma}\int_{i_02^{-n\gamma}+2^{-n\gamma+1}}^{\infty}|t-s|^{-2}\mathrm{d}t\leq N,
\end{align*}
where $N=N(d,\gamma,N_1,N_2)$. The theorem is proved.
\end{proof}

\vspace{3mm}
{\bf Proof of Theorem \ref{20.12.21.13.23}}
\vspace{3mm}

By Lemma \ref{20.12.21.13.02}, $K_{\varepsilon}$ is a Calder\'on-Zygmund kernel relative to $\cP_n^{\gamma}$ (See \cite[Definition 3.1]{krylov2001calderon}). Therefore, by \cite[Theorem 4.1]{krylov2001calderon}, we obtain the desired result. The theorem is proved.

\subsection{Sharp-Maximal function estimates and proof of Theorem \ref{20.12.17.11.23}}
Throughout this subsection, we fix a $p_0 \in (1,2]$ and an operator $\cT_{\varepsilon,T}$ satisfying the $\left(\left\lfloor\frac{d}{p_0}\right\rfloor+2\right)$-times regular condition with $(\gamma,N_1,N_2)$. Recall
$$
Q_b^{\gamma}(t,x)=(t-b^{\gamma},t+b^{\gamma})\times \{ y \in \bR^d : |x-y| < b\}.
$$
\begin{lem}
\label{20.11.15.15.14}
Let  $t_0\in\bR$ and $b>0$.
Assume that $f\in C_c^{\infty}(\bR^{d+1})$ has a support in $(t_0-3b^{\gamma},t_0+3b^{\gamma})\times B_{3b}(0)$. 
Then for any $(t,x)\in Q_b^{\gamma}(t_0,0)$,
$$
\aint_{Q_b^{\gamma}(t_0,0)}|\cT_{\varepsilon,T} f(s,y)|^{p_0}\mathrm{d}y\mathrm{d}s\leq NT^{p_0(1-\varepsilon)} \bM|f|^{p_0}(t,x),
$$
where $N=N(d,p_0,\gamma,\varepsilon,N_1,N_2,N_3)$.
\end{lem}
\begin{proof}
By Theorem \ref{20.12.21.13.23},
\begin{align*}
    \aint_{Q_b^{\gamma}(t_0,0)}|\cT_{\varepsilon,T} f(s,y)|^{p_0}\mathrm{d}y\mathrm{d}s&\leq NT^{p_0(1-\varepsilon)}|Q_b^{\gamma}(t_0,0)|^{-1} \int_{\bR^{d+1}}|f(s,y)|^{p_0}\mathrm{d}y\mathrm{d}s \\
    &=NT^{p_0(1-\varepsilon)} |Q_b^{\gamma}(t_0,0)|^{-1}    \int_{t_0-3b^{\gamma}}^{t_0+3b^{\gamma}}\int_{B_{3b}(0)}|f(s,y)|^{p_0}\mathrm{d}y\mathrm{d}s\\
    &\leq NT^{p_0(1-\varepsilon)}\bM|f|^{p_0}(t,x),
\end{align*}
where $N=N(d,p_0,\gamma,\varepsilon,N_1,N_2,N_3)$. The lemma is proved.
\end{proof}

\begin{lem}
\label{20.12.21.16.26}
Let $t_0\in\bR$ and $b>0$.
Assume that $f\in C_c^{\infty}(\bR^{d+1})$ has a support in $(t_0-3b^{\gamma},\infty)\times \bR^d$. Then for any $(t,x)\in Q_b^{\gamma}(t_0,0)$,
$$
\aint_{Q_b^{\gamma}(t_0,0)}|\cT_{\varepsilon,T} f(s,y)|^{p_0}\mathrm{d}s\mathrm{d}y\leq NT^{p_0(1-\varepsilon)}\bM|f|^{p_0}(t,x),
$$
where $N=N(d,p_0,\gamma,\varepsilon,N_1,N_2,N_3)$.
\end{lem}
\begin{proof}
Choose a $\eta\in C^{\infty}(\bR)$ satisfying
\begin{itemize}
    \item $\eta(s)\in[0,1]$ for all $s\in\bR$
    \item  $\eta(s)=1$ for all $s\leq t_0+2b^{\gamma}$
    \item $\eta(s)=0$ for all $s\geq t_0+3b^{\gamma}$.
\end{itemize}
Observe that if $(t,x)\in Q_b^{\gamma}(t_0,0)$, then 
$$
\cT_{\varepsilon,T} (f\eta)(t,x)=\cT_{\varepsilon,T} f(t,x) \quad\text{and}\quad \bM|f\eta|^{p_0}(t,x)\leq \bM|f|^{p_0}(t,x).
$$
Thus we may assume that $f(s,y)=0$ if $|s-t_0|\geq 3b^{\gamma}$. 
Next choose a $\zeta\in C^{\infty}(\bR^d)$ satisfying
\begin{itemize}
    \item $\zeta(y)\in[0,1]$ for all $y\in\bR^d$
    \item  $\zeta(y)=1$ for all $y\in B_{2b}(0)$
    \item   $\zeta(y)=0$ for all $y\in \bR^d\setminus B_{5b/2}(0)$.
\end{itemize}
Note that $\cT_{\varepsilon,T} f=\cT_{\varepsilon,T} (f\zeta)+\cT_{\varepsilon,T} (f(1-\zeta))$ and $\cT(f\zeta)$ can be estimated by Lemma \ref{20.11.15.15.14}.
Thus it suffices to estimate  $\cT_{\varepsilon,T} (f(1-\zeta))$ and we may assume that $f(s,y)=0$ if $|y|<2b$.
In total, we assume that $f(s,y)=0$ if $|s-t_0|\geq 3b^{\gamma}$ or $|y|<2b$ without loss of generality.
Hence if $(s,y)\in Q_b^{\gamma}(t_0
,0)$ and $|z|<b$, then $|y-z|\leq 2b$ and $f(r,y-z)=0$. By Lemma \ref{20.11.16.14.53} and H\"older's inequality,
\begin{align*}
 \left|\int_{\bR^d}K_{\varepsilon}(s,r,y-z)f(r,z)\mathrm{d}z\right|&=\left|\int_{|z|\geq b}K_{\varepsilon}(s,r,z)f(r,y-z)\mathrm{d}z\right|\\
 &\leq N \int_{b}^{\infty}(\lambda^dK_{\varepsilon,0,\alpha}^{p_0'}(s,r,\lambda))^{1/p_0'}\left(\int_{B_{2b+\lambda}(x)}|f(r,z)|^{p_0}\mathrm{d}z\right)^{1/p_0}\mathrm{d}\lambda,
\end{align*}
where $N=N(d,p_0)$, $|\alpha|=1$, and
$$
K_{\varepsilon,0,\alpha}^{p_0'}(s,r,\lambda):=\int_{\bS^{d-1}}|D_x^{\alpha} K_{\varepsilon}(s,r,\lambda w)|^{p_0'}\sigma(dw).
$$
Recalling
    $$
    h(\varepsilon,t-s,T)=1_{\varepsilon\in[0,1)}1_{|t-s|<T}+1_{\varepsilon=1},
    $$
we also have
$$
|s-r|^{-\varepsilon-\frac{(d-d(p_0))}{\gamma}+\frac{d}{p_0'\gamma}}h(\varepsilon,s-r,T)\leq T^{1-\varepsilon}|s-r|^{-1-\frac{(d-d(p_0))}{\gamma}+\frac{d}{p_0'\gamma}}.
$$
By Lemma \ref{20.12.17.20.21}$(i)$ and H\"older's inequality,
\begin{align*}
    |\cT_{\varepsilon,T}(s,y)|&\leq T^{1-\varepsilon}\int_{t_0-3b^{\gamma}}^{s}\left(\int_{b}^{\infty}\lambda^{-p_0d(p_0)-1}\int_{B_{2b+\lambda}(x)}|f(r,z)|^{p_0}\mathrm{d}z\mathrm{d}\lambda \right)^{1/p_0}|s-r|^{-1-\frac{(d-d(p_0))}{\gamma}+\frac{d}{p_0'\gamma}}\mathrm{d}r\\
    &\leq T^{1-\varepsilon}\left(\int_{b}^{\infty}\lambda^{-p_0d(p_0)-1}\int_{t_0-3b^{\gamma}}^{s}\int_{B_{2b+\lambda}(x)}|f(r,z)|^{p_0}|s-r|^{-1-\frac{(d-d(p_0))}{\gamma}+\frac{d}{p_0'\gamma}}\mathrm{d}z \mathrm{d}r \mathrm{d}\lambda \right)^{1/p_0}\\
    &\qquad \times \left(\int_{t_0-3b^{\gamma}}^{s}|s-r|^{-1-\frac{(d-d(p_0))}{\gamma}+\frac{d}{p_0'\gamma}}\mathrm{d}r\right)^{1/p_0'}\\
    &\leq NT^{1-\varepsilon}b^{-\frac{(d-d(p_0))}{p_0'}+\frac{d}{p_0'^2}}\\
    &\qquad\times\left(\int_{b}^{\infty}\lambda^{-p_0d(p_0)-1}\int_{t_0-3b^{\gamma}}^{s}\int_{B_{2b+\lambda}(x)}|f(r,z)|^{p_0}|s-r|^{-1-\frac{(d-d(p_0))}{\gamma}+\frac{d}{p_0'\gamma}}\mathrm{d}z \mathrm{d}r \mathrm{d}\lambda \right)^{1/p_0}.
\end{align*}
By virtue of Fubini's theorem,
\begin{align*}
    &\int_{t_0-b^{\gamma}}^{t_0+b^{\gamma}}|\cT_{\varepsilon,T}(s,y)|^{p_0}\mathrm{d}s\\
    &\leq NT^{p_0(1-\varepsilon)}b^{-(p_0-1)(d-d(p_0))+\frac{(p_0-1)d}{p_0'}}\\
    &\quad\times\left(\int_{b}^{\infty}\lambda^{-p_0d(p_0)-1}\int_{t_0-3b^{\gamma}}^{t_0+b^{\gamma}}\int_{B_{2b+\lambda}(x)}|f(r,z)|^{p_0}\int_{r}^{t_0+b^{\gamma}}|s-r|^{-1-\frac{(d-d(p_0))}{\gamma}+\frac{d}{p_0'\gamma}}\mathrm{d}s\mathrm{d}z \mathrm{d}r \mathrm{d}\lambda\right)\\
    &\leq NT^{p_0(1-\varepsilon)}b^{-p_0(d-d(p_0))+\frac{p_0d}{p_0'}}\left(\int_{b}^{\infty}\lambda^{-p_0d(p_0)-1}\int_{t_0-3b^{\gamma}}^{t_0+b^{\gamma}}\int_{B_{2b+\lambda}(x)}|f(r,z)|^{p_0}\mathrm{d}z \mathrm{d}r \mathrm{d}\lambda\right)\\
    &\leq NT^{p_0(1-\varepsilon)}b^{-p_0(d-d(p_0))+\frac{p_0d}{p_0'}+\gamma-p_0d(p_0)+d}\bM|f|^{p_0}(t,x)\leq NT^{p_0(1-\varepsilon)}b^{\gamma}\bM|f|^{p_0}(t,x).
\end{align*}
Therefore,
\begin{align*}
    \aint_{Q_b^{\gamma}(t_0,0)}|\cT f(s,y)|^{p_0}\mathrm{d}y\mathrm{d}s&\leq N T^{p_0(1-\varepsilon)}\bM|f|^{p_0}(t,x),
\end{align*}
where $N=N(d,p_0,\gamma,\varepsilon,N_1,N_2,N_3)$. The lemma is proved.
\end{proof}

\begin{lem}
\label{20.12.21.16.24}
Let $t_0\in\bR$ and $b>0$.
Assume that $f\in C_c^{\infty}(\bR^{d+1})$ has a support in $(-\infty,t_0-2b^{\gamma})\times \bR^d$. Then for any $(t,x)\in Q_b^{\gamma}(t_0,0)$,
$$
\sup_{(s_1,y_1),(s_2,y_2)\in Q_b^{\gamma}(t_0,0)}|\cT_{\varepsilon,T} f(s_1,y_1)-\cT_{\varepsilon,T} f(s_2,y_2)|^{p_0}\leq NT^{p_0(1-\varepsilon)}\bM|f|^{p_0}(t,x),
$$
where $N=N(d,p_0,\gamma,\varepsilon,N_1,N_2,N_3)$.
\end{lem}
\begin{proof}
Recall that
\begin{equation*}
    \begin{gathered}
    h(\varepsilon,s-r,T):=1_{\varepsilon\in[0,1)}1_{|s-r|\leq T}+1_{\varepsilon=1},\\
    H(r,\lambda):=r^{-\frac{d}{\gamma}}\left(r^{-\frac{1}{\gamma}}1_{\lambda\leq r^{\frac{1}{\gamma}}}+r^{\frac{p_0d(p_0)}{\gamma}}\lambda^{-p_0d(p_0)-1}1_{r^{\frac{1}{\gamma}}\leq \lambda}\right),\\
    K_{\varepsilon,m,\alpha}^{p_0'}(s,r,\lambda):=\int_{\bS^{d-1}}|\p_t^mD_x^2 K_{\varepsilon}(s,r,\lambda w)|^{p_0'}\sigma(dw).
    \end{gathered}
\end{equation*}
By the triangle inequality,
\begin{align*}
    |\cT_{\varepsilon,T}f(s_1,y_1)-\cT_{\varepsilon,T}f(s_2,y_2)|&\leq |\cT_{\varepsilon,T}f(s_1,y_1)-\cT_{\varepsilon,T}f(s_1,y_2)|+|\cT_{\varepsilon,T}f(s_1,y_2)-\cT_{\varepsilon,T}f(s_2,y_2)|\\
    &=: I+II.
\end{align*}
We claim that
$$
I+II\leq NT^{1-\varepsilon}(\bM|f|^{p_0}(t,x))^{1/p_0}.
$$

First, we estimate $I$. By the fundamental theorem of calculus, Lemma \ref{20.11.16.14.53} with $R_1=2b$ and $R_2=0$ and Lemma \ref{20.12.17.20.21} $(ii)$,
\begin{equation*}
\begin{aligned}
    &I=|\cT_{\varepsilon,T} f(s_1,y_1)-\cT_{\varepsilon,T} f(s_1,y_2)|\\
    &= \left|\int_{-\infty}^{t_0-2b^{\gamma}}\int_{\bR^d}K_{\varepsilon}(s_1,r,y_1-y)-K_{\varepsilon}(s_1,r,y_2-y)f(r,y)h(\varepsilon,s_1-r,T)\mathrm{d}y\mathrm{d}r\right|\\
    &= |y_2-y_1|\left|\int_{-\infty}^{t_0-2b^{\gamma}}\int_{\bR^d}\int_{0}^{1}\nabla_x K(s_1,r,y_{\theta}-y)f(r,y)h(\varepsilon,s_1-r,T)\mathrm{d}\theta \mathrm{d}y\mathrm{d}r\right|\\
    & \leq 2b\left|\int_{0}^1\int_{-\infty}^{t_0-2b^{\gamma}}\int_{\bR^d}\nabla_xK_{\varepsilon}(s_1,r,y)f(r,y_{\theta}-y)h(\varepsilon,s_1-r,T) \mathrm{d}y\mathrm{d}r\mathrm{d}\theta\right|\\
    &\leq Nb\int_{-\infty}^{t_0-2b^{\gamma}}\int_{0}^{\infty}\left(\int_{B_{2b+\lambda}(x)}|f(r,z)|^{p_0}\right)^{1/p_0}(\lambda^dK_{\varepsilon,0,\alpha}^{p_0'}(s_1,r,\lambda))^{1/p_0'}\mathrm{d}\lambda h(\varepsilon,s_1-r,T)\mathrm{d}r\\
    &\leq N\int_{-\infty}^{t_0-2b^{\gamma}}\left(\int_{0}^{\infty}H(|s_1-r|,\lambda)\int_{B_{2b+\lambda}(x)}|f(r,z)|^{p_0}\mathrm{d}z\mathrm{d}\lambda \right)^{1/p_0}|s_1-r|^{-\varepsilon-\frac{1}{\gamma}}h(\varepsilon,s_1-r,T)\mathrm{d}r,
\end{aligned}    
\end{equation*}
where $N=N(d,p_0,\gamma,\varepsilon,N_1,N_2)$, $y_{\theta}:=(1-\theta)y_1+\theta y_2$ and $|\alpha|=2$. By Lemma \ref{20.12.28.16.02} with $(m,|\alpha|)=(0,2)$,
$$
I=|\cT_{\varepsilon,T} f(s_1,y_1)-\cT_{\varepsilon,T} f(s_1,y_2)|\leq NT^{1-\varepsilon}(\bM|f|^{p_0}(t,x))^{1/p_0}.
$$
For $II$, we consider the two cases: $\varepsilon\in[0,1)$ and $\varepsilon=1$.
\smallskip

\textbf{Case 1.} $\varepsilon\in[0,1)$.
\smallskip

By Lemma \ref{20.11.16.14.53} with $R_1=2b$ and $R_2=0$ and Lemma \ref{20.12.17.20.21}, for $(s,y)\in Q_b^{\gamma}(t_0,0)$,
\begin{align*}
 &|\cT_{\varepsilon,T}f(s,y)|=\left|\int_{-\infty}^{t_0-2b^{\gamma}}\int_{\bR^d}K_{\varepsilon}(s,r,y-z)f(r,z)1_{|s-r|<T}\mathrm{d}z\mathrm{d}r\right|\\
 &\leq N \int_{-\infty}^{t_0-2b^{\gamma}}\int_{0}^{\infty}(\lambda^dK_{\varepsilon,0,\alpha}^{p_0'}(s,r,\lambda))^{1/p_0'}\left(\int_{B_{2b+\lambda}(x)}|f(r,z)|^{p_0}\mathrm{d}z\right)^{1/p_0}1_{|s-r|<T}\mathrm{d}\lambda  \mathrm{d}r\\
 &\leq N\int_{-\infty}^{t_0-2b^{\gamma}}\left(\int_{0}^{\infty}H(|s-r|,\lambda)\int_{B_{2b+\lambda}(x)}|f(r,z)|^{p_0}\mathrm{d}z\mathrm{d}\lambda \right)^{1/p_0}|s-r|^{-\varepsilon}1_{|s-r|<T} \mathrm{d}r,
\end{align*}
where $N=N(d,p_0,\varepsilon,N_1,N_2)$ and $|\alpha|=1$. By Lemma \ref{20.12.28.16.02} with $(m,|\alpha|)=(0,1)$,
$$
II\leq |\cT_{\varepsilon,T} f(s_1,y_2)|+|\cT_{\varepsilon,T} f(s_2,y_2)|\leq NT^{1-\varepsilon}(\bM|f|^{p_0}(t,x))^{1/p_0}.
$$

\textbf{Case 2.} $\varepsilon=1$.
\smallskip

By the fundamental theorem of calculus, Lemma \ref{20.11.16.14.53} with $R_1=2b$ and $R_2=0$, and Lemma \ref{20.12.17.20.21},
\begin{equation*}
\begin{aligned}
&|\cT_{1,T} f(s_1,y_2)-\cT_{1,T} f(s_2,y_2)|\\
&= \left|\int_{-\infty}^{t_0-2b^{\gamma}}\int_{\bR^d}K_{1}(s_1,r,y_2-y)-K_{1}(s_2,r,y_2-y)f(r,y)\mathrm{d}y\mathrm{d}r\right|\\
&= |s_2-s_1|\left|\int_{-\infty}^{t_0-2b^{\gamma}}\int_{\bR^d}\int_{0}^{1}\p_t K_{1}(s_{\theta},r,y_2-y)f(r,y)\mathrm{d}\theta \mathrm{d}y\mathrm{d}r\right|\\
& \leq 2b^{\gamma}\left|\int_{0}^1\int_{-\infty}^{t_0-2b^{\gamma}}\int_{\bR^d}\p_t K_{1}(s_{\theta},r,y)f(r,y_2-y) \mathrm{d}y \mathrm{d}r\mathrm{d}\theta\right|\\
&\leq Nb^{\gamma}\int_{0}^1\int_{-\infty}^{t_0-2b^{\gamma}}\int_{0}^{\infty}\left(\int_{B_{2b+\lambda}(x)}|f(r,z)|^{p_0}\right)^{1/p_0}(\lambda^dK_{1,1,\alpha}^{p_0'}(s_{\theta},r,\lambda))^{1/p_0'}\mathrm{d}\lambda \mathrm{d}r\mathrm{d}\theta\\
&\leq Nb^{\gamma}\int_0^1\int_{-\infty}^{t_0-2b^{\gamma}}\left(\int_{0}^{\infty}H(|s_{\theta}-r|,\lambda)\int_{B_{2b+\lambda}(x)}|f(r,z)|^{p_0}\mathrm{d}z\mathrm{d}\lambda \right)^{1/p_0}|s_{\theta}-r|^{-2}\mathrm{d}r\mathrm{d}\theta,
\end{aligned}    
\end{equation*}
where $N=N(d,p_0,\gamma,\varepsilon,N_1,N_2)$, $s_{\theta}:=(1-\theta)s_1+\theta s_2$ and $|\alpha|=1$. By Lemma \ref{20.12.28.16.02} with $(m,|\alpha|)=(1,1)$,
$$
|\cT_{1,T} f(s_1,y_2)-\cT_{1,T} f(s_2,y_2)|\leq N(\bM|f|^{p_0}(t,x))^{1/p_0}.
$$
The lemma is proved.
\end{proof}

\begin{thm}
\label{20.12.21.16.38}
There exists a constant $N=N(d,p_0,\gamma,\varepsilon,N_1,N_2,N_3)$ such that for any $f\in C_c^{\infty}(\bR^{d+1})$,
$$
(\cT_{\varepsilon,T}f)^{\sharp}(t,x)\leq NT^{1-\varepsilon}(\bM|f|^{p_0}(t,x))^{1/p_0}.
$$
\end{thm}
\begin{proof}
Let $b>0$, $t_0\in\bR$, and $(t,x)\in Q_b^{\gamma}(t_0,0)$.
Choose a $\eta\in C^{\infty}(\bR)$ satisfying
\begin{itemize}
    \item  $\eta(t)\in[0,1]$ for all $t\in \bR$.
    \item  $\eta(t)=1$ for all $t<t_0-8b^{\gamma}/3$
    \item  $\eta(t)=0$ for all $t\geq t_0-7b^{\gamma}/3$.
\end{itemize}
Put $\eta^*:=1-\eta$. Then
\begin{align*}
    |\cT_{\varepsilon,T} f(s_1,y_1)-\cT_{\varepsilon,T} f(s_0,y_0)|&\leq |\cT_{\varepsilon,T} (f\eta)(s_1,y_1)-\cT_{\varepsilon,T} (f\eta)(s_0,y_0)|\\
    &\quad +|\cT_{\varepsilon,T} (f\eta^*)(s_1,y_1)-\cT_{\varepsilon,T} (f\eta^*)(s_0,y_0)|.
\end{align*}
By Lemma \ref{20.12.21.16.24},
\begin{align*}
\aint_{Q_b^{\gamma}(t_0,0)}\aint_{Q_b^{\gamma}(t_0,0)}|\cT_{\varepsilon,T} (f\eta)(s_1,y_1)-\cT_{\varepsilon,T} (f\eta)(s_0,y_0)|^{p_0}\mathrm{d}y_1\mathrm{d}s_1\mathrm{d}y_0\mathrm{d}s_0\leq NT^{p_0(1-\varepsilon)}\bM|f|^{p_0}(t,x)
\end{align*}
and by Lemma \ref{20.12.21.16.26},
\begin{align*}
 &\aint_{Q_b^{\gamma}(t_0,0)}\aint_{Q_b^{\gamma}(t_0,0)}|\cT_{\varepsilon,T} (f\eta^*)(s_1,y_1)-\cT_{\varepsilon,T} (f\eta^*)(s_0,y_0)|^{p_0}\mathrm{d}y_1\mathrm{d}s_1\mathrm{d}y_0\mathrm{d}s_0\\
 &\leq 2\aint_{Q_b^{\gamma}(t_0,0)}|\cT_{\varepsilon,T} (f\eta^*)(s,y)|^{p_0}\mathrm{d}s\mathrm{d}y\leq NT^{p_0(1-\varepsilon)}\bM|f|^{p_0}(t,x).   
\end{align*}
Hence,
$$
\aint_{Q_b^{\gamma}(t_0,0)}\aint_{Q_b^{\gamma}(t_0,0)}|\cT_{\varepsilon,T} f(s_1,y_1)-\cT_{\varepsilon,T} f(s_0,y_0)|^{p_0}\mathrm{d}y_1\mathrm{d}s_1\mathrm{d}y_0\mathrm{d}s_0\leq NT^{p_0(1-\varepsilon)}\bM|f|^{p_0}(t,x),
$$
where $N=N(d,p_0,\gamma,\varepsilon,N_1,N_2,N_3)$. For $x_0\in\bR^{d}$, denote 
$$
\tau_{x_0}f(t,x):=f(t,x_0+x).
$$
Since $\cT_{\varepsilon,T}$ and $\tau_{x_0}$ are commutative,
\begin{align*}
    &\aint_{Q_b^{\gamma}(t_0,x_0)}\aint_{Q_b^{\gamma}(t_0,x_0)}|\cT_{\varepsilon,T} f(s_1,y_1)-\cT_{\varepsilon,T} f(s_0,y_0)|^{p_0}\mathrm{d}y_1\mathrm{d}s_1\mathrm{d}y_0\mathrm{d}s_0\\
    &=\aint_{Q_b^{\gamma}(t_0,0)}\aint_{Q_b^{\gamma}(t_0,0)}|\cT_{\varepsilon,T} (\tau_{x_0}f)(t,x)-\cT_{\varepsilon,T} (\tau_{x_0}f)(s,y)|^{p_0}\mathrm{d}x\mathrm{d}t\mathrm{d}y\mathrm{d}s\\
    &\leq NT^{p_0(1-\varepsilon)}\bM|\tau_{x_0}f|^{p_0}(t,x)=NT^{p_0(1-\varepsilon)}\bM|f|^{p_0}(t,x_0+x).
\end{align*}
Therefore, by Jensen's inequality, for $(t,x)\in Q_b^{\gamma}(t_0,0)$ and $x_0\in\bR^{d}$
\begin{align*}
  &\left(\aint_{Q_b^{\gamma}(t_0,x_0)}\aint_{Q_b^{\gamma}(t_0,x_0)}|\cT_{\varepsilon,T} f(s_1,y_1)-\cT_{\varepsilon,T} f(s_0,y_0)|\mathrm{d}y_1\mathrm{d}s_1\mathrm{d}y_0\mathrm{d}s_0\right)^{p_0}\\
  &\leq \aint_{Q_b^{\gamma}(t_0,x_0)}\aint_{Q_b^{\gamma}(t_0,x_0)}|\cT_{\varepsilon,T} f(s_1,y_1)-\cT_{\varepsilon,T} f(s_0,y_0)|^{p_0}\mathrm{d}y_1\mathrm{d}s_1\mathrm{d}y_0\mathrm{d}s_0\\
  &\leq N T^{p_0(1-\varepsilon)}\bM|f|^{p_0}(t,x_0+x),
\end{align*}
where $N=N(d,p_0,\gamma,\varepsilon,N_1,N_2,N_3)$. Taking the supremum on both sides with respect to all $Q_b^{\gamma}$ containing $(t,x_0+x)$, we obtain the desired result. The theorem is proved.
\end{proof}

\vspace{3mm}
{\bf Proof of Theorem \ref{20.12.17.11.23}}
\vspace{3mm}

Because of the similarity, we only prove $(i)$. Let $w\in A_p(\bR^{d+1})$.
Choose a $p_0 \in (1,R_{p,d+1}^w] \subseteq (1,2]$ so that $w \in A_{p/p_0}(\bR^{d+1})$ and $\lfloor d/R_{p,d+1}^w \rfloor =  \lfloor d/p_0 \rfloor$.
Then, it is obvious that $\cT_{\varepsilon,T}$ satisfies satisfies $(\lfloor d/p_0\rfloor+2)$ times regular condition with $(\gamma,N_1,N_2)$. 
Thus, by Theorem \ref{20.12.21.16.38},
\begin{align}
							\label{2021041101}
\|(\cT_{\varepsilon,T}f)^{\sharp}\|_{L_p(\bR^{d+1},w)}\leq NT^{1-\varepsilon}\|(\bM|f|^{p_0})^{1/p_0}\|_{L_p(\bR^{d+1},w)},
\end{align}
where $N=N(d,p,\gamma,\varepsilon,N_1,N_2,N_3,[w]_{A_p(\bR^{d+1})})$. 
Moreover, recalling  $w\in A_{p/p_0}(\bR^{d+1})$ and applying Theorem \ref{20.12.21.16.37}(ii) and Remark \ref{21.02.23.13.03}, we have
\begin{align}
							\label{2021041102}
\|(\bM|f|^{p_0})^{1/p_0}\|_{L_p(\bR^{d+1},w)}&=\|\bM|f|^{p_0}\|_{L_{p/p_0}(\bR^{d+1},w)}\leq N\||f|^{p_0}\|_{L_{p/p_0}(\bR^{d+1},w)}=N\|f\|_{L_p(\bR^{d+1},w)},
\end{align}
where $N=N(d,p,\gamma,[w]_{A_p(\bR^{d+1})})$. 
Finally, using Theorem \ref{20.12.21.16.37}(i),  \eqref{2021041101}, and \eqref{2021041102}, we obtain
\begin{align*}
    \|\cT_{\varepsilon,T}f\|_{L_p(\bR^{d+1},w)}&\leq N\|(\cT_{\varepsilon,T}f)^{\sharp}\|_{L_p(\bR^{d+1},w)}\leq NT^{1-\varepsilon}\|(\bM|f|^{p_0})^{1/p_0}\|_{L_p(\bR^{d+1},w)}\leq NT^{1-\varepsilon}\|f\|_{L_p(\bR^{d+1},w)},
\end{align*}
where $N=N(d,p,\gamma,\varepsilon,N_1,N_2,N_3,[w]_{A_p(\bR^{d+1})})$. The theorem is proved.

\mysection{Proof of Theorems \ref{21.01.08.17.12} and \ref{21.01.08.17.12-2}}
											\label{pf main theorems}

Recall the kernel 
\begin{equation*}
    \begin{gathered}
    p(t,s,x):=\frac{1}{(2\pi)^{d/2}}\int_{\bR^d} \exp\left(\int_{s}^t\psi(r,\xi)\mathrm{d}r\right)\mathrm{e}^{ix\cdot\xi}\mathrm{d}\xi \cdot 1_{t>s\geq0},\\ P_{\varepsilon}(t,s,x):=(-\Delta)^{\varepsilon\gamma/2}_xp(t,s,x),
    \end{gathered}
\end{equation*}
and the operator
$$
    \cK_{\varepsilon,T} f(t,x):=\int_{-\infty}^{t}\int_{\bR^d}\left(1_{\varepsilon\in[0,1)}1_{|t-s|<T}+1_{\varepsilon=1}\right)P_{\varepsilon}(t,s,x-y)f(s,y)\mathrm{d}y\mathrm{d}s.
$$
Throughout the section, we fix $\varepsilon \in [0,1]$, $\nu\in\bR$, $p,q\in(1,\infty)$, $w\in A_p(\bR^{d+1})$, $w_1\in A_q(\bR)$, and $w_2\in A_p(\bR^d)$. 
If $\nu\neq0$, then we additionally assume that $w(t,\cdot)\in A_p(\bR^d)$ for almost all $t\in[0,T]$ and
$$
\esssup_{t\in[0,T]}[w(t,\cdot)]_{A_p(\bR^d)}=:N_0<\infty.
$$

Below, to treat two types of weighted spaces (time-space mixed or not) simultaneously, we introduce a unified notation.
We use the notation $\bB_0(T)$ and $\bB_1(T)$ to denote either
$$
(\bB_0(T),\bB_1(T))=\left(\bH_p^{\nu}((0,T)\times\bR^d,w), \bH_p^{\nu + \gamma}((0,T)\times\bR^d,w)\right)
$$
or
$$
(\bB_0(T),\bB_1(T))=\left( L_q((0,T),w_1;H_p^{\nu}(\bR^d,w_2)), u\in L_q((0,T),w_1;H_p^{\nu+\gamma}(\bR^d,w_2))\right).
$$

\begin{thm}[A priori estimate]
\label{21.03.01.17.40}
Suppose that $\psi(t,\xi)$ satisfies the ellipticity condition with $(\gamma,\kappa)$ and has a $(\lfloor d/R_{p,d+1}^w\rfloor+2)$-times (resp. $(\lfloor d/R_{q,1}^{w_1}\rfloor\vee\lfloor d/R_{p,d}^{w_2}\rfloor+2)$-times) regular upper bound with $(\gamma,M)$. Then, for $f\in \bB_0(T)$,
\begin{align*}
    \|\cK_{\varepsilon,T}(1_{[0,T]}f)\|_{\bB_1(T)}&\leq NT^{1-\varepsilon}\|f\|_{\bB_0(T)},\\
    \|\psi(\cdot,-i\nabla)\cK_{0,T}(1_{[0,T]}f)\|_{\bB_1(T)}&\leq N\|f\|_{\bB_0(T)},
\end{align*}
where $N=N(d,p,\gamma,\kappa,M,[w]_{A_p(\bR^{d+1})},\nu,N_0)$ (resp. $N=N(d,p,q,\nu,\gamma,\kappa,M,[w_1]_{A_p(\bR^{d})},[w_2]_{A_p(\bR^{d})})$).
\end{thm}
\begin{proof}
Due to Theorem \ref{21.01.13.14.11} and \ref{21.03.01.17.32}, the family of the operators $\cK_{\varepsilon,T}$ is a particular case of $\cT_{\varepsilon,T}$ in Theorem \ref{20.12.17.11.23}. For the second estimate, we can demonstrate that
\begin{align*}
    \psi(t,-i\nabla)\cK_{0,T}(1_{[0,T]}f)(t,x)&=\int_0^t\int_{\bR^d}\psi(t,-i\nabla)p(t,s,x-y)f(s,y)\mathrm{d}y\mathrm{d}s\\
    &=\int_0^t\int_{\bR^d}\partial_tp(t,s,x-y)f(s,y)\mathrm{d}y\mathrm{d}s,
\end{align*}
where $\partial_tp$ is also a specific instance of a $\cT_{1,T}$ operator as defined in Theorem \ref{20.12.17.11.23}. Therefore, these estimates can be easily derived using Theorem \ref{20.12.17.11.23}. The theorem is proved.
\end{proof}

\begin{corollary}[Uniqueness of a solution]
Let $f \in \bB_0(T)$. Suppose that $\psi(t,\xi)$ satisfies the ellipticity condition with $(\gamma,\kappa)$ and has a $(\lfloor d/R_{p,d+1}^w\rfloor+2)$-times (or $( \lfloor d/R_{q,1}^{w_1}\rfloor\vee\lfloor d/R_{p,d}^{w_2}\rfloor+2)$-times) regular upper bound with $(\gamma,M)$. Then, the Cauchy problem
\begin{equation}
\label{21.03.01.17.53}
\begin{cases}
\p_tu(t,x)=\psi(t,-i\nabla)u(t,x)+f(t,x),\quad &(t,x)\in(0,T)\times\mathbb{R}^d,\\
u(0,x)=0,\quad & x\in\mathbb{R}^d,
\end{cases}
\end{equation}
has at most one solution in the sense of Definition \ref{21.03.01.15.51}.
\end{corollary}
\begin{proof}

Let $u,v\in \bB_1(T)$ be solutions to Cauchy problem \eqref{21.03.01.17.53}. 
Then by the definition of solution, there exist $u_n,v_n\in C_p^{1,\infty}([0,T]\times\bR^d)$ such that $u_n(0,\cdot),v_n(0,\cdot)=0$,
\begin{align*}
\p_tu_n-\psi(t,-i\nabla)u_n \to f,\quad \p_tv_n-\psi(t,-i\nabla)v_n\to f\quad\text{in}\quad \bB_0(T)
\end{align*}
and
\begin{align*}
u_n\to u \quad v_n\to v\quad\text{in}\quad \bB_1(T).
\end{align*}
Denote
\begin{equation*}
    \begin{gathered}
    f_n:=\p_tu_n-\psi(t,-i\nabla)u_n\in C_p^{\infty}([0,T]\times\bR^d),\\
    g_n:=\p_tv_n-\psi(t,-i\nabla)v_n\in C_p^{\infty}([0,T]\times\bR^d).
    \end{gathered}
\end{equation*}
Then by Theorem \ref{21.01.08.11.03}, we have
\begin{align}
								\label{20210617 eqn 01}
    u_n(t,x)=\cK_{0,T} f_n(t,x),\quad v_n(t,x)=\cK_{0,T} g_n(t,x).
\end{align}
Recall that the operator $\cK_{0,T}$ is continuous from $\bB_0(T)$ to $\bB_1(T)$ due to Theorem \ref{21.03.01.17.40}.
Therefore, taking limits in \eqref{20210617 eqn 01}, we conclude that $u=\cK_{0,T} f =v$ in $\bB_1(T)$ since both $f_n$ and $g_n$ converge to $f$ in $\bB_0(T)$. 
The corollary is proved.
\end{proof}

\begin{corollary}[Existence of a solution]
Let $f \in \bB_0(T)$.
Suppose that $\psi(t,\xi)$ satisfies the ellipticity condition with $(\gamma,\kappa)$ and has a $( \lfloor d/R_{p,d+1}^w\rfloor+2)$-times (or $( \lfloor d/R_{q,1}^{w_1}\rfloor\vee\lfloor d/R_{p,d}^{w_2}\rfloor+2)$-times) regular upper bound with $(\gamma,M)$. Then, the Cauchy problem
\begin{equation}
\label{21.03.01.18.00}
\begin{cases}
\p_tu(t,x)=\psi(t,-i\nabla)u(t,x)+f(t,x),\quad &(t,x)\in(0,T)\times\mathbb{R}^d,\\
u(0,x)=0,\quad & x\in\mathbb{R}^d,
\end{cases}
\end{equation}
has a solution $u\in \bB_1(T)$ and $u$ has the integral representation
$$
u(t,x):=\cK_{0,T}f(t,x)=\int_{0}^{t}\int_{\bR^d}p(t,s,x-y)f(s,y)\mathrm{d}y\mathrm{d}s.
$$
\end{corollary}
\begin{proof} By Theorem \ref{21.03.01.15.34}, there exists a sequence $\{f_n\}_{n=1}^{\infty}\subseteq C_c^{\infty}((0,T)\times\bR^d)$ such that $f_n\to f$ in $\bB_0(T)$. 
For each $f_n$, due to Theorem  \ref{21.01.08.11.03}, there exists a unique solution $u_n$ to \eqref{21.03.01.18.00} such that
$$
u_n(t,x)=\cK_{0,T}f_n(t,x)\in C_p^{1,\infty}([0,T]\times\bR^d).
$$
By Theorem  \ref{21.03.01.17.40} and well-known embedding property of the Sobolev space (cf. Theorem \ref{21.02.23.17.15}(iii) and (vii)), 
\begin{equation*}
    \begin{aligned}
    &\|(-\Delta)^{\frac{\varepsilon\gamma}{2}}u_n\|_{\bB_0(T)}=\|\cK_{\varepsilon,T}f_n\|_{\bB_0(T)}\leq N T^{1-\varepsilon}\|f_n\|_{\bB_0(T)} \quad \forall \varepsilon\in[0,1],\\
    &\|\partial_tu_n\|_{\bB_0(T)}\leq\|\psi(\cdot,-i\nabla)\cK_{\varepsilon,T}f_n\|_{\bB_0(T)}+\|f_n\|_{\bB_0(T)}\leq N\|f_n\|_{\bB_0(T)},
    \end{aligned}
\end{equation*}
where $N$ is independent of $n$, $\varepsilon$ and $T$. 
Thus, due to the linearity of the equation and the operator, we have
$$
\|\partial_tu_n - \partial_tu_m\|_{\bB_0(T)}+\|\psi(\cdot,-i\nabla)u_n - \psi(\cdot,-i\nabla)u_m\|_{\bB_0(T)}+\|u_n - u_m\|_{\bB_1(T)}\leq N(1+T)\|f_n-f_m\|_{\bB_0(T)}
$$
for all natural numbers $n$ and $m$. 
Finally, since $\bB_0(T)$ and $\bB_1(T)$ are Banach spaces (see Theorem \ref{21.02.23.17.15} $(i)$ and \ref{21.03.01.14.59} $(i)$), there exists a $u\in \bB_1(T)$ such that $u_n(0,\cdot)=0$, $u_n\to u$ in $\bB_1(T)$ and 
$$
\partial_tu_n-\psi(t,-i\nabla)u_n\to f\quad \text{in}\quad \bB_0(T).
$$
The corollary is proved.
\end{proof}

\appendix

\mysection{Properties of function spaces}
															\label{append}

We prove some properties of Sobolev spaces with $A_p$-weights for the completeness of our paper.
Most properties can be easily deduced from weighted theories in \cite{fackler2020weighted, grafakos2014classical, kurtz1979results, miller1982weighted} and classical Sobolev space theories. Note that most approximations based on Sobolev mollifiers are non-trivial in weighted spaces since our weights are neither bounded above nor bounded below in general.

\begin{thm}
				\label{21.02.23.17.15}
Let $p\in(1,\infty)$, $w\in A_p(\bR^d)$, and $\nu\in\bR$.
\begin{enumerate}[(i)]
    \item $H_p^{\nu}(\bR^d,w)$ is a Banach space equipped with the norm 
$$
\|f\|_{H_p^{\nu}(\bR^d,w)}:=\|(1-\Delta)^{\nu/2}f\|_{L_p(\bR^d,w)}
$$
    \item For $\nu_0\in\bR$, $(1-\Delta)^{\nu_0/2}$ is a bijective isometry from $H_p^{\nu}(\bR^d,w)$ to $H_p^{\nu+\nu_0}(\bR^d,w)$.
    \item If $\nu_1\geq \nu_0$, then $H_p^{\nu_1}(\bR^d,w)$ is continuously embedded into $H_p^{\nu_0}(\bR^d,w)$. More precisely, there exists a positive constant $N$ such that
    $$
    \|f\|_{H_p^{\nu_0}(\bR^d,w)}\leq N\|f\|_{H_p^{\nu_1}(\bR^d,w)},\quad \forall f\in \cS(\bR^d).
    $$
    \item If $\nu$ is a nonnegative integer, then
    $$
    \|f\|_{H_p^{\nu}(\bR^d,w)}\simeq \sum_{|\alpha|\leq \nu}\|D^{\alpha}f\|_{L_p(\bR^d,w)}.
    $$
    \item $C_c^\infty(\bR^{d})$ is dense in $H_p^{\nu}(\bR^d,w)$.
    \item The topological dual space of $H_p^{\nu}(\bR^d,w)$ is $H_{p'}^{-\nu}(\bR^d,\bar{w})$, where
    $$
    \frac{1}{p}+\frac{1}{p'}=1,\quad \bar{w}:=w^{-\frac{1}{p-1}}\in A_{p'}(\bR^d).
    $$
    Moreover, for $f\in H_p^{\nu}(\bR^d,w)$,
    $$
    \sup_{\|\phi\|_{H_{p'}^{-\nu}(\bR^d,\bar{w})\leq1}}\left|(f,\phi)\right|=\|f\|_{H_p^{\nu}(\bR^d,w)}
    $$
    \item If $\nu\geq0$, then the norm $\|\cdot\|_{H_p^{\nu}(\bR^d,w)}$ is equivalent to $\|\cdot\|_{L_p(\bR^d,w)}+\|\cdot\|_{\dot{H}_p^{\nu}(\bR^d,w)}$, where
    $$
    \|f\|_{\dot{H}_p^{\nu}(\bR^d,w)}:=\|(-\Delta)^{\nu/2}f\|_{L_p(\bR^d,w)}.
    $$
\end{enumerate}
\end{thm}
\begin{proof}
For ($i$) $\sim$ ($iii$), see \cite[Section 3]{miller1982weighted}. For ($iv$), see \cite[Theorem 3.3]{miller1982weighted}. 

($v$) Note that $\cS(\bR^d)$ is dense in $H_p^{\nu}(\bR^d,w)$ (see \cite[Section 3]{miller1982weighted}).
In particular, 
$$
\cS(\bR^d) \subset H_p^{\bar \nu}(\bR^d,w) \quad \forall \bar \nu \in \bR.
$$
Thus it is obviously sufficient to prove that $C_c^{\infty}(\bR^d)$ is dense in $\cS(\bR^d)$ under the norm in $H_p^{\nu}(\bR^d,w)$. 
Let $f\in\cS(\bR^d)$ and $\eta\in C_c^{\infty}(\bR^d)$ be a nonnegative function satisfying $\eta(0)=1$ and $\eta_n(x):=\eta(x/n)$. Clearly, $f\eta_n\in C_c^{\infty}(\bR^d)$. By Leibniz's product rule and the dominated convergence theorem, for any multi-index $\alpha$,
$$
\|D^{\alpha}(f\eta_n)-D^{\alpha}f\|_{L_p(\bR^d,w)}\to 0.
$$
Thus by ($iv$), 
\begin{align}
						\label{2021071401}
f\eta_n\to f ~\text{in}~ H_p^{\nu}(\bR^d,w)
\end{align}
if $\nu$ is a positive integer.  
Moreover, for general $\nu$, one can find a positive integer $\nu_1$ such that $\nu_1 \geq \nu$.
Therefore, due to ($iii$),  \eqref{2021071401} holds for all $\nu \in \bR$. 

($vi$) Since $(1-\Delta)^{\nu/2}(1-\Delta)^{-\nu/2}=(1-\Delta)^{-\nu/2}(1-\Delta)^{\nu/2}$ is an identity operator on $H_q^{\nu'}(\bR^d,w)$ for all $q\in(1,\infty)$ and $\nu'\in\bR$, it suffices to show that the topological dual space of $L_p(\bR^d,w)$ is $L_{p'}(\bR^d,\bar{w})$ and
$$
\sup_{\|\phi\|_{L_{p'}(\bR^d,\bar{w})}\leq1}\left|\int_{\bR^d}f(x)\phi(x) \mathrm{d}x\right|=\|f\|_{L_p(\bR^d,w)}.
$$
Recall Riesz's theorem, first. That is, any bounded linear function on $L_p(\bR^d)$ is given by a function in $\in L_{p'}(\bR^d)$.
We prove the statement by fitting our weighted classes  in the setting of classical Riesz's theorem as follows.
\begin{itemize}
    \item Let $f$ be a bounded linear functional on $L_p(\bR^d,w)$. Then it is easy to check that $fw^{-1/p}$ is a bounded linear functional on $L_p(\bR^d)$, thus, $fw^{-1/p}\in L_{p'}(\bR^d)$. This certainly implies that $f\in L_{p'}(\bR^d,\bar{w})$. 
    \item By H\"older's inequality,
    \begin{align*}
\sup_{\|\phi\|_{L_{p'}(\bR^d,\bar{w})}\leq1}\left|\int_{\bR^d}f(x)\phi(x) \mathrm{d}x\right|
&=\sup_{\|\phi\|_{L_{p'}(\bR^d,\bar{w})}\leq1}\left|\int_{\bR^d}f(x)w^{1/p}(x) \phi(x) w^{-1/p}(x) \mathrm{d}x\right|\leq\|f\|_{L_p(\bR^d,w)}.
    \end{align*}
    For $f\in L_p(\bR^d,w)$, $fw^{1/p}\in L_p(\bR^d)$. Thus,
    $$
    \|f\|_{L_p(\bR^d,w)}=\|fw^{1/p}\|_{L_p(\bR^d)}=\sup_{\|\phi\|_{L_{p'}(\bR^d)}\leq 1}\left|\int_{\bR^d}f(x)(w(x))^{1/p}\phi(x) \mathrm{d}x\right|.
    $$
    Since $\|\phi\|_{L_{p'}(\bR^d)}\leq1$,
    $$
    \|\phi w^{1/p}\|_{L_{p'}(\bR^d,\bar{w})}^{p'}=\int_{\bR^d}|\phi(x)|^{p'}(w(x))^{\frac{p'}{p}}(w(x))^{-\frac{1}{p-1}}\mathrm{d}x=\|\phi\|_{L_{p'}(\bR^d)}^{p'}\leq1.
    $$
    Therefore,
    $$
    \left|\int_{\bR^d}f(x)(w(x))^{1/p}\phi(x) \mathrm{d}x\right|\leq \sup_{\|\phi\|_{L_{p'}(\bR^d,\bar{w})}\leq1}\left|\int_{\bR^d}f(x)\phi(x) \mathrm{d}x\right|.
    $$
\end{itemize}

($vii$) Let
$$
m_1(\xi):=\frac{|\xi|^{\nu}}{(1+|\xi|^2)^{\nu/2}},\quad m_2(\xi):=\frac{(1+|\xi|^2)^{\nu/2}}{1+|\xi|^{\nu}}.
$$
One can easily check that (\textit{e.g.} \cite[Example 6.2.9]{grafakos2014classical}), for any multi-index $\alpha$,
$$
|D^{\alpha}_{\xi}m_1(\xi)|+|D^{\alpha}_{\xi}m_2(\xi)|\leq N(\nu,|\alpha|)|\xi|^{-|\alpha|}.
$$
Thus, for $s\in[1,\infty)$ and multi-index $\alpha$,
$$
\sum_{i=1}^2\sup_{R>0}\left(R^{s|\alpha|-d}\int_{R<|\xi|<2R}|D^{\alpha}_{\xi}m_i(\xi)|^{s}\mathrm{d}\xi\right)^{1/s}\leq N(d,\nu,|\alpha|).
$$
By the weighted multiplier theorem (\cite[Theorem 2]{kurtz1979results}), we have the following equivalence of the two norms
$$
\|f\|_{L_p(\bR^d,w)}+\|(-\Delta)^{\nu/2}f\|_{L_p(\bR^d,w)}\simeq\|f\|_{H_p^{\nu}(\bR^d,w)}.
$$
The theorem is proved.
\end{proof}

\begin{thm}\label{21.03.01.14.59}
Let $p\in(1,\infty)$, $w\in A_p(\bR^{d+1})$, and $\nu\in\bR$. Then,
\begin{enumerate}[(i)]
    \item $\bH_p^{\nu}((0,T)\times\bR^d,w)$ is a Banach space equipped with the norm 
$$
\|f\|_{\bH_p^{\nu}((0,T)\times\bR^d,w)}:=\|f\|_{L_p((0,T)\times\bR^d,w)}+\|(-\Delta)^{\nu}f\|_{L_p((0,T)\times\bR^d,w)}.
$$
    \item The topological dual space of $L_p((0,T)\times\bR^d,w)$ is $L_{p'}((0,T)\times\bR^d,\bar{w})$, where
    $$
    \frac{1}{p}+\frac{1}{p'}=1,\quad \bar{w}:=w^{-\frac{1}{p-1}}\in A_{p'}(\bR^{d+1}).
    $$
\end{enumerate}
\end{thm}
\begin{proof}
$(i)$ 
Note that $\|f\|_{\bH_p^{\nu}((0,T)\times\bR^d,w)}$ is a norm due to some properties of $L_p$-space.
Thus it suffices to prove the completeness. Suppose that $\{f_n\}_{n=1}^{\infty}\subseteq \bH_p^{\nu}((0,T)\times\bR^d,w)$ is a Cauchy sequence.

\textbf{Case 1.} $\nu\geq0$.

By the completeness of $L_p$-spaces, there exist $f$ and $g$ in $L_p((0,T)\times\bR^d,w)$ such that
$$
\lim_{n\to\infty}(\|f_n-f\|_{L_p((0,T)\times\bR^d,w)}+\|(-\Delta)^{\nu/2}f_n-g\|_{L_p((0,T)\times\bR^d,w)})=0.
$$
For $\phi\in C_c^{\infty}(\bR^{d+1})$,
\begin{equation*}
\begin{aligned}
&\int_0^T((-\Delta)^{\nu/2}f(t,\cdot),\phi(t,\cdot))\mathrm{d}t=\int_{-\infty}^{\infty}\int_{\bR^d}f(t,x)1_{0<t<T}(-\Delta)^{\nu/2}\phi(t,x)\mathrm{d}x\mathrm{d}t\\
&=\lim_{n\to\infty}\int_{-\infty}^{\infty}\int_{\bR^d}f_n(t,x)1_{0<t<T}(-\Delta)^{\nu/2}\phi(t,x)\mathrm{d}x\mathrm{d}t\\
&=\lim_{n\to\infty}\int_{-\infty}^{\infty}\int_{\bR^d}(-\Delta)^{\nu/2}f_n(t,x)1_{0<t<T}\phi(t,x)\mathrm{d}x\mathrm{d}t=\int_0^T\int_{\bR^d}g(t,x)\phi(t,x)\mathrm{d}x\mathrm{d}t.    
\end{aligned}    
\end{equation*}
By Theorem \ref{21.02.23.17.15} $(v)$, $(-\Delta)^{\nu/2}f1_{(0,T)}$ is a bounded linear functional on $L_{p'}(\bR^{d+1},\bar{w})$. By virtue of Theorem \ref{21.02.23.17.15} $(vi)$, $(-\Delta)^{\nu/2}f1_{(0,T)}\in L_p(\bR^{d+1},w)$. 
Therefore,
$$
(-\Delta)^{\nu/2}f1_{(0,T)}= g,
$$
which certainly implies that $f_n\to f$ in $\bH_p^{\nu}((0,T)\times\bR^d,w)$.

\textbf{Case 2.} $\nu<0$.

There exists $g\in L_p((0,T)\times\bR^d,w)$ such that
$$
\lim_{n\to\infty}\|(1-\Delta)^{\nu/2}f_n-g\|_{L_p((0,T)\times\bR^d,w)}=0.
$$
Denote
$$
f:=(1-\Delta)^{-\nu/2}g\in \bH_p^{\nu}((0,T)\times\bR^d,w),
$$
then $f_n\to f$ in $\bH_p^{\nu}((0,T)\times\bR^d,w)$.

$(ii)$ Let $f$ be a bounded linear functional on $L_p((0,T)\times\bR^d,w)$. Then, it is easy to check that $fw^{-1/p}$ is a bounded linear functional on $L_p((0,T)\times\bR^d)$, thus, $fw^{-1/p}\in L_{p'
}((0,T)\times\bR^d)$. This certainly implies that $f\in L_{p'}((0,T)\times\bR^d,\bar{w})$.

\end{proof}

\begin{assumption}
\label{21.01.14.15.33}
For almost all $t\in[0,T]$, $w(t,\cdot)\in A_p(\bR^d)$ and 
$$
\esssup_{t\in[0,T]}[w(t,\cdot)]_{A_p(\bR^d)}=: N_0<\infty.
$$
\end{assumption}
\begin{lem}
			\label{21.02.24.16.49}
Let $p\in(1,\infty)$, $w\in A_p(\bR^{d+1})$, and $f \in \cS(\bR^{d+1})$. 
Suppose that Assumption \ref{21.01.14.15.33} holds. 
Additionally, assume that
\begin{equation}
								\label{2021071510}
\sup_{R>0}\left(R^{2|\alpha|-d}\int_{R<|\xi|<2R}|D^{\alpha}_{\xi}m(\xi)|^{2}\mathrm{d}\xi\right)^{1/2}\leq N^*,\quad\forall |\alpha|\leq d+1,
\end{equation}
where $m(\xi)$ is a function defined on $\bR^d$.
Then, there exists a constant $N=N(d,p,N^*,N_0)$ such that
$$
\|T_mf(t,\cdot)\|_{L_p(\bR^d,w(t,\cdot))}\leq N\|f(t,\cdot)\|_{L_p(\bR^d,(w(t,\cdot))}\quad (t- a.e.),
$$
where
$$
T_mf(t,x):=\cF^{-1}[m\cF[f(t,\cdot)]](x).
$$
\end{lem}
\begin{proof}
Consider the operator 
$$
g \mapsto T_mg(x):=\cF^{-1}[m\cF[g(\cdot)]](x).
$$
Then, by Theorem \cite[Theorem 6.2.7]{grafakos2014classical}, the operator $T_m:L_1(\bR^d)\to L_{1,\infty}(\bR^d)$ is bounded. 
Moreover, recall that $w(t,\cdot) \in A_p(\bR^d)$ $(t - a.e.)$.
Thus, applying \cite[Corollaries 6.10, 6.11, and Remark 6.14]{fackler2020weighted}, we have, for almost all $t\in[0,T]$,
\begin{align*}
\|T_m f(t,\cdot)\|_{L_p(\bR^d,w(t,\cdot))}&\leq N [w(t,\cdot)]_{A_{p}(\bR^d)}^{1\vee(p-1)^{-1}}[w(t,\cdot)^{-1/(p-1)}]_{A_{p'}(\bR^d)}^{1\vee(p'-1)^{-1}}\|f(t,\cdot)\|_{L_p(\bR^d,w(t,\cdot))}\\
&\leq N\|f(t,\cdot)\|_{L_p(\bR^d,w(t,\cdot))},
\end{align*}
where $N=N(d,p,N^*,N_0)$. The lemma is proved.
\end{proof}

\begin{thm}\label{21.03.01.14.59-2}
Let $p\in(1,\infty)$, $w\in A_p(\bR^{d+1})$, and $\nu\in\bR$. Suppose that Assumption \ref{21.01.14.15.33} holds.
\begin{enumerate}[(i)]
    \item For $\nu_0\in\bR$, $(1-\Delta)^{\nu_0/2}$ is a bijective mapping from $\bH_p^{\nu}((0,T)\times\bR^d,w)$ to $\bH_p^{\nu+\nu_0}((0,T)\times\bR^d,w)$.
    \item If $\nu_1\geq \nu_0$, then $\bH_p^{\nu_1}((0,T)\times\bR^d,w)$ is continuously embedded into $\bH_p^{\nu_0}((0,T)\times\bR^d,w)$. More precisely, there exists a positive constant $N$ such that
    $$
    \|f\|_{\bH_p^{\nu_0}((0,T)\times\bR^d,w)}\leq N\|f\|_{\bH_p^{\nu_1}((0,T)\times\bR^d,w)},\quad \forall f\in \cS(\bR^{d+1}).
    $$
    \item If $\nu$ is a nonnegative integer, then
    $$
    \|f\|_{\bH_p^{\nu}((0,T)\times\bR^d,w)}\simeq \sum_{|\alpha|\leq \nu}\|D^{\alpha}_xf\|_{L_p((0,T)\times\bR^d,w)}.
    $$
    \item The topological dual space of $\bH_p^{\nu}((0,T)\times\bR^d,w)$ is $\bH_{p'}^{-\nu}((0,T)\times\bR^d,\bar{w})$, where
    $$
    \frac{1}{p}+\frac{1}{p'}=1,\quad \bar{w}:=w^{-\frac{1}{p-1}}\in A_p(\bR^{d+1}).
    $$
\end{enumerate}
\end{thm}
\begin{proof}
$(i)$ Recall the definition of $\bH_p^{\nu}((0,T)\times\bR^d,w)$ (Definition \ref{defn solu space}(ii)). 
Then it suffices to show that for all $\nu\geq0$,
\begin{align}
						\label{20220103 01}
\|f\|_{L_p((0,T) \times \bR^d,w)}+\|(-\Delta)^{\nu/2}f\|_{L_p((0,T) \times \bR^d,w)}\simeq\|f\|_{H_p^{\nu}((0,T) \times \bR^d,w)}.
\end{align}
We prove this equivalence of the two norms above by using Lemma \ref{21.02.24.16.49}.
Let $\nu\geq0$ and put
$$
m_1(\xi):=\frac{1}{(1+|\xi|^2)^{\nu/2}},\quad m_2(\xi):=\frac{|\xi|^{\nu}}{(1+|\xi|^2)^{\nu/2}},\quad m_3(\xi):=\frac{(1+|\xi|^2)^{\nu/2}}{1+|\xi|^{\nu}}
$$
It is easy to check that $m_1$, $m_2$ and $m_3$ satisfy \eqref{2021071510}. Thus, by Lemma \ref{21.02.24.16.49},
\begin{align*}
    &\|f\|_{L_p((0,T)\times\bR^d,w)}=\|T_{m_1}(1-\Delta)^{\nu/2}f\|_{L_p((0,T)\times\bR^d,w)}\leq N\|(1-\Delta)^{\nu/2}f\|_{L_p((0,T)\times\bR^d,w)},\\
    &\|(-\Delta)^{\nu/2}f\|_{L_p((0,T)\times\bR^d,w)}=\|T_{m_2}(1-\Delta)^{\nu/2}f\|_{L_p((0,T)\times\bR^d,w)}\leq N\|(1-\Delta)^{\nu/2}f\|_{L_p((0,T)\times\bR^d,w)},\\
    &\|(1-\Delta)^{\nu/2}f\|_{L_p((0,T)\times\bR^d,w)}=\|T_{m_3}(1+(-\Delta)^{\nu/2})f\|_{L_p((0,T)\times\bR^d,w)}\leq N\|f\|_{\bH_p^{\nu}((0,T)\times\bR^d,w)}.
\end{align*}

$(ii)$ The result is an easy consequence of $(i)$
Indeed, define $m_4(\xi):=(1+|\xi|^2)^{\nu_0/2}(1+|\xi|^2)^{-\nu_1/2}$. 
Then $m_4$ satisfies \eqref{2021071510}. 
Thus by $(ii)$ and Lemma \ref{21.02.24.16.49},
\begin{align*}
\|f\|_{\bH_p^{\nu_0}((0,T)\times\bR^d,w)}&\simeq \|(1-\Delta)^{\nu_0/2}f\|_{L_p((0,T)\times\bR^d,w)}=\|T_{m_4}(1-\Delta)^{\nu_1/2}f\|_{L_p((0,T)\times\bR^d,w)}\\
&\leq N\|(1-\Delta)^{\nu_1/2}f\|_{L_p((0,T)\times\bR^d,w)}\simeq \|f\|_{\bH_p^{\nu_1}((0,T)\times\bR^d,w)}.
\end{align*}

$(iii)$ For a multi-index $\alpha$ satisfying $|\alpha|\leq\nu$, we denote
$$
m_5(\xi):=\frac{(i\xi)^{\alpha}}{(1+|\xi|^2)^{\nu/2}}.
$$
It is easy to show that \eqref{2021071510} holds with $m=m_5$. 
Thus by $(ii)$ and Lemma \ref{21.02.24.16.49},
\begin{equation*}
    \begin{gathered}
    \|D^{\alpha}_xf\|_{L_p((0,T)\times\bR^d,w)}=\|T_{m_5}(1-\Delta)^{\nu/2}f\|_{L_p((0,T)\times\bR^d,w)}\leq N\|f\|_{\bH_{p}^{\nu}((0,T)\times\bR^d,w)}.
    \end{gathered}
\end{equation*}
To prove the other direction, it is sufficient to show
\begin{equation}
\label{21.02.25.19.16}
\|(1-\Delta)^{1/2}f\|_{L_p((0,T)\times\bR^d,w)}\leq N\|f\|_{L_p((0,T)\times\bR^d,w)}+N\sum_{j=1}^{d}\|D_{x^j}f\|_{L_p((0,T)\times\bR^d,w)}
\end{equation}
due to the mathematical induction. By \eqref{20220103 01} and $(ii)$,
\begin{align*}
\|(1-\Delta)^{1/2}f\|_{L_p((0,T)\times\bR^d,w)}&\leq N\|g\|_{L_p((0,T)\times\bR^d,w)}+N\|\Delta g\|_{L_p((0,T)\times\bR^d,w)}\\
&\leq N\|f\|_{L_p((0,T)\times\bR^d,w)}+N\sum_{i=1}^d\|T_{m^i} D_{x^i}f\|_{L_p((0,T)\times\bR^d,w)},
\end{align*}
where $g:=(1-\Delta)^{-1/2}f$ and
$$
m^i(\xi):=\frac{i\xi^i}{(1+|\xi|^{2})^{1/2}},\quad i=1,2,\cdots,d.
$$
One can check that $m^i$ is a form of $m_5$. Thus \eqref{2021071510} holds with $m=m^i$.
Finally, applying Lemma \ref{21.02.24.16.49} again, we obtain \eqref{21.02.25.19.16}.

$(iv)$ Due to (ii), $(1-\Delta)^{\nu/2}(1-\Delta)^{-\nu/2}=(1-\Delta)^{-\nu/2}(1-\Delta)^{\nu/2}$ is an identity operator on $\bH_{q}^{\nu'}((0,T)\times\bR^d)$ for all $q\in(1,\infty)$ and $\nu'\in\bR$. 
Thus the statement can be proved by following the proof of Theorem \ref{21.02.23.17.15}(vi).
The theorem is proved.
\end{proof}

\begin{defn}
Let $\bH_c^{\infty}((0,T)\times\bR^d)$ ($\bH^{\infty}((0,T)\times\bR^d)$, resp.) be a set of all functions $f$ on $(0,T)\times\bR^d$ defined by
$$
f(t,x):=\sum_{i=1}^nf_i(t)g_i(x),
$$
where $f_i\in C_c^{\infty}((0,T))$ and $g_i\in C_c^{\infty}(\bR^d)$ ($g_i\in\cS(\bR^d)$, resp.) for all $i$.
\end{defn}

\begin{thm}
\label{21.03.01.15.34}
\begin{enumerate}[(i)]
    \item $\bH_c^{\infty}((0,T)\times\bR^d)$ is dense in $L_p((0,T)\times\bR^d,w)$ for all $p\in(1,\infty)$ and $w\in A_p(\bR^{d+1})$. 
    If additionally Assumption \ref{21.01.14.15.33} holds, then $\bH_c^{\infty}((0,T)\times\bR^d)$ is dense in $\bH_p^{\nu}((0,T)\times\bR^d,w)$ for all $\nu \in \bR$, $p\in(1,\infty)$, and $w\in A_p(\bR^{d+1})$.
    \item  $\bH_c^{\infty}((0,T)\times\bR^d)$ is dense in $L_q((0,T),w_1;H_p^{\nu}(\bR^d,w_2))$ for all $\nu \in \bR$, $p,q\in(1,\infty)$, $w_1\in A_q(\bR)$ and $w_2\in A_p(\bR^d)$.
\end{enumerate}

\end{thm}
\begin{proof}
Because of the similarity of the proof, we only prove $(i)$. Let $\nu \in \bR$, $p\in(1,\infty)$, and $w\in A_p(\bR^{d+1})$. 
We divide the proof depending on the value of $\nu$.
\smallskip

\textbf{Case 1.} Assume $\nu=0$. Suppose that the statement is not true. 
Then by Hahn-Banach theorem (\textit{e.g.} \cite[Theorem 5.19]{rudin2006real}) and Theorem \ref{21.03.01.14.59}$(ii)$, there is a nonzero $h\in L_{p'}((0,T)\times\bR^d,\bar{w})$ such that
$$
\int_{0}^{T}\int_{\bR^d}h(t,x)f(t,x)\mathrm{d}x\mathrm{d}t=0,\quad \forall f\in \bH_c^{\infty}((0,T)\times\bR^d).
$$
In particular,
$$
\int_{t_0}^{t_1}\int_{\bR^d}h(t,x)g(x)\mathrm{d}x\mathrm{d}t=0,\quad \forall g\in C_c^{\infty}(\bR^d),\,0< t_0<t_1<T.
$$
By Fubini's theorem, $\int_{\bR^d}h(t,x)g(x)\mathrm{d}x$ is a measurable function with respect to $t$ and moreover by a basic property of the Lebesgue integral, we have
\begin{align}
						\label{2021071540}
\int_{\bR^d}h(t,x)g(x)\mathrm{d}x=0,
\end{align}
for almost all $t\in[0,T]$. 
Finally, since \eqref{2021071540} holds for all $g\in C_c^{\infty}(\bR^d)$, we conclude that $h(t,x)=0$ for almost all $(t,x)\in (0,T)\times\bR^d$.  It is a contradiction since $h$ is not zero.

\smallskip
\textbf{Case 2.} Assume that $\nu\neq0$. Note that for any $f\in \bH_c^{\infty}((0,T)\times\bR^d)$,
$$
(1-\Delta)^{\nu/2}f\in \bH^{\infty}((0,T)\times\bR^d).
$$
Moreover, it is obvious that
$$
\bH_c^{\infty}((0,T)\times\bR^d) \subset \bH^{\infty}((0,T)\times\bR^d).
$$
Thus, by Case 1, $\bH^{\infty}((0,T)\times\bR^d)$ is dense in $L_p((0,T)\times\bR^d,w)$. 
Using Theorem \ref{21.03.01.14.59-2} $(i)$, we obtain the result. The theorem is proved.
\end{proof}

We finish the appendix by proving the fact in Remark \ref{rem 2021-12-29}.
\begin{lem}
						\label{lem 2021-12-29}
Let $-d<\alpha<d(p-1)$  and denote $w_{\alpha}(t,x):=(t^2+|x|^2)^{\alpha/2}$.
Then there exists a positive constant $N(d,\alpha)$ such that for any balls $B_R(x_0)\subset \bR^d$ and almost all $t\in[0,T]$,
\begin{equation}
\label{21.12.29.12.43}
    \begin{aligned}
    &\left(\aint_{B_R(x_0)}w_{\alpha}(t,x)\mathrm{d}x\right)\left(\aint_{B_R(x_0)}w_{\alpha}(t,x)^{-\frac{1}{p-1}}\mathrm{d}x\right)^{p-1}\\
    &=\left(\aint_{B_R(x_0)}w_{\alpha}(t,x)\mathrm{d}x\right)\left(\aint_{B_R(x_0)}w_{-\alpha/(p-1)}(t,x)\mathrm{d}x\right)^{p-1}\leq N.
\end{aligned}
\end{equation}
\end{lem}

\begin{proof}
\textbf{Case 1.} Suppose that $B_R(x_0)$ satisfies $|x_0|\geq 3R/2$. Then, 
\begin{itemize}
    \item for $x\in B_R(x_0)$, $|x_0|-R\leq|x|\leq|x_0|+R$.
    \item $|x_0|/3\leq|x_0|-R\leq|x_0|+R\leq 5|x_0|/3\leq 4(|x_0|-R)$.
\end{itemize}
These imply that for $t\in\bR$,
\begin{equation*}
    \aint_{B_R(x_0)}w_{\alpha}(t,x)\mathrm{d}x\simeq (t^2+|x_0|^2)^{\alpha/2},\quad \forall \alpha\in\bR,
\end{equation*}
where the comparability constant depends only on $d$ and $\alpha$. Therefore, \eqref{21.12.29.12.43} holds.

\textbf{Case 2.} Next, we assume that $B_R(x_0)$ satisfies $|x_0|< 3R/2$ and $t\in\bR\setminus\{0\}$. Since $B_R(x_0)\subseteq B_{3R}(0)$,
\begin{equation*}
    \aint_{B_R(x_0)}w_{\alpha}(t,x)\mathrm{d}x\leq N(d)\aint_{B_{3R}(0)}w_{\alpha}(t,x)\mathrm{d}x,
\end{equation*}
and by changing variables $x=|t|y$,
$$
\aint_{B_{3R}(0)}w_{\alpha}(t,x)\mathrm{d}x=N(d)t^{\alpha}\aint_{B_{3R/|t|}(0)}w_{\alpha}(1,y)\mathrm{d}y
$$
Due to the Lebesgue differentiation theorem, there exists $R_1>0$ such that if $0<R<R_1$, then
\begin{equation}
\label{21.12.29.12.55}
    \aint_{B_{R}(0)}|w_{\alpha}(1,y)-1|\mathrm{d}y< 1.
\end{equation}

\textbf{Case 2-1.} If $0<3R/|t|<R_1$, then by \eqref{21.12.29.12.55},
$$
\aint_{B_{3R/|t|}(0)}w_{\alpha}(1,y)\mathrm{d}y< 2,\quad \forall \alpha\in\bR.
$$
Therefore, \eqref{21.12.29.12.43} holds.

\textbf{Case 2-2.} Now, suppose that $3R/|t|\geq R_1$. Note that
$$
\aint_{B_{3R/|t|}(0)}w_{\alpha}(1,y)\mathrm{d}y=\frac{N(d)|t|^d}{(3R)^d}\int_0^{3R/|t|}(1+r^2)^{\alpha/2}r^{d-1}\mathrm{d}r.
$$
If $-d<\alpha<0$, then
\begin{equation}
\label{21.12.29.12.56}
    \begin{aligned}
    \frac{N(d)|t|^d}{(3R)^d}\int_0^{3R/|t|}(1+r^2)^{\alpha/2}r^{d-1}\mathrm{d}r
 &\leq \frac{N(d)|t|^d}{(3R)^d}\min\left( \int_0^{3R/|t|}r^{d-1}\mathrm{d}r, \int_0^{3R/|t|}r^{\alpha}r^{d-1}\mathrm{d}r \right)\\
    &=N(d,\alpha)\min(1,R^{\alpha}/|t|^{\alpha})
\end{aligned}
\end{equation}
If $\alpha\geq0$, then
\begin{equation}
\label{21.12.29.12.57}
\begin{aligned}
    \aint_{B_{3R/|t|}(0)}w_{\alpha}(1,y)\mathrm{d}y&=\frac{N(d)|t|^d}{(3R)^d}\int_0^{3R/|t|}(1+r^2)^{\alpha/2}r^{d-1}\mathrm{d}r\\
    &\leq N(d,R_1)(1+9R^2/t^2)^{\alpha/2}\leq N(d,\alpha,R_1)\max(1,R^{\alpha}/|t|^{\alpha}).
\end{aligned}
\end{equation}
Combining \eqref{21.12.29.12.56} and \eqref{21.12.29.12.57},
\begin{align*}
    &\left(\aint_{B_R(x_0)}w_{\alpha}(t,x)\mathrm{d}x\right)\left(\aint_{B_R(x_0)}w_{-\alpha/(p-1)}(t,x)\mathrm{d}x\right)^{p-1}\\
    &\leq N\left(\max(1,R^{\alpha}/|t|^{\alpha})\min(1,R^{-\alpha}/|t|^{-\alpha})1_{\alpha\geq0}+\min(1,R^{\alpha}/|t|^{\alpha})\max(1,R^{-\alpha}/|t|^{-\alpha})1_{-d<\alpha<0}\right)\\
    &\leq N(d,p,\alpha,R_1).
\end{align*}
The lemma is proved.
\end{proof}

\textbf{Declarations of interest}

Declarations of interest: none

\bibliographystyle{plain}

\begin{thebibliography}{10}

\bibitem{chang2012stochastic}
T. Chang, K. Lee, On a stochastic partial differential equation with a fractional Laplacian operator, \textit{Stoch. Process. Appl.} \textbf{122} (2012), no. 9, 3288–3311, DOI : 10.1016/j.spa.2012.04.015.

\bibitem{chen2018lp}
Z.Q. Chen, X. Zhang, $L^p$-maximal hypoelliptic regularity of nonlocal kinetic Fokker–Planck operators, \textit{J. Math. Pures Appl.} \textbf{116} (2018), 52-87, DOI : 10.1016/j.matpur.2017.10.003.

\bibitem{choi2020maximal}
J.-H. Choi, I. Kim, A maximal $L_p$-regularity theory to initial value problems with time measurable nonlocal operators generated by additive processes, \textit{Stoch. Partial. Differ. Equ. Anal. Comput.} (2023), DOI : 10.1007/s40072-023-00286-w.

\bibitem{dong2012lp}
H. Dong, D. Kim, On $L_p$-estimates for a class of non-local elliptic equations, \textit{J. Funct. Anal.} \textbf{262} (2012), no.3, 1166-1199, DOI : 10.1016/j.jfa.2011.11.002.


\bibitem{dong18Apweights}
H. Dong, D. Kim, On $L_p$-estimates for elliptic and parabolic equations with $A_p$ weights, \textit{Trans. Am. Math. Soc.} \textbf{370} (2018), no.7, 5081-5130, DOI : 10.1090/tran/7161.

\bibitem{dong2021nonlocal}
H. Dong, P. Jung, D. Kim, Boundedness of non-local operators with spatially dependent coefficients and $L_p$-estimates for non-local equations, \textit{Calc. Var. Partial Differ. Equ.} \textbf{62} (2023), 62, DOI : 10.1007/s00526-022-02392-4.

\bibitem{dong2021sobolev}
H. Dong, Y. Liu, Sobolev estimates for fractional parabolic equations with space-time non-local operators \textit{Calc. Var. Partial Differ. Equ.} \textbf{62} (2023), 96, DOI : 10.1007/s00526-023-02431-8.

\bibitem{fackler2020weighted}
S. Fackler, T.P. Hyt\"onen, N. Lindemulder.
\newblock Weighted estimates for operator-valued Fourier multipliers, \textit{Collect. Math.} \textbf{71} (2020), 511-548, DOI : 10.1007/s13348-019-00275-0.

\bibitem{gallarati2017maximal}
C. Gallarati, M. Veraar, Maximal Regularity for Non-autonomous Equations with Measurable Dependence on Time, \textit{Potential Anal} \textbf{46} (2017), no.3, 527–567, DOI : 10.1007/s11118-016-9593-7.

\bibitem{grafakos2014classical}
L. Grafakos, \textit{Classical Fourier Analysis}, 3rd ed., Graduate Texts in Mathematics Vol. 249, Springer, New York, 2014, DOI : 10.1007/978-1-4939-1194-3.

\bibitem{gyongy2021lp}
I. Gy\"ongy, S. Wu, On $L_p$-solvability of stochastic integro-differential equations, \textit{Stoch. Partial Differ. Equ. Anal. Comput.} \textbf{9} (2021), no. 2, 295–342, DOI : 10.1007/s40072-019-00160-8.

\bibitem{han2020weighted}
B.-S. Han, K.-H. Kim, D. Park, Weighted $L_q(L_p)$-estimates with Muckenhoupt weights for the diffusion-wave equations with time-fractional derivatives, \textit{J. Differ. Equ.} \textbf{269} (2020), no.4, 3515-3550, DOI : 10.1016/j.jde.2020.03.005.

\bibitem{han2021regularity}
B.-S. Han, A regularity theory for stochastic partial differential equations driven by multiplicative space-time white noise with the random fractional Laplacians, \textit{Stoch. Partial Differ. Equ. Anal. Comput.} \textbf{9} (2021), no.4, 940-983, DOI : 10.1007/s40072-021-00189-8.

\bibitem{niel2001}
N. Jacob, \textit{Pseudo Differential Operators and Markov Processes, Volume I : Fourier Analysis and Semigroups}, Imperial College Press, 2001, DOI : 10.1142/p245.

\bibitem{niel2002}
N. Jacob, \textit{Pseudo Differential Operators and Markov Processes, Volume II : Generators and Their Potential Theory}, Imperial College Press, 2002, DOI : 10.1142/p264.

\bibitem{niel2005}
N. Jacob, \textit{Pseudo differential operators and Markov processes, Volume III : Markov Processes and Applications},  Imperial College Press, 2005, DOI : 10.1142/p395.

\bibitem{kang2021lp}
J. Kang, D. Park. An $L_q(L_p)$-theory for time-fractional diffusion equations with nonlocal operators generated by L\'evy processes with low intensity of small jumps (2021), arXiv preprint arXiv:2110.01800.

\bibitem{kim2013parabolic}
I. Kim, K.-H. Kim, P. Kim, Parabolic Littlewood-Paley inequality for $\phi(-\Delta)$-type operators and applications to stochastic integro-differential equations, \textit{Adv. Math.} \textbf{249} (2013), 161–203, DOI : 10.1016/j.aim.2013.09.008. 


\bibitem{kim2015parabolic}
I. Kim, K.-H. Kim, S. Lim, Parabolic BMO estimates for pseudo-differential operators of arbitrary order, \textit{J. Math. Anal. Appl.} \textbf{427} (2015), no.2, 557-580, DOI : 10.1016/j.jmaa.2015.02.065.

\bibitem{kim2016Lp}
I. Kim, K.-H. Kim, An $L_p$-theory for a class of non-local elliptic equations related to nonsymmetric measurable kernels, \textit{J. Math. Anal. Appl.} \textbf{434} (2016), no. 2, 1302-1335, DOI : 10.1016/j.jmaa.2015.09.075.

\bibitem{kim2016lp}
I. Kim, K.-H. Kim, An $L_p$-theory for stochastic partial differential equations driven by L\'evy processes with pseudo-differential operators of arbitrary order, \textit{Stoch. Process. Appl.} \textbf{126} (2016), no.9, 2761-2786, DOI : 10.1016/j.spa.2016.03.001.

\bibitem{kim2016lplq}
I. Kim, S. Lim, K.-H. Kim, An $L_q(L_p)$-Theory for parabolic pseudo-differential equations: Calder\'on-Zygmund approach, \textit{Potential Anal.} \textbf{45} (2016), no.3, 463-483, DOI : 10.1007/s11118-016-9552-3.

\bibitem{kim2018lp}
I. Kim, An $L_p$-Lipschitz theory for parabolic equations with time measurable pseudo-differential operators, \textit{Commun. Pure Appl. Anal.} \textbf{17} (2018), no.6, 2751-2771, DOI : 10.3934/cpaa.2018130.

\bibitem{kim2019lp}
I. Kim, K.-H. Kim, P. Kim, An $L_p$-theory for diffusion equations related to stochastic processes with non-stationary independent increment, \textit{Trans. Am. Math. Soc.} \textbf{371} (2019), no.5, 3417-3450, DOI : 10.1090/tran/7410.

\bibitem{kim2012lp}
K.-H. Kim, P. Kim, An $L_p$-theory of a class of stochastic equations with the random fractional Laplacian driven by L\'evy processes, \textit{Stoch. Process. Appl.} \textbf{122} (2012), no. 12, 3921–3952, DOI : 10.1016/j.spa.2012.08.001

\bibitem{kim2021lq}
K.-H. Kim, D. Park, J. Ryu, An $L_q(L_p)$-theory for diffusion equations with space-time nonlocal operators, \textit{J. Differ. Equ.} \textbf{287} (2021), 376-427, DOI : 10.1016/j.jde.2021.04.003.

\bibitem{kim2021sobolev}
K.-H. Kim, D. Park, J. Ryu, A Sobolev space theory for the Stochastic Partial Differential Equations with space-time non-local operators (2021), arXiv preprint arXiv:2105.03013.

\bibitem{krylov2008lectures}
N.V. Krylov, \textit{Lectures on elliptic and parabolic equations in Sobolev spaces}, Graduate Studies in Mathematics Vol. 96, American Mathematical Society, Providence, 2008, DOI : 10.1090/gsm/096.

\bibitem{krylov2001calderon}
N.V.Krylov,  On the Calder\'on-Zygmund theorem with applications to parabolic equations,
\newblock {\em Algebra i Anal.} \textbf{13} (2001), no.4, 1-25 (in Russian); translation in {\em St. Petersburg Math. J.}, \textbf{13} (2002), no.4, 509-526,

\bibitem{kurtz1979results}
D.S. Kurtz, R.L. Wheeden, Results on weighted norm inequalities for multipliers, \textit{Trans. Am. Math. Soc.} \textbf{255} (1979), 343-362, DOI : 10.2307/1998180.

\bibitem{lady1988}
O. A. Lady\v{z}henskaja,  V. A. Solonnikov, N. N. Ural'ceva, \textit{Linear and quasi-linear equations of parabolic type}, Translations of Mathematical Monographs Vol. 23, American Mathematical Society, Providence, 1968, DOI : 10.1090/mmono/023.



\bibitem{emiel21}
E. Lorist, \textit{Vector-valued harmonic analysis with applications to SPDE} (Ph.D diss.) (2021), DOI : 10.4233/uuid:c3b05a34-b399-481c-838a-f123ea614f42.

\bibitem{mikulevivcius1992}
R. Mikulevi\v{c}ius, H. Pragarauskas, On the Cauchy problem for certain integro-
differential operators in Sobolev and H\"older spaces, \textit{Lith. Math. J.} \textbf{32} (1992), no.2, 238–264, 1992. 


\bibitem{mikulevivcius2014}
R. Mikulevi\v{c}ius, H. Pragarauskas, On the Cauchy problem for integro-differential
operators in Sobolev classes and the martingale problem, \textit{J. Differ. Equ.} \textbf{256} (2014), no.4, 1581–1626, DOI : 10.1016/j.jde.2013.11.008.


\bibitem{mikulevivcius2017p}
R. Mikulevi\v{c}ius, C. Phonsom, On $L^p$ theory for parabolic and elliptic integro-differential equations with scalable operators in the whole space, \textit{Stoch. Partial Differ. Equ. Anal. Comput.} \textbf{5} (2017), no.4, 472–519, DOI : 10.1007/s40072-017-0095-4.

\bibitem{mikulevivcius2019cauchy}
R. Mikulevi\v{c}ius, C. Phonsom, On the Cauchy problem for integro-differential equations in the scale of spaces of generalized smoothness, \textit{Potential Anal.} \textbf{50} (2019), no.3, 467-519. DOI : 10.1007/s11118-018-9690-x.

\bibitem{miller1982weighted}
N. Miller, Weighted sobolev spaces and pseudodifferential operators with smooth symbols, \textit{Trans. Am. Math. Soc.} \textbf{269} (1982), no.1, 91-109, DOI : 10.2307/1998595.

\bibitem{neerven2012stochastic}
J.V. Neerven, M. Veraar, L. Weis, Stochastic maximal $L^p$-regularity, \textit{Ann. Probab.} \textbf{40} (2012), no.2, 788-812, DOI : 10.1214/10-AOP626.

\bibitem{neerven2012maximal}
J.V. Neerven, M. Veraar, L. Weis, Maximal $L^p$-Regularity for Stochastic Evolution Equations, \textit{SIAM  J. Math. Anal.} \textbf{44} (2012), no.4, 1372–1414, DOI : 10.1137/110832525.

\bibitem{neerven2015maximal}
J.V. Neerven, M. Veraar, L. Weis, Maximal $\gamma$-regularity, \textit{J. Evol. Equ.} \textbf{15} (2015), 361–402, DOI : 10.1007/s00028-014-0264-0.

\bibitem{park2021maximal}
D. Park, Weighted maximal $L_q(L_p)$-regularity theory for time-fractional diffusion-wave equations with variable coefficients, \textit{J. Evol. Equ.} \textbf{23} (2023), 12,  DOI : s00028-022-00866-8.

\bibitem{portal2019stochastic}
P. Portal, M. Veraar, Stochastic maximal regularity for rough time-dependent problems, \textit{Stoch. Partial Differ. Equ. Anal. Comput.}  \textbf{7} (2019), 541–597, DOI : 10.1007/s40072-019-00134-w.

\bibitem{rudin2006real}
W. Rudin, \textit{Real and complex analysis, 3rd ed.}, McGraw-Hill, Inc., 1987.

\bibitem{stein2016harmonic}
E.M. Stein, Harmonic analysis : real-variable methods, orthogonality, and oscillatory integrals, Princeton Mathematical series Vol. 43, Princeton University Press, 2016, DOI : 10.1515/9781400883929.

\bibitem{taka1984}
K. Takashi, Pseudo-differential operators and Markov processes, \textit{J. Math. Soc. Japan}, \textbf{36} (1984), no.3, 387-418, DOI: 10.2969/jmsj/03630387.

\bibitem{zhang2013maximal}
X. Zhang, $L^p$-maximal regularity of nonlocal parabolic equations and applications, \textit{Ann. l'Inst. Henri Poincar\'e C, Anal. non lin\'eaire} \textbf{30}, no.4 (2013), 573-614, DOI : 10.1016/j.anihpc.2012.10.006.

\bibitem{zhang2013lp}
X. Zhang, $L^p$-solvability of nonlocal parabolic equations with spatial dependent and non-smooth kernels, In \textit{Emerging Topics on Differential Equations and Their Applications} pp. 247–262, Nankai Series in Pure, Applied Mathematics and Theoretical Physics Vol. 10, World Scientific, 2013, DOI : 10.1142/9789814449755\_0020.


\end{thebibliography}

\end{document}